\documentclass[reqno]{amsart}
\usepackage{geometry}

\usepackage{amsmath,amsfonts,amssymb,mathtools}

\usepackage{color}

\usepackage{amscd}
\usepackage{bbm}
\usepackage{enumerate}
\usepackage{galois}
\usepackage{mathrsfs}
\usepackage{xypic}
\usepackage{hyperref}
\usepackage{mathrsfs}
\usepackage{tikz-cd}
\usepackage{tikz}
\usepackage{pgfplots}

\usetikzlibrary{arrows.meta}

\def\pdf#1#2{\frac{\partial #1}{\partial #2}}
\def\pdff#1#2#3{\frac{\partial^2 #1}{\partial #2\partial #3}}

\def\tangentvector#1{\pdf{}{#1}}
\def\dif{\mathrm{d}}

\newcommand{\p}{\partial}

\newcommand{\al}{\alpha}
\newcommand{\av}{\alpha^\vee}

\newcommand{\lm}{\lambda}

\newcommand{\bfx}{\mathbf{x}}

\DeclareMathOperator{\nd}{d\!}

\numberwithin{equation}{section}

\newtheorem{thm}{Theorem}[section] 
\newtheorem{cor}[thm]{Corollary}
\newtheorem{lem}[thm]{Lemma}
\newtheorem{prop}[thm]{Proposition}
\theoremstyle{definition} 
\newtheorem{defn}[thm]{Definition} 

\newtheorem{ex}{Example}[section]
\theoremstyle{remark} 
\newtheorem{rem}{Remark}[section]

\begin{document}
	
	\title[Affine Weyl Groups and Generalized Frobenius Manifolds \uppercase\expandafter{\romannumeral1}]{Generalized Frobenius Manifold Structures on the Orbit Spaces of Affine Weyl Groups I}
	\author{Lingrui Jiang}
	\email{jlr24@mails.tsinghua.edu.cn}
	\address{Department of Mathematical Sciences, Tsinghua University,
		Beijing 100084, P. R.~China}
	\author{Si-Qi Liu}
	\email{liusq@tsinghua.edu.cn}
	\address{Department of Mathematical Sciences, Tsinghua University,
		Beijing 100084, P. R.~China}
	\author{Yingchao Tian}
	\email{tianyc23@mails.tsinghua.edu.cn}
	\address{Department of Mathematical Sciences, Tsinghua University,
		Beijing 100084, P. R.~China}
	\author{Youjin Zhang}
	\email{youjin@tsinghua.edu.cn}
	\address{Department of Mathematical Sciences, Tsinghua University,
		Beijing 100084, P. R.~China}
	%
	
	\begin{abstract}
		We present an approach to construct a class of generalized Frobenius manifold structures on the orbit spaces of affine Weyl groups, and prove that their monodromy groups are parabolic subgroups of the associated affine Weyl groups.
	\end{abstract}
	\maketitle
	
	\vspace*{6pt}\tableofcontents  
	
	\section{Introduction}\label{intro}
	The notion of Frobenius manifold was introduced by Dubrovin in the early 1990s in his study on the relation of integrable systems and 2-dimensional topological field theories \cite{Du-fm, 2d-tft}. This geometric structure and its generalizations play a central role in topological field theory, Gromov-Witten theory, singularity theory and in the theory of integrable systems (see \cite{Lorenzoni-1, Lorenzoni-2, Brini-1, Brini-2, Brini-3, 2d-tft, normal-form, fjrw, givental,
		hertling, gfm, manin} and the references therein). 
	
	In this paper, we construct on the orbit spaces of affine Weyl groups a class of generalized Frobenius manifold structures. These generalized structures satisfy all the axioms of Dubrovin's definition of Frobenius manifolds except for the flatness of the unit vector fields. The simplest example in this class corresponds to the affine Weyl group of type $A_1$,
	which naturally appears in the study of Hodge integrals and their remarkable connections to two important discrete integrable systems: the Volterra hierarchy and the $q$-deformed KdV hierarchy\,\textemdash\,both fundamental in soliton theory\cite{ZLQW, gfm}.	
	
	Our approach parallels the construction of a Frobenius manifold structure on the orbit space of a Coxeter group or on a certain extension of an affine Weyl group \cite{2d-tft, DZ1998, DSZZ, SYS, S, zdf1, zdf2}. For these cases, a flat pencil of metrics is obtained by projecting the Euclidean metric (on the space where the group acts) onto the orbit space, based on which the Frobenius manifold structure is constructed. 
	The key to these constructions lies in the linear dependence of both the contravariant components of the metric and its Levi-Civita connection on a distinguished coordinate of the orbit space\,\textemdash\,this special coordinate naturally gives rise to the flat unit vector field of the Frobenius manifold.
	However, for the generalized Frobenius manifolds that we are considering, the unit vector field is not flat, implying the absence of such a special coordinate. To overcome this, we proceed as follows: Given an affine Weyl group $W_a$ acting on an $\ell$-dimensional Euclidean space $V$, we introduce a parameter $\lambda$ into the invariant Fourier polynomials, yielding the coordinates $z^1,\dots, z^\ell$ on the orbit space. The projected contravariant metric 
	$\left(g_\lambda^{\al\beta}(z)\right)$ then depends on $\lambda$. Our approach is to choose the invariant $\lambda$-Fourier polynomials such that $\left(g_\lambda^{\al\beta}(z)\right)$ and the Christoffel symbols of their Levi-Civita connections depend linearly on $\lambda$. We prove in this paper that as long as this linear dependence holds, there exists a corresponding generalized Frobenius manifold structure on a dense open subset of the orbit space.
	
	In the present paper, we detail this construction, study properties of the resulting generalized Frobenius manifolds including their monodromy groups, and provide some explicit examples. Further examples will appear in our subsequent paper \cite{JLTZ-2}. The paper is organized as follows. In Sect.\,\ref{basic settings}, we introduce the invariant $\lambda$-Fourier polynomial ring and study its properties. In Sect.\,\ref{GFM structure}, we construct generalized Frobenius manifold structures on orbit spaces of affine Weyl groups, and summarize the main results in Theorem \ref{main}. In Sect.\,\ref{Polynomial Property and Monodromy Groups}, we analyze properties of the monodromy groups of these generalized Frobenius manifolds. In Sect.\,\ref{examples}, we present some representative examples. In Sect.\,\ref{conclusion}, we conclude the paper with remarks and further directions.
	
	\section{Invariant $\lambda$-Fourier polynomial rings and proper generators}
	\label{basic settings}
	Let us begin with several basic settings.
	\subsection{Affine coordinates}
	
	Let $R$ be an irreducible reduced root system in an $\ell$-dimensional Euclidean space $V$ with inner product $(\cdot\, ,\cdot )$, and $\alpha_1,\dots,\alpha_\ell$ be a fixed basis of simple roots. Denote by $\alpha_1^\vee,\dots, \alpha_\ell^\vee$ the corresponding coroots given by
	\begin{align*}
		\alpha_j^\vee =\frac{2\alpha_j}{(\alpha_j,\alpha_j)},\quad j=1,\dots,\ell,
	\end{align*}
	then
	the fundamental weights $\omega_1,\dots,\omega_\ell \in V$ of $R$ are defined by the relations
	\begin{align*}
		(\omega_i, \alpha_j^\vee) = \delta_{ij},\quad i, j=1,\dots,\ell.
	\end{align*}
	We fix a weight
	\begin{align}
		\label{omega = sum mr omegar}
		\omega = \sum_{j=1}^{\ell} m_j \omega_j,\quad  m_j \in \mathbb Z_{\ge 0}
	\end{align}
	of $R$, and introduce affine coordinates $(x^1,\dots,x^\ell; c)$ on $V$ by
	\begin{align}
		\label{x=c omega+}
		\bfx = c\omega + x^1\alpha_1^\vee + \dots + x^\ell\alpha_\ell^\vee,
	\end{align}
	where $c\in\mathbb{R}$ is a fixed parameter. 
	We will also use the following expressions of vectors in $V$:
	\begin{align}
		\bfx =c\omega +\hat{\bfx}= (x^1 +\theta_1 c)\alpha_1^\vee + \dots + (x^\ell + \theta_\ell c)\alpha_\ell^\vee,
	\end{align}
	with $\hat {\bfx}=x^1\alpha_1^\vee + \dots + x^\ell\alpha_\ell^\vee$, and
	\begin{align}
		\label{weight}
		\theta_j = (\omega_j, \omega),~ j=1,\dots,\ell.
	\end{align}
	
	Denote by $W_a(R)$ the affine Weyl group of $R$. It acts on $V$ by
	affine transformations
	\begin{align}
		\phi\colon\bfx \mapsto \phi_0(\bfx) + \sum_{j=1}^{\ell} k_j\alpha_j^\vee,\quad \phi_0 \in W(R),\  k_1,\dots,k_\ell \in \mathbb Z,
	\end{align}
	where $W = W(R)$ is the Weyl group of $R$ generated by the reflections 
	\begin{align*}
		\sigma_j(\bfx) = \bfx - (\alpha_j^\vee, \bfx)\alpha_j,\quad \bfx \in V,\ j=1,\dots,\ell.
	\end{align*}

	\subsection{The invariant $\lambda$-Fourier polynomial ring and basic generators}
	In this subsection, we are to introduce the  $\lambda$-Fourier polynomial ring $\mathscr A$ and its invariant subring $\mathscr A^W$ w.r.t. the action of the affine Weyl group $W_a(R)$, and define the set of basic generators $\{y^1,\dots, y^\ell\}$. To this end, we first introduce a bigger ring $\mathscr F$, called the Fourier function ring, 
	as follows:
	\begin{align*}
		\mathscr{F} = \mathrm{span}_{\mathbb C}\left\{f(x^1,\dots,x^\ell;c)=\mathrm{e}^{2\pi i(t_0c+t_1x^1+\dots+t_\ell x^\ell)}\mid t_0, t_1,\dots,t_\ell\in \mathbb C\right\}.
	\end{align*}
	For any $f \in \mathscr F$, we can uniquely write it in the form $f(\bfx;c)$. Then we have a $W_a(R)$-action on $\mathscr F$ given by
	\begin{align}
		(\phi^* f)(\bfx; c) = f(\phi(\bfx); c),\quad \phi \in W_a(R).
	\end{align}
	Note that if we write $f$ as a $f(\hat{\bfx}; c)$, then the action is given by
	\begin{align}
		(\phi^* f)(\hat{\bfx}; c) = (\phi^* f)(\bfx-c\omega; c) = f(\phi(\bfx)-c\omega; c),\quad \phi \in W_a(R),
	\end{align}
	which in general does not coincide with $f(\phi(\hat{\bfx});c)$. We introduce a $\mathbb C$-gradation on $\mathscr F$ by the assignment
	\begin{align}
		\deg \mathrm{e}^{2\pi ix^j} = \theta_j,\quad  \deg \mathrm{e}^{2\pi i c} = -1,
	\end{align}
	where $\theta_j$ is defined in \eqref{weight}. Then for any given linear function 
	\begin{equation}\label{zh-11-18-1}
		\psi(\bfx;c) = t_1(x^1+\theta_1c)+\dots+t_\ell (x^\ell+\theta_\ell c) + t_0c,\quad t_1,\dots,t_\ell \in \mathbb C
	\end{equation} and constant $A_0 \in \mathbb C^\times$, we have
	\begin{align}
		\label{deg of mono}
		\deg \left(A_0~ \mathrm{e}^{2\pi i\psi(\bfx;c)}\right)=\psi(\mathbf{0};-1).
	\end{align}
	Here $\mathbf{0}$ is the zero vector in $V$.
	\begin{lem}
		\label{maintain degrees}
		The affine Weyl group $W_a(R)$ acts quasi-homogeneously on $\mathscr F$, i.e., it preserves the degrees of monomials.
	\end{lem}
	\begin{proof}
		Let $f(\bfx;c) \in \mathscr F$ be a nonzero monomial, which can be represented in the form 
		\[f(\bfx;c)=A_0 \mathrm{e}^{2\pi i\psi(\bfx;c)}\] 
		with a linear function $\psi$ of the form \eqref{zh-11-18-1} and a nonzero constant $A_0$. On one hand, for any given translation
		\[
		\phi\colon \bfx \mapsto \bfx + \alpha,\quad \alpha=\sum_{r=1}^\ell n_r\alpha_r^\vee
		\]
		of $W_a(R)$, its action on $f$ is given by
		\[
		(\phi^* f)(\bfx; c)= A_0 \mathrm{e}^{2\pi i\psi(\bfx+\alpha;c)}= \mathrm{e}^{2\pi i\psi(\alpha;0)}f(\bfx; c).
		\]
		Therefore, $\phi^* f$ and $f$ merely differ by a nonzero factor, hence the action by $\phi$ preserves the degree.
		
		On the other hand, for any $\phi \in W(R)$ is given by
		\[
		\deg(\phi^* f) = \deg \left(A_0~ \mathrm{e}^{2\pi i\psi(\phi(\bfx); c)}\right) = \psi(\phi(\mathbf{0}); c).
		\]
		Since $\phi$ is linear, we have $\phi(\mathbf{0}) = \mathbf{0}$, hence $\deg(\phi^* f) = \deg f$. The lemma is proved.
	\end{proof}
	
	Denote by $\mathscr F^W$ the invariant subring of $\mathscr F$ w.r.t. the action of $W_a(R)$. We introduce the invariant functions $Y_1,\dots,Y_\ell\in \mathscr F^W$ by 
	\begin{align}\label{zzh-1}
		Y_j(x^1,\dots,x^\ell;c) = \frac 1{N_j}\sum_{\sigma\in W(R)} \mathrm{e}^{2\pi i(\omega_j, \sigma(\bfx))},\quad j=1,\dots,\ell,
	\end{align}
	where $N_j = \#\{\sigma\in W(R)\colon\sigma(\omega_j) = \omega_j\}$ and $\bfx = c\omega + x^1\alpha_1^\vee + \dots + x^\ell\alpha_\ell^\vee$.

	\begin{prop}
		\label{quasi-homogeneous Yj}
		The functions $Y_j$ belong to $\mathscr F^W$ and are quasi-homogeneous of degree $0$.
	\end{prop}
	\begin{proof}
		The invariance of $Y_j$ under the action of $W_a(R)$ follows from their definition. For any $\sigma \in W_a(R)$, we know from Lemma \ref{maintain degrees} that
		\begin{align*}
			\deg \mathrm{e}^{2\pi i(\omega_j, \sigma(\bfx))} = \deg \mathrm{e}^{2\pi i(\omega_j, \bfx)}= \deg \mathrm{e}^{2\pi i(\theta_j c + x^j)} = -\theta_j + \theta_j = 0,
		\end{align*}
		thus we arrive at $\deg Y_j=0$. The proposition is proved.
	\end{proof}

	Let us proceed to introduce the $\lambda$-Fourier polynomial ring $\mathscr A$ and its invariant subring $\mathscr A^W$ w.r.t. the action of $W_a(R)$. Denote
	\begin{align}
		\kappa = \gcd\{(\omega, \alpha_r)~|~r = 1,\dots,\ell\},
	\end{align}
	where $\gcd$ refers to the \textit{greatest common divisor}. This makes sense since
	\begin{align*}
		(\omega, \alpha_r) = m_r(\omega_r, \alpha_r) = m_r\frac{(\alpha_r,\alpha_r)}{2},\quad r=1,\dots, \ell
	\end{align*}
	are nonnegative rational numbers. Let
	\begin{align}\label{zzh-2}
		\lambda = \mathrm{e}^{-2\pi i\kappa c},
	\end{align}
	then we have
	\begin{align}
		\deg \lambda = \kappa.
	\end{align}
	Define the $\lambda$-Fourier polynomial ring $\mathscr A$ as a subring of $\mathscr F$ by
	\begin{align*}
		\mathscr A = \mathbb C[\lambda] \otimes \mathbb C\bigl[\mathrm{e}^{2\pi i x^j}, \mathrm{e}^{-2\pi i x^j}\mid j = 1,\dots,\ell\bigr] \subset\mathscr F.
	\end{align*}
	Explicitly, $\mathscr A$ is composed of elements of the form
	\begin{align*}
		\mathrm{e}^{2\pi i\left(-a_0\kappa c+a_1x^1+\dots+a_\ell x^\ell\right)},\quad a_0 \in \mathbb Z_{\ge 0}, a_1,\dots,a_\ell \in \mathbb Z
	\end{align*}
	and their linear combinations.
	We call the subring of  $\mathscr A$ defined by
	\[\mathscr A^W = \mathscr A \cap \mathscr F^W\] 
	the invariant $\lambda$-Fourier polynomial ring. Then we have 
	$\lambda = \mathrm{e}^{-2\pi i\kappa c}\in \mathscr A^W$. Note that $W_a(R)$ does not act on $\mathscr A$, since results of the action may involve negative power terms of $\lambda$, which are not in $\mathscr A$.
	
	Let
	\begin{align}
		\label{yj definition}
		y^j(\bfx;c) = \mathrm{e}^{-2\pi i\theta_j c}~Y_j(\bfx),~j = 1,\dots,\ell.
	\end{align}
	It follows from Proposition \ref{quasi-homogeneous Yj} that $y^j$ is quasi-homogeneous of degree $\theta_j$.
	
	\begin{lem}
		We have $y^j \in \mathscr A^W$, $j = 1,\dots,\ell$.
	\end{lem}
	\begin{proof}			
		It is obvious that $y^j$ is $W_a(R)$-invariant, so we only need to check that $y^j \in \mathscr A$. It suffices to show that for any $\sigma \in W(R)$, the monomial
		\begin{align}
			\label{mono in A}
			\mathrm{e}^{-2\pi i \theta_j c}\mathrm{e}^{2\pi i(\omega_j, \sigma(\bfx))}\in \mathscr A.
		\end{align}
		Observe that
		\begin{align}
			-2\pi i\theta_j c + 2\pi i(\omega_j, \sigma(\bfx)) ~=&~ -2\pi ic(\omega_j, \omega) + 2\pi i (\sigma^{-1}(\omega_j), \bfx)\notag\\
			=&~ -2\pi ic(\omega_j - \sigma^{-1}(\omega_j), \omega) + 2\pi i\sum_{r=1}^\ell x^r(\sigma^{-1}(\omega_j), \alpha_r^\vee).
		\end{align}
		Since $\sigma^{-1} \in W(R)$, there exist non-negative integers $n_1,\dots,n_\ell$ such that \cite{bourbaki}
		\begin{align}
			\label{sigma(weight)}
			\sigma^{-1}(\omega_j) = \omega_j - \sum_{s=1}^{\ell}n_s\alpha_s,
		\end{align}
		from which it follows that
		\begin{align}
			(\omega_j - \sigma^{-1}(\omega_j), \omega) =\sum_{s=1}^\ell n_s\left(\alpha_s,\omega\right)=m \kappa,
		\end{align}
		where $m$ is a non-negative integer. We also have
		\begin{align}
			\label{too long to find a name}
			\left(\sigma^{-1}(\omega_j), \alpha_r^\vee\right) = (\omega_j, \alpha_r^\vee) - \sum_{s=1}^\ell n_s(\alpha_s, \alpha_r^\vee) = \delta_{jr} - \sum_{s=1}^{\ell}n_s c_{sr},
		\end{align}
		where $c_{sr}$ are the $(s,r)$-entries of the Cartan matrix of the root system $R$. Therefore, there exists $a_0 \in \mathbb Z_{\ge 0}$ and $a_1,\dots a_\ell \in \mathbb Z$ such that
		\begin{align*}
			\mathrm{e}^{-2\pi i \theta_j c}\mathrm{e}^{2\pi i(\omega_j, \sigma(\bfx))}= \mathrm{e}^{2\pi i (-a_0\kappa c+a_1x^1+\dots+a_\ell x^\ell)}\in \mathscr A.
		\end{align*}
		The lemma is proved.	
	\end{proof}
	
	\begin{thm}
		\label{poly generators}
		We have $\mathscr A^W = \mathbb C[y^1,\dots,y^\ell; \lambda]$.
	\end{thm}
	A proof of this theorem will be given in the next subsection. 
	We call $\{y^1,\dots,y^\ell\}$ the set of basic generators of $\mathscr A^W$.
	
	\begin{ex}[$(R,\omega) = (A_2,\omega_2)$]
		\label{ex-A_2}
		Let $e_1, e_2, e_3$ be an the orthonormal basis of the Euclidean space $\mathbb{R}^3$, and $V$ be the hyperplane spanned by
		\[\al_1=e_1-e_2,\quad \al_2=e_2-e_3.\]
		Then we can take $\al_1, \al_2$ as a basis of simple roots for the root system $R$ of type $A_2$ in $V$.
		We have
		\[\av_1=\al_1,\quad \av_2=\al_2,\quad
		\omega_1=\frac23\av_1+\frac13\av_2,\quad \omega_2=\frac13\av_1+\frac23\av_2.\]
		Let us take $\omega=\omega_2$, then the affine coordinates $(x^1, x^2; c)$ on $V$ are given by
		\[\mathbf{x}=c\omega_2+x^1 \av_1+x^2 \av_2,\quad \mathbf{x}\in V.\]
		From the definition \eqref{zzh-1} of $Y_j$, we have
		\begin{align*}
			Y_1&=\mathrm{e}^{2\pi i \left(x^1+\frac13 c\right)}+\mathrm{e}^{2\pi i \left(x^2-x^1+\frac13 c\right)}+ \mathrm{e}^{-2\pi i \left(x^2+\frac23 c\right)},\\
			Y_2&=\mathrm{e}^{2\pi i \left(x^2+\frac23c\right)}+ \mathrm{e}^{-2\pi i \left(x^1+\frac13c\right)}+\mathrm{e}^{2\pi i \left(x^1-x^2-\frac13c\right)}.
		\end{align*}
		Since $\kappa = 1$, we have $\lambda = \mathrm{e}^{-2\pi ic}$, and
		the basic generators are given by
		\begin{align*}
			y^1&=e^{-\frac23\pi ic} Y_1=\mathrm{e}^{2\pi i x^1}+\mathrm{e}^{2\pi i (x^2-x^1)}+\lambda \mathrm{e}^{-2\pi i x^2},\\
			y^2&=e^{-\frac43\pi ic} Y_2=\mathrm{e}^{2\pi i x^2}+\lambda \mathrm{e}^{-2\pi i x^1}+\lambda \mathrm{e}^{2\pi i (x^1-x^2)}.
		\end{align*}
	\end{ex}
	
	\subsection{The $\varphi_\lambda$-transformation and the null-transformation}
	The parameter $\lambda$ of the $\lambda$-Fourier polynomials $y^1,\dots,y^\ell$ is defined by \eqref{zzh-2}, so it lies on the unit circle. However, since $y^j$ depend polynomially on $\lambda$, the definition of $y^j$ can be extended naturally to any $\lambda \in \mathbb C$, including $\lambda = 0$ in particular.
	\begin{defn}
		\label{varphi lambda}
		For any fixed $\lambda \in\mathbb C$, the basic generators $y^1,\dots,y^\ell$ yield a transformation 
		\[\varphi_\lambda\colon (x^1,\dots,x^\ell)\mapsto (y^1,\dots,y^\ell).\] 
		We call it the $\varphi_\lambda$-transformation of $(y^1,\dots,y^\ell)$. In particular, the $\varphi_0$ transformation is called the null-transformation.
	\end{defn}
	
	\begin{ex}[$(R,\omega) = (A_2,\omega_2)$]
		From the Example \ref{ex-A_2}, we know that the $\varphi_1$-transformation of $(y^1, y^2)$ is given by
		\begin{align*}
			y^1\big|_{\lambda=1}&=\mathrm{e}^{2\pi i x^1}+\mathrm{e}^{2\pi i (x^2-x^1)}+ \mathrm{e}^{-2\pi i x^2},\\
			y^2\big|_{\lambda=1}&=\mathrm{e}^{2\pi i x_2}+ \mathrm{e}^{-2\pi i x^1}+ \mathrm{e}^{2\pi i (x^1-x^2)}.
		\end{align*}
		The corresponding null-transformation is given by
		\begin{align*}
			y^1\big|_{\lambda=0}&=\mathrm{e}^{2\pi i x^1}+\mathrm{e}^{2\pi i (x^2-x^1)},\\
			y^2\big|_{\lambda=0}&=\mathrm{e}^{2\pi i x^2}.
		\end{align*}
		Note that $y^2$ has only one term $\mathrm{e}^{2\pi i x^2}$, this result will be generalized in Theorem \ref{only leading term}.
	\end{ex}
	
	\begin{prop}\label{zh-2026-7-6-1}
		The Jacobian determinant of the null-transformation is given by
		\begin{align}
			\label{determinant of null trans}
			\mathbf{J}_0(\hat{\bfx}) = C\cdot \mathrm{e}^{\pi i\sum_{\alpha \in R^+}(\alpha, \hat{\bfx})}\prod_{\beta\in R^+,\,(\beta,\omega)=0}\left(1-\mathrm{e}^{-2\pi i(\beta, \hat{\bfx})}\right),
		\end{align}
		where $R^+$ is the set of all positive roots of $R$ and $C$ is a nonzero constant.
	\end{prop}
	\begin{proof}
		From  \cite{bourbaki} and  \cite{DZ1998} we know that the Jacobian determinant of $\varphi_\lambda$ can be represented in the form
		\begin{align}
			\mathbf{J}(\hat{\bfx};\lambda) ~=&~ C\cdot\mathrm{e}^{-2\pi i c\sum_j(\omega_j,\omega) + \pi i\sum_{\alpha \in R^+}\left(c(\alpha, \omega)+(\alpha, \hat{\bfx})\right)}\prod_{\beta\in R^+}\left(1-\mathrm{e}^{-2\pi i\left(c(\beta, \omega)+(\beta, \hat{\bfx})\right)}\right)\notag\\
			\label{determinant of phi lambda}
			=&~C\cdot\mathrm{e}^{ \pi i\sum_{\alpha \in R^+}(\alpha, \hat{\bfx})}\prod_{\beta\in R^+}\left(1-\mathrm{e}^{-2\pi i\left(c(\beta, \omega)+(\beta, \hat{\bfx})\right)}\right).
		\end{align}
		Here we use the fact that \cite{bourbaki}
		\begin{align}
			\frac 12\sum_{\alpha\in R^+}\alpha = \sum_{j=1}^\ell\omega_j.
		\end{align}
		Since $\mathbf{J}(\hat{\bfx};\lambda)$ is a polynomial of $\lambda$, we have $\mathbf J_0(\hat{\bfx}) = \mathbf J(\hat{\bfx}; 0)$, so the constant term of $\mathbf{J}(\hat{\bfx};\lambda)$ w.r.t. $\lambda$ yields \eqref{determinant of null trans}. The proposition is proved.
	\end{proof}
	
	\begin{cor}
		\label{functional independence}
		The Jacobian determinant of the null-transformation is non-vanishing except on an $(\ell-1)$-dimensional divisor. In particular, $y^1\big|_{\lambda = 0}, \dots,y^\ell\big|_{\lambda = 0}$ are functionally independent.
	\end{cor}
	
	We now turn to the proof of Theorem \ref{poly generators}.
	
	\begin{proof}[Proof of Theorem \ref{poly generators}]
		For any given $f \in \mathscr A^W$, we need to show that it can be represented in terms of a polynomial of $y^1,\dots, y^\ell$ and $\lambda$. By using Lemma \ref{maintain degrees}, we know that $f$ can be decomposed as the sum of quasi-homogeneous elements of $\mathscr A^W$, 	so we may assume without lose of generality that $f$ is quasi-homogeneous of degree $n$.
		
		From \cite{bourbaki} we know that the Fourier polynomial $f|_{\lambda=1}$ of $x^1, \dots, x^\ell$ can be represented in terms of a polynomial of $y^1\big|_{\lambda=1}$,\dots,$y^\ell\big|_{\lambda=1}$ as follows:
		\begin{align}
			f|_{\lambda = 1} = \sum_{\beta_1,\dots,\beta_\ell  \ge 0} a_{\beta_1,\dots,\beta_\ell } (y^1\big|_{\lambda = 1})^{\beta_1}\dots (y^\ell\big|_{\lambda = 1})^{\beta_\ell},
		\end{align}
		so it follows from the quasi-homogeneity of $f$ and of $y^1,\dots, y^\ell$ that
		\begin{align}
			\label{raw f}
			f = \sum_{\beta_1,\dots,\beta_\ell  \ge 0} a_{\beta_1,\dots,\beta_\ell}\lambda^{n-\theta_1\beta_1 - \dots - \theta_\ell\beta_\ell} (y^1)^{\beta_1}\cdots (y^\ell)^{\beta_\ell},
		\end{align}
		where each $n-\theta_1\beta_1 - \dots - \theta_\ell\beta_\ell\in\mathbb{Z}$. Hence we can represent $f$ in the form
		\begin{align}
			\label{medium f new}
			f = \sum_{j\in Z} \lambda^j P_j(y^1,\dots,y^\ell),
		\end{align}
		where $P_j$ are polynomials. 
		Let $q$ be the smallest integer such that $P_q=0$. If $q<0$, then by viewing $f$ and $y^1,\dots,y^\ell$ as functions of $x^1,\dots,x^\ell$ and $\lambda$, and by comparing the $\lambda^q$ term in \eqref{medium f new} we arrive at
		\begin{align}
			P_q(y^1\big|_{\lambda = 0}, \dots,y^\ell\big|_{\lambda = 0}) =0.
		\end{align}
		This leads to a contradiction with the fact, as shown by Corollary \ref{functional independence}, that $y^1\big|_{\lambda = 0}, \dots,y^\ell \big|_{\lambda = 0}$ are functionally independent. Thus $q$ must be a non-negative integer, and the theorem is proved.
	\end{proof}

	\subsection{Leading terms of the basic generators}
	From the definition \eqref{yj definition} of $y^j$, we know that it contains a term
	\begin{align}
		\label{leading term}
		\mathrm{e}^{-2\pi i\theta_j c}\mathrm{e}^{2\pi i(\omega_j, {\bfx})} = \mathrm{e}^{2\pi ix^j},
	\end{align}
	which we call the leading term of $y^j$. Since not involving $\lambda$, this term also appears in $y^j|_{\lambda = 0}$. The following theorem shows that for some indices $j$, $y^j|_{\lambda = 0}$ consist merely of their leading terms.
	
	\begin{thm}
		\label{only leading term}
		Denote
		\begin{equation}\label{zzh-4}
			S=\{r\in\{1,\dots,\ell\}\colon m_r=(\omega, \alpha_r^\vee) > 0\},
		\end{equation}
		then we have
		\begin{align}
			y^j\big|_{\lambda = 0}= \mathrm{e}^{2\pi ix^j},\quad j\in S.
		\end{align}
	\end{thm}
	\begin{proof}
		For any given $j \in S$, from \eqref{zzh-1} and \eqref{yj definition} we obtain
		\begin{align}
			\label{phi 0}
			y^j\big|_{\lambda=0} =\frac 1{N_j}\sum_{\substack{\sigma \in W(R)\\(\omega_j- \sigma^{-1}(\omega_j) ,\omega) = 0}} \mathrm{e}^{2\pi i(\sigma^{-1}(\omega_j), \hat{\bfx})}.
		\end{align}
		Assume that $\sigma \in W$ satisfies the condition
		\begin{align}
			\label{omega,omega=0}
			(\omega_j - \sigma^{-1}(\omega_j), \omega) = 0,
		\end{align}
		then from \eqref{sigma(weight)} it follows that
		\[
		\sum_{r=1}^\ell n_rm_r(\alpha_r, \omega_r) = 0.
		\]
		By using the fact that all terms on the left hand side of the above equality are nonnegative and $m_j>0$, we obtain $n_j=0$, and
		\begin{align}
			\label{inner prod invariant}
			(\sigma^{-1}(\omega_j), \omega_j) = (\omega_j, \omega_j) - \sum_{r=1}^\ell n_r(\alpha_r, \omega_j) = (\omega_j, \omega_j).
		\end{align}
		Since $\sigma^{-1}$ is an orthogonal transformation, the equation \eqref{inner prod invariant} implies that $\sigma^{-1}(\omega_j) = \omega_j$. Therefore, from \eqref{phi 0} it follows that
		\begin{align}
			y^j\big|_{\lambda=0} =\frac 1{N_j}\sum_{\substack{\sigma \in W(R)\\\omega_j = \sigma^{-1}(\omega_j)}} \mathrm{e}^{2\pi i(\omega_j, \hat{\bfx})} =\frac 1{N_j}\sum_{\substack{\sigma \in W(R)\\\omega_j = \sigma^{-1}(\omega_j)}} \mathrm{e}^{2\pi ix^j} = \mathrm{e}^{2\pi ix^j}.
		\end{align}
		The theorem is proved.
	\end{proof}
	
	\begin{cor}
		The following identity holds true:
		\begin{align}
			\label{sum mr logyr = 2pi(omega,x)}
			\sum_{r=1}^\ell m_r\log\left( y^r\big|_{\lambda=0}\right) = 2\pi i(\omega, \hat{\bfx}).
		\end{align}
	\end{cor}
	
	\subsection{Proper factors and proper generators}
	
	In order to construct generalized Frobenius manifold structures on the orbit space of $W_a(R)$, we introduce in this subsection the notion proper generators and proper factors of $\mathscr A^W$.
	
	Consider the subalgebra 
	$\mathscr R = \mathbb C[y^1,\dots,y^\ell]$
	of $\mathscr A^W$. It has the following two properties:
	\begin{itemize}
		\item $\mathscr R$ is a graded subring of $\mathscr A$. This means that for any $f \in \mathscr R$, if we decompose it into quasi-homogeneous components as $f_1+\cdots+f_n$, then each $f_k \in \mathscr R$. This property follows from the quasi-homogeneity of $y^1,\dots,y^\ell$.
		\item The tensor product decomposition $\mathscr A^W = \mathscr R \otimes\mathbb C[\lambda]$ holds. This means that the natural homomorphism \[h\colon\mathscr R \otimes\mathbb C[\lambda] \to \mathscr A^W,\quad f\otimes g \mapsto fg, \quad f \in \mathscr R, \, g \in \mathbb C[\lambda]\]
		is an isomorphism.
	\end{itemize}
	These properties give rise to the following definition.
	\begin{defn}
		A subalgebra $\mathscr P$ of $\mathscr A^W$ is called a proper factor, if it is graded and $\mathscr A^W = \mathscr P\otimes\mathbb C[\lambda]$.
	\end{defn}
	In what follows, we are to find a set of generators for each proper factor $\mathscr P$, called proper generators.
	
	\begin{defn}\label{zzh-6}
		Suppose $z^1,\dots,z^\ell \in \mathscr A^W$. We say that $\{z^1,\dots,z^\ell\}$ is a set of proper generators of $\mathscr A^W$ if $\deg z^j = \theta_j$ and
		\begin{align}
			\label{proper gen defn}
			z^j\big|_{\lambda=0}=y^j\big|_{\lambda=0},\quad j=1,\dots,\ell.
		\end{align}
	\end{defn}
	
	\begin{lem}
		\label{sj is poly of yj and lambda}
		Let $z^1,\dots,z^\ell \in \mathscr A^W$, then these functions form a set of  proper generators of $\mathscr A^W$ if and only if for any $1\le j\le \ell$ there exists a polynomial $s^j$ of $\lambda$ and $\{y^r\big|~1 \le r \le \ell,~ \theta_r < \theta_j\}$ such that either $s^j=0$ or $\deg s^j = \theta_j - \kappa$, and 
		\begin{align}
			\label{proper gen cond}
			z^j = y^j + \lambda s^j.
		\end{align}
		\begin{proof}
			If $z^1,\dots,z^\ell$ satisfy the relation \eqref{proper gen cond}, then it is obvious to see that they form a set of proper generators of $\mathscr A^W$. 
			
			Now we suppose that $\{z^1,\dots,z^\ell\}$ is a set of proper generators of $\mathscr A^W$, then for any $1\le j\le \ell$ we have $(z^j - y^j) \in \mathscr A^W$ and
			\begin{align*}
				\left(z^j-y^j\right)\big|_{\lambda=0}=0.
			\end{align*}
			Hence there exists $s^j \in \mathscr A^W$ such that $z^j-y^j = \lambda s^j$, and either $s^j=0$ or $\deg s^j = \theta_j - \kappa$. By using Theorem \ref{poly generators} and by degree counting, we know that $s^j$ is a polynomial of $\lambda$ and $y^r$'s whose degrees are less than $\theta_j$. The lemma is proved.
		\end{proof}
	\end{lem}
	\begin{prop}
		\label{zi gen}
		Let $\{z^1,\dots,z^\ell\}$ be a set of proper generators of $\mathscr A^W$, then
		\begin{align}
			\mathscr A^W=\mathbb C[z^1,\dots,z^\ell;\lambda].
		\end{align}
		\begin{proof}
			It suffices to show that $y^j \in \mathbb C[z^1,\dots,z^\ell;\lambda]$ for $j=1,\dots,\ell$. Assume that there exist some $y^j$ which do not belong to $\mathbb C[z^1,\dots,z^\ell;\lambda]$, and let $y^k$ be one of those with the lowest degree. It follows from \eqref{proper gen cond} that $s^k \not \in \mathbb C[z^1,\dots,z^\ell;\lambda]$. However, since $\theta_k$ is the lowest degree among the ones of $y^j$ that do not belong to $C[z^1,\dots,z^\ell;\lambda]$, we know that $y^r \in \mathbb C[z^1,\dots,z^\ell;\lambda]$ for each $r$ satisfying $\theta_r < \theta_k$. Then we know from Lemma \ref{sj is poly of yj and lambda} that $s^k \in \mathbb C[z^1,\dots,z^\ell;\lambda]$, which leads to a contradiction. The proposition is proved.
		\end{proof}
	\end{prop}
	\begin{cor}
		\label{equal det}
		For any set $\{z^1,\dots,z^\ell\}$ of proper generators of $\mathscr A^W$, we have
		\begin{align}
			\det\left(\pdf{z^j}{x^p}\right)=\det\left(\pdf{y^j}{x^p}\right)=\mathbf{J}(\hat{\bfx};\lambda),\label{zzh-3}
		\end{align}
		where $\mathbf{J}(\hat{\bfx};\lambda)$ is defined by \eqref{determinant of phi lambda}. Thus
		$z^1,\dots,z^\ell$ are functionally independent for each $\lambda \in \mathbb C$.
	\end{cor}
	\begin{proof}
		For each $1\le j\le \ell$ we have $z^j = y^j + \lambda s^j$, where $s^j$ is a certain polynomial of $\lambda$ and $y^r$ that satisfy the condition $\theta_r<\theta_j$. Thus from the relations
		\begin{align}
			\pdf{z^j}{x^p}=\pdf{y^j}{x^p} + \lambda\sum_{\theta_r < \theta_j}\pdf{s^j}{y^r}\pdf{y^r}{x^p}
		\end{align}
		we arrive at \eqref{zzh-3}. The corollary is proved.
	\end{proof}
	
	Suppose $\{z^1,\dots,z^\ell\}$ is a set of proper generators, then $\mathscr P = \mathbb C[z^1,\dots,z^\ell]$ is a proper factor. In fact, it follows from the corollary above that $z^1,\dots,z^\ell$ and $\lambda$ are algebraically independent (as functions of $x^1,\dots,x^\ell$ and $\lambda$), hence $\mathscr A^W = \mathbb C[z^1,\dots,z^\ell]\otimes\mathbb C[\lambda]$. Since $z^1,\dots,z^\ell$ are quasi-homogeneous, the subalgebra $\mathbb C[z^1,\dots,z^\ell]$ is graded. The following theorem shows that all proper factors have such a form. Furthermore, each proper factor is generated by a unique set of proper generators.
	
	\begin{thm}
		Suppose $\mathscr P$ is a proper factor of $\mathscr A^W$, then there exists a unique set of proper generator $\{z^1,\dots,z^\ell\}$ such that $\mathscr P = \mathbb C[z^1,\dots,z^\ell]$.
		\begin{proof}
			By definition, we have $\mathscr A^W = \mathscr P\otimes \mathbb C[\lambda]$, hence each $y^i \in \mathscr A^W$ can be represented in the form
			\begin{align}
				\label{compare homo}
				y^i = \sum_{\alpha}p^i_\alpha T^i_\alpha(\lambda),\quad p^i_\alpha \in \mathscr P,\  T^i_\alpha(\lambda)\in\mathbb C[\lambda].
			\end{align}
			In particular, we have
			\[y^i\big|_{\lambda=0} = \Bigl(\sum_{\alpha}p^i_\alpha T^i_\alpha(0)\Bigr)\Bigl|_{\lambda=0},\]
			where $\sum_{\alpha}p^i_\alpha T^i_\alpha(0) \in \mathscr P$. So we define $z^i$ to be the $\theta_i$-degree component of $\sum_{\alpha} p^i_\alpha T^i_\alpha(0)$, then $\{z^1,\dots,z^\ell\}$ is a set of proper generators. Since $\mathscr P$ is graded, we have $z^i \in \mathscr P$, thus $\mathbb C[z^1,\dots,z^\ell] \subset \mathscr P$. However, it follows from the properties of proper generators that $\mathscr A^W = \mathbb C[z^1,\dots,z^\ell] \otimes\mathbb C[\lambda]$. Since $\mathbb C[\lambda]$ is faithfully flat, we have \[\mathbb C[z^1,\dots,z^\ell] = \mathscr P.\]
			
			Now, let us prove the uniqueness of the set of proper generators. Suppose there is another set of proper generators $\{\tilde{z}^1,\dots,\tilde{z}^\ell\}$ such that $\mathscr P = \mathbb C[\tilde{z}^1,\dots,\tilde{z}^\ell]$, then for each $i=1,\dots,\ell$ we have
			\begin{align}
				\label{unique vanish}
				z^i - \tilde{z}^i = \lambda (s^i - \tilde s^i).
			\end{align}
			Here we write $\tilde{z}^i = y^i+\lambda \tilde{s}^i$. Since $z^i - \tilde{z}^i \in \mathscr P = \mathbb C[z^1,\dots,z^\ell]$ and $(s^i - \tilde s^i) \in \mathscr A^W = \mathbb C[z^1,\dots,z^\ell;\lambda]$, if $s^i - \tilde s^i$ do not vanish, then \eqref{unique vanish} gives a nontrivial algebraic relation of $z^1,\dots,z^\ell$ and $\lambda$, which contradicts with their algebraic independence. Therefore, $z^i = \tilde{z}^i$ for each $i=1,\dots,\ell$.
		\end{proof}
	\end{thm}
	
	The above theorem shows that the map
	\begin{align}
		\{\text{proper~generators}\} \to \{\text{proper~factors}\},\quad \{z^1,\dots,z^\ell\} \mapsto \mathbb C[z^1,\dots,z^\ell].
	\end{align}
	is a bijection. As a result, we can characterize a proper factor by its proper generators.
	
	For any given set $\{z^1,\dots,z^\ell\}$ of proper generators of $\mathscr A^W$, we still call the transformation which maps $(x^1,\dots,x^\ell)$ to $(z^1,\dots,z^\ell)$ the $\varphi_\lambda$-transformation, and the corresponding $\varphi_0$-transformation the null-transformation. Note that $\{z^1,\dots,z^\ell\}$ and $\{y^1,\dots,y^\ell\}$ may lead to distinct $\varphi_\lambda$-transformations, but they always share the same null-transformation. In particular, for any element $r$ of the set $S$ defined by \eqref{zzh-4} we have
	\begin{align}
		\label{zr lambda=0 = ...}
		z^r\big|_{\lambda = 0} = \mathrm{e}^{2\pi i x^r}.
	\end{align}
	Hence from Corollary \ref{sum mr logyr = 2pi(omega,x)}
	it follows that
	\begin{align}
		\label{mrzr}
		\sum_{r=1}^\ell m_r\log\left(z^r\big|_{\lambda=0}\right) = 2\pi i(\omega, \hat{\bfx}).
	\end{align}
	
	In the next section, we will impose an additional condition on proper generators of $\mathscr A^W$ in order to construct generalized Frobenius manifold structures on the orbits space of the affine Weyl group $W_a(R)$.
	
	\section{The construction of generalized Frobenius manifold structures}
	\label{GFM structure}
	For a given irreducible reduced root system $R$ in an $\ell$-dimensional Euclidean space $V$, let us fix a weight $\omega$ of the form \eqref{omega = sum mr omegar}, a proper factor $\mathscr P$ of $\mathscr A^W$ with the corresponding set $\{z^1,\dots,z^\ell\}$ of proper generators. We are to impose a certain additional condition on these proper generators, so that we construct a generalized Frobenius manifold structure on a dense open subset of the spectrum of $\mathscr P$.
	
	\subsection{The orbit space and pencil generators}
	Consider the manifold $\mathcal M = \mathbb{C}^\ell$ on which $z^1,\dots,z^\ell$ are global coordinates. Taking $x^1,\dots,x^\ell$ as coordinates on the complexified Euclidean space $V\otimes\mathbb C$, we have a collection of holomorphic maps
	\begin{equation}\label{zh-06-01a}
		\varphi_\lambda\colon V\otimes\mathbb C \to \mathcal M,\quad (x^1,\dots,x^\ell)\mapsto (z^1,\dots,z^\ell),\quad \lambda\in\mathbb C.
	\end{equation}
	Here the $\varphi_\lambda$-transformation is defined in the above section.
	We call $\mathcal M$ the \emph{orbit space of the affine Weyl group} $W_a(R)$ w.r.t. the proper generators $z^1,\dots,z^\ell$. In this section, we will use this collection of maps $\{\varphi_\lambda\}_{\lambda \in \mathbb C}$ to introduce a structure of generalized Frobenius manifold on a dense open subset of $\mathcal M$.
	
	The inner product $(\cdot\,,\cdot)$ of the Euclidean space $V$ yields a contravariant metric 
	\begin{equation}
		\label{a=(.,.)}
		a(\dif x^i, \dif x^j) =a^{ij}\quad \mbox{with}\ (a^{ij})=\left((\alpha_i^\vee, \alpha_j^\vee)\right)^{-1},\quad i, j=1,\dots,\ell
	\end{equation}
	on $V\otimes \mathbb{C}$. 
	By a contravariant metric, we mean a non-degenerate, symmetric bilinear form on the holomorphic cotangent bundle $T^*(V\otimes\mathbb{C})$. We will also call $a$ a metric on $V\otimes \mathbb{C}$. 
	
	\begin{prop} The numbers $\theta_j$ defined by \eqref{weight} can be represented in the form
		\begin{equation}\label{theta r}
			\theta_j = a^{jr}m_r,\quad j=1,\dots,\ell.
		\end{equation}
		Here and in what follows, we assume summation over repeated upper and lower indices.
		\begin{proof}
			From $\omega = \sum_{s}m_s\omega_s = \sum_p \theta_p\alpha_p^\vee$ it follows that
			\begin{align}
				\label{tmp trivial}
				\sum_{p = 1}^\ell\theta_p(\alpha_p^\vee, \alpha_s^\vee)= (\omega, \alpha_s^\vee) =m_s,
			\end{align}
			thus \eqref{theta r} holds true. The proposition is proved.
		\end{proof}
	\end{prop}
	
	The $\varphi_\lambda$-transformation induces a collection of pushing forward $\{(\varphi_\lambda)_*\}_{\lambda \in \mathbb C}$, which transforms $(a^{ij})$ to a contravariant metric $(\varphi_\lambda)_*a$ on $\mathcal M$. We denote 
	\[g_\lambda=\frac 1{4\pi^2}(\varphi_\lambda)_*a=\bigl(g_\lambda^{ij}\bigr),\] where
	\begin{align}\label{zzh-5}
		g_\lambda^{ij} = \frac 1{4\pi^2} \pdf{z^i}{x^r} a^{rs}\pdf{z^j}{x^s},\quad i,j = 1,\dots,\ell.
	\end{align}
	Note that we introduce the factor $1/{4\pi^2}$ to simplify the expressions of $g_\lambda^{ij}$'s. We denote 
	\begin{align}
		\Gamma^{ij}_{g_\lambda,k}=-g_{\lambda}^{ir}\Gamma_{g_{\lambda},rk}^j,
		\quad i, j, k=1,\dots,\ell,
	\end{align}
	and call them the Christoffel symbols of the Levi-Civita connection of the contravariant metric $g_\lambda$, or simply the Christoffel symbols of $g_\lambda$.
	
	\begin{prop} For all indices $i,j,k$, we have
		$g^{ij}_\lambda \in \mathscr A^W$ and $\Gamma^{ij}_{g_\lambda,k} \in \mathscr A^W$.
	\end{prop}
	\begin{proof}
		We first consider the entries of the contravariant metric $g_\lambda$. By definition, we have $g^{ij}_\lambda \in \mathscr A$, hence it suffices to show that $g^{ij}_\lambda$ are invariant w.r.t.  the action of $W_a(R)$. Since the invariance of $g^{ij}_\lambda$ w.r.t.  translations is obvious, we consider actions by elements of the Weyl group $W(R)$. Write $z^j$ as $
		z^j=z^j(x^1,\dots,x^\ell;c)$. 
		For any given $\sigma \in W(R)$, since $z^j(x^1,\dots,x^\ell;c)=z^j(\sigma(x^1,\dots,x^\ell;c))$, we have
		\begin{align}
			\pdf{z^j}{x^r}(x^1,\dots,x^\ell;c)=\pdf{z^j}{x^p}(\sigma(x^1,\dots,x^\ell;c))\pdf{\tilde{x}^p}{x^r},
		\end{align}
		where $\sigma(x^1,\dots,x^\ell;c)=(\tilde{x}^1,\dots,\tilde{x}^\ell;c)$.			Note that $\pdf{\tilde{x}^p}{x^r}$ are constants since $\sigma$ is a linear transformation. Therefore,
		\begin{align*}
			&g^{ij}_\lambda(x^1,\dots,x^\ell;c)=\frac 1{4\pi^2}\pdf{z^i}{x^r}(x^1,\dots,x^\ell;c) a^{rs}\pdf{z^j}{x^s}(x^1,\dots,x^\ell;c)\\
			&\quad =\frac 1{4\pi^2}\pdf{z^i}{x^p}(\sigma(x^1,\dots,x^\ell;c))\pdf{z^j}{x^q}(\sigma(x^1,\dots,x^\ell;c))\pdf{\tilde{x}^p}{x^r}a^{rs}\pdf{\tilde{x}^q}{x^s}.
		\end{align*}
		Since $\sigma$ is an orthogonal transformation, we know that 
		\begin{align*}
			\pdf{\tilde{x}^p}{x^r}a^{rs}\pdf{\tilde{x}^q}{x^s}=a^{pq},
		\end{align*}
		hence we arrive at
		$g^{ij}_\lambda(x^1,\dots,x^\ell;c) = g^{ij}_\lambda(\sigma(x^1,\dots,x^\ell;c))$.
		
		Now we consider the Christoffel symbols $\Gamma^{ij}_{g_\lambda,k}$ of the contravariant metric $g_\lambda$. Since $g_\lambda$ is invariant w.r.t.  the action of $W_a(R)$, it suffices to show that $\Gamma^{ij}_{g_\lambda,k} \in \mathscr A$. From Corollary \ref{equal det}  and Proposition \ref{zh-2026-7-6-1}, it follows that
		\begin{align*}
			\det\left(\pdf{z^i}{x^j}\right) = \mathbf{J}(\hat{\bfx};\lambda) = C\cdot \mathrm{e}^{ \pi i\sum_{\alpha \in R^+}(\alpha, \hat{\bfx})}\prod_{\beta\in R^+}\left(1-\mathrm{e}^{-2\pi i\left(\beta, \bfx\right)}\right),
		\end{align*}
		where $\bfx = c\omega + \hat{\bfx}$.
		For the Christoffel symbols $\Gamma^{ij}_{g_\lambda,k}$, the formula
		\begin{align}
			\Gamma^{ij}_{g_\lambda,k} = \frac 1{4\pi^2}\pdf{z^i}{x^p}\pdff{z^j}{x^q}{x^r}a^{pq}\pdf{x^r}{z^k}
		\end{align}
		holds, hence we have
		\[
		\Gamma^{ij}_{g_\lambda,k} = \frac{P^{ij}_k(\hat{\bfx};\lambda)}{\mathbf{J}(\hat{\bfx};\lambda)},
		\]
		where $P_k^{ij}\in\mathscr A$. Note that $\mathbf{J}(\hat{\bfx};\lambda)$ is anti-invariant and $\Gamma_{\lambda, k}^{ij}$ are invariant w.r.t.  the action of the Weyl group $W(R)$, so the Fourier polynomials $P_k^{ij}(\hat{\bfx};\lambda)$ must be anti-invariant w.r.t.  the action of $W(R)$. Note that for each $\beta \in  R^+$, if $(\beta,\bfx)\in\mathbb Z$, 
		\begin{align}
			\sigma_{\beta}(\bfx) = \bfx - (\beta, \bfx)\beta^\vee = \bfx + n_1\alpha_1^\vee+\cdots+n_\ell\alpha_\ell^\vee,\quad n_1,\dots,n_\ell \in \mathbb Z.\label{(beta,x)}
		\end{align}
		Here $\beta^\vee \in  R^\vee$ is a $\mathbb Z$-combination of $\alpha_1^\vee,\dots,\alpha_\ell^\vee$. It follows from the anti-invariance of $P_k^{ij}$ and \eqref{(beta,x)} that, as long as $(\beta, \bfx)\in \mathbb Z$,
		\[P_k^{ij}(\bfx) = - P_k^{ij}(\sigma_\beta(\bfx)) = -P_k^{ij}(\bfx),\]
		hence $P_k^{ij}$ vanishes along $(\beta,\bfx) \in\mathbb Z$, all zeros of $1-\mathrm{e}^{-2\pi i\left(\beta, \bfx\right)}$. Since $1-\mathrm{e}^{-2\pi i\left(\beta, \bfx\right)}$ has no square factor in $\mathscr A$, it follows from Nullstellensatz's Theorem that $P_k^{ij}$ is divisible by $1-\mathrm{e}^{-2\pi i\left(\beta, \bfx\right)}$. Therefore, it is also divisible by $\mathbf{J}(\hat{\bfx};\lambda)$, and we have $\Gamma^{ij}_{g_\lambda,k} \in \mathscr A$. The proposition is proved.
	\end{proof}
	
	\begin{rem}
		\label{global def}
		This proposition implies that the Christoffel symbols $\Gamma^{ij}_{g_\lambda,k}$ can be defined globally on $\mathcal M$, although $g_\lambda$ may degenerate at some points.
	\end{rem}
	
	The above proposition shows, in particular, that $g^{ij}_{\lambda}$ and $\Gamma_{g_\lambda,k}^{ij}$ are polynomials of $\lambda$. In order to construct generalized Frobenius manifold structures on a dense open subset of the orbit space $\mathcal M$, we are to impose an additional condition on the proper generators of $\mathscr A^W$, so that $g_{\lambda}$ yields a flat pencil of metrics \cite{Du-1998} on $\mathcal M$. 
	
	\begin{defn}
		Suppose $\{z^1,\dots,z^\ell\}$ be a set of proper generators of $\mathscr A^W$. We say that it is a set of \emph{pencil generators} of $\mathscr A^W$ if the associated metric $g_\lambda=\bigl(g^{ij}_\lambda\bigr)$ and its Christoffel symbols $\Gamma_{g_\lambda,k}^{ij}$ depend at most linearly on $\lambda$, and the metric $\partial_\lambda g_\lambda^{ij}$ is non-degenerate at generic points.
	\end{defn}
	From the above definition, for a set of proper generators $\{z^1,\dots,z^\ell\}$, $g^{ij}_\lambda$ can be represented in the form	
	\[g^{ij}_\lambda = g^{ij} + \lambda \eta^{ij},\] 
	where $\eta=(\eta^{ij})$ is non-degenerate at generic points of $\mathcal{M}$, and the Christoffel symbols $\Gamma^{ij}_{g_\lambda,k}$ of $g_\lambda$ can be represented in terms of that of $g$ and $\eta$ as follows:
	\begin{equation}\label{zzh-7}
		\Gamma^{ij}_{g_{\lambda},k}=\Gamma^{ij}_{g,k}+\lambda \Gamma^{ij}_{\eta,k}.
	\end{equation}
	Therefore, a set of \emph{pencil generators} of $\mathscr A^W$ yields a flat pencil \cite{Du-1998} of metrics $g, \eta$.
	Note that the flat metric $g$ is independent of the choice of proper generators, since
	\begin{align}
		g= g_0= \frac 1{4\pi^2}(\varphi_0)_*a
	\end{align}
	is induced merely by the null-transformation which does not depend on the choice of proper generators.
	
	We will present examples of pencil generators in Sect.\,\ref{examples} and in a subsequent paper \cite{JLTZ-2}. Hereafter, ${z^1,\dots,z^\ell}$ denotes a set of pencil generators for $\mathscr A^W$. 
	
	\subsection{Quasi-homogeneous flat coordinates}
	We are to show in this subsection the existence of a system of quasi-homogeneous flat coordinates  for the flat metric $\eta$.  
	
	Flat coordinates of $\eta$ satisfy the system of equations
	\begin{align}
		\label{flat equation}
		\pdff{t}{z^i}{z^j} = \Gamma_{\eta, ij}^k\pdf{t}{z^k},\quad i,j = 1,\dots,\ell.
	\end{align}
	Here $\Gamma_{\eta, ij}^k$ are the Christoffel symbols of the Levi-Civita connection of $\eta$, they are given by
	\begin{align}
		\label{Gamma ijk def}
		\Gamma_{\eta, ij}^k=\frac 12 \eta^{kr}\left(\pdf{\eta_{ir}}{z^j} +\pdf{\eta_{jr}}{z^i} -\pdf{\eta_{ij}}{z^r} \right).
	\end{align}
	From Definition \ref{zzh-6} it follows that
	\begin{align}
		\deg g^{ij} =\theta_i +\theta_j,\quad \deg \eta^{ij}=\theta_i + \theta_j - \kappa,
	\end{align}
	which lead to
	\begin{align}
		\label{Gamma ijk are homo}
		\deg g_{ij} = -\theta_i - \theta_j,\quad  \deg \eta_{ij}=\kappa - \theta_i - \theta_j,\quad \deg \Gamma_{\eta, ij}^k=\theta_k-\theta_i-\theta_j.
	\end{align}
	We are to use the above quasi-homogeneity property to prove the existence of quasi-homogeneous flat coordinates of $\eta$. To this end, let us first clarify the solution space of the system of equations \eqref{flat equation}. Let
	$\mathcal{W}$ be the universal covering space of $\mathcal M\setminus D_\eta$ with
	\begin{equation}\label{zh-05-25a}
		D_\eta = \{z \in \mathcal M\mid \det\eta(z) = 0\}.
	\end{equation}
	We consider solutions of the system of equations \eqref{flat equation} on the space
	\begin{align*}
		\mathscr Q(\mathcal{W}) = \mathscr{O}(\mathcal{W})/\mathbb C,
	\end{align*}
	where $\mathscr{O}(\mathcal{W})$ refers to the holomorphic function ring on $\mathcal{W}$. 
	Denote the solution space of \eqref{flat equation} by 
	\[
	\mathcal K =\left\{t \in \mathscr Q(\mathcal{W})\mid\pdff{t}{z^i}{z^j} = \Gamma_{\eta, ij}^k\pdf{t}{z^k},\ i,j = 1,\dots,\ell\right\}.
	\]
	It follows from the theory of differential equations that $\mathcal K$ is an $\ell$-dimensional space, since a solution is uniquely determined by the values of its $\ell$ partial derivatives at an arbitrary point.
	
	Now we introduce the Euler vector field
	\begin{align}
		\label{Euler}
		E = \sum_{r=1}^{\ell}\theta_r z^r \pdf{}{z^r}.
	\end{align}
	Then for any quasi-homogeneous function $f$ of degree $n$, the following identity holds true:
	\begin{align}
		\mathcal L_E f = nf,
	\end{align}
	here $\mathcal L_E$ is the Lie derivative along $E$. Hence quasi-homogeneous functions can be seen as \textit{eigenvectors} of $\mathcal L_E$, of which the \textit{eigenvalues} are their degrees. Now we consider the action of $\mathcal L_E$ on the solution space $\mathcal K$.
	
	\begin{prop}
		The space $\mathcal K$ is $\mathcal L_E$-invariant, i.e., $\mathcal L_E(\mathcal K) \subseteq \mathcal K$.
	\end{prop}
	\begin{proof}
		Denote
		\begin{align}
			\mathcal L_{ij} = \pdff{}{z^i}{z^j} - \Gamma_{\eta, ij}^k\pdf{}{z^k}.
		\end{align}
		By using \eqref{Gamma ijk are homo} we obtain 
		\[\mathcal L_E \Gamma_{\eta, ij}^k = (\theta_k-\theta_i-\theta_j)\Gamma_{\eta, ij}^k,\]
		from which it follows that
		\begin{align}
			\left[\mathcal L_E, \mathcal L_{ij}\right] = -(\theta_i + \theta_j)\mathcal L_{ij}.
		\end{align}
		Hence if $t \in \ker\mathcal L_{ij}$, then $\mathcal L_E t\in \ker\mathcal L_{ij}$. The proposition is proved.
	\end{proof}
	
	In what follows, we assume for conciseness that $\mathcal L_E|_{\mathcal K}$ is diagonalizable. This assumption of diagonalization of the vector field $E$ implies that $\eta$ admits $\ell$ quasi-homogeneous flat coordinates which we denote by $t^1,\dots,t^\ell$. Let
	\begin{align}\label{zh-06-02a}
		d_\alpha = \deg t^\alpha,\quad \alpha = 1,\dots,\ell,
	\end{align}
	then we have
	\begin{align}
		E = \sum_{r=1}^\ell \theta_r z^r \pdf{t^\alpha}{z^r}\tangentvector{t^\alpha} = \sum_{\alpha=1}^\ell d_\alpha t^\alpha\tangentvector{t^\alpha}.
	\end{align}
	
	Now let us consider the flat pencil of metrics $g$ and $\eta$ in the flat coordinates $t^1,\dots,t^\ell$. We denote their components by $g^{\alpha\beta}$ and $\eta^{\alpha\beta}$ with Greek indices respectively. Recall that we use $g^{ij}$ and $\eta^{ij}$ with Latin indices to denote the components of $g$ and $\eta$ in the  coordinates $z^1,\dots, z^\ell$. Since $\eta^{\alpha\beta}$ are constants, from 	\begin{align}\label{zh-05-25b}
		\deg \eta^{\alpha\beta} =d_\alpha + d_\beta - \kappa 
	\end{align}
	it follows that $\eta^{\alpha\beta}$ vanishes if $d_\alpha + d_\beta \not= \kappa$. Thus by performing, if necessary, a degree preserving linear transformation, we can choose the flat coordinates $t^1,\dots, t^n$ of $\eta$ such that
	\begin{align}
		\label{eta standard}
		\eta^{\alpha\beta} = \delta_{\beta, \alpha^*}.
	\end{align}
	Here $ *$ is an involution
	\[
	*\colon \{1,\dots,l\} \to \{1,\dots,l\},\quad \alpha \to \alpha^*,
	\]
	which satisfies the relation $d_{\alpha} + d_{\alpha^*} = \kappa$. 
	Usually, we prefer to choose the flat coordinates so that $\alpha^* =\ell+1-\alpha$. 
	
	
	We note that $d_1,\dots,d_\ell$ are the eigenvalues of the operator $\nabla^\eta E$ on each point of $\mathcal M\setminus D_\eta$, corresponding to the eigenvectors $\tangentvector{t^1},\dots,\tangentvector{t^\ell}$, so they are also the eigenvalues of the representation matrix $A$ of $\nabla^\eta E$ in the basis $\tangentvector{z^1},\dots,\tangentvector{z^\ell}$. Since
	\begin{align}
		\nabla_{\tangentvector{z^i}}^\eta E=\nabla_{\tangentvector{z^i}}^\eta\sum_{k=1}^\ell \theta_k z^k \tangentvector{z^k}=\left(\sum_{k=1}^{\ell}\theta_kz^k\Gamma_{\eta,ik}^j + \theta_i \delta^j_i\right)\tangentvector{z^j},
	\end{align}
	the matrix $A$ has entries
	\begin{align}
		\label{A matrix}
		A_j^i =\sum_{k=1}^{\ell}\theta_k z^k\Gamma_{\eta,j k}^i + \theta_j \delta^i_j.
	\end{align}
	It has eigenvectors $\left(\pdf{t^\alpha}{z^1},\dots,\pdf{t^\alpha}{z^\ell}\right)^T$ with eigenvalues $d_\alpha$, i.e.,
	\begin{equation}
		A\begin{pmatrix}
			\pdf{t^\alpha}{z^1}\\\vdots\\\pdf{t^\alpha}{z^r}
		\end{pmatrix}=d_\alpha\begin{pmatrix}
			\pdf{t^\alpha}{z^1}\\\vdots\\\pdf{t^\alpha}{z^r}
		\end{pmatrix}.\label{Axidxi}
	\end{equation}
	
	\begin{lem}\label{zh-06-02b}
		All $d_\alpha$ have positive real parts.
	\end{lem}
	The proof of this lemma will be given in Appendix \ref{degge0}. In fact, we will show in Corollary \ref{deg rational} that all $d_\alpha$ are positive rational numbers.
	
	\subsection{The Frobenius algebra}
	\label{the frobenius algebra}
	Denote by 
	\begin{align}
		\Gamma_\gamma^{\alpha\beta}=-g^{\alpha\epsilon}\Gamma_{\epsilon\gamma}^\beta,\quad \al, \beta, \gamma=1,\dots,\ell
	\end{align}
	the Christoffel symbols of the Levi-Civita connection of the contravariant $g$ in the flat coordinates $t^1,\dots, t^\ell$ of $\eta$.
	Then from \eqref{zzh-7} we know that $\Gamma_\gamma^{\alpha\beta}$ are also the Christoffel symbols of the Levi-Civita connection of $g+\lambda\eta$ for each $\lambda \in \mathbb C$. Moreover, we know from Remark \ref{global def} that $\Gamma^{\alpha\beta}_\gamma$ are defined globally on $\mathcal M \setminus D_\eta$.
	
	We define an algebra structure on $T^*(\mathcal M\setminus D_\eta)$ by
	\begin{align}
		\label{Frobenius alg}
		\dif t^\alpha \cdot\dif t^\beta =\frac {\kappa}{d_\beta}\Gamma_\gamma^{\alpha\beta}\dif t^\gamma,\quad \alpha, \beta=1,\dots,\ell,
	\end{align}
	where $d_\beta$ is the degree of $t^\beta$. From Lemma \ref{zh-06-02b}, we know in particular that $d_\beta \not= 0$. This algebra is called the dual Frobenius algebra on $T^*(\mathcal M\setminus D_\eta)$. Denote
	\begin{align*}
		D_0 = \{z \in \mathcal M \, |\, \exists\, i\in S,\, z^i = 0\},\quad D = D_\eta \cup D_0,
	\end{align*}
	and $\mathcal M_D = \mathcal M \setminus D$, then $\mathcal M_D$ is a dense open subset of $\mathcal M$.
	Let us define a one-form
	\begin{align}\label{zh-05-24b}
		\omega_e = -\sum_{r=1}^{\ell}\frac1{\kappa}m_r \dif\log z^r
	\end{align}
	on $\mathcal M_D$. We are to show in this subsection that the multiplication \eqref{Frobenius alg} endows $T^*\mathcal M_D$ a Frobenius algebra structure with the unity $\omega_e$ and the bilinear form $\langle\cdot\,, \cdot\rangle$ given by the flat metric $\eta$.
	
	\begin{prop}\label{zh-05-23}
		The one-form \eqref{zh-05-24b}
		is a left unity of the algebra structure defined by \eqref{Frobenius alg} on $T^*\mathcal{M}_D$.
	\end{prop}
	\begin{proof}
		As a closed condition, it suffices to prove this proposition on generic points of $\{\det g \not= 0\}$ on $\mathcal M_D$. From the definition of the multiplication \eqref{Frobenius alg} it follows that
		\begin{align}
			\omega_e\cdot \dif t^\beta = \frac {\kappa}{d_{\beta}}~\nabla_g^{\omega_e}\dif t^\beta=\frac {\kappa}{d_\beta}\nabla^g_{g^\sharp(\omega_e)}\dif t^\beta.
		\end{align}
		Here $\nabla_g$ (resp. $\nabla^g$) is the contravariant (resp. covariant dual) Levi-Civita connection, and $g^\sharp$ is the canonical isomorphism
		\begin{align}
			g^\sharp\colon T^*\mathcal M_D \to  T^{**}\mathcal M_D=T\mathcal M_D,\quad  \omega \mapsto g(\omega,\cdot).
		\end{align}
		So it suffices to show that
		\begin{align}
			\label{tmp unit target}
			\frac {\kappa}{d_\beta}\nabla^g_{g^\sharp(\omega_e)}\dif t^\beta=\dif t^\beta.
		\end{align}
		Since the relation between the coordinates $z^1,\dots, z^\ell$ and $t^1,\dots, t^\ell$ is independent of $\lambda$, in the following proof of the validity of \eqref{tmp unit target} we can assume that $z^i$ and $t^\alpha$ are related with $x^1,\dots,x^\ell$ via the null-transformation $\varphi_0$. Note that the null-transformation $\varphi_0$ transforms the constant metric $a$ to $g$, hence $x^1,\dots,x^\ell$ are the flat coordinates of $g$. From \eqref{mrzr} we obtain
		\begin{align}
			\varphi_0^*\omega_e = -\frac{2}{\kappa}\pi i \dif (\omega, \hat{\mathbf{x}}),
		\end{align}
		so it follows that
		\begin{align*}
			\frac {\kappa}{d_\beta}\nabla^g_{g^\sharp(\omega_e)}\dif t^\beta\bigr|_{\lambda=0} &=-2\pi i\frac 1{d_\beta}\sum_{s=1}^\ell  m_s \nabla^g_{g^\sharp \dif x^s}\left(\left.\pdf{t^\beta}{x^p}\right|_{\lambda = 0}\dif x^p\right)\\
			&=-2\pi i\frac 1{d_\beta}\sum_{r,s=1}^\ell  m_s \frac 1{4\pi^2} a^{rs}\nabla^a_{\tangentvector{x^r}}\left(\left.\pdf{t^\beta}{x^p}\right|_{\lambda = 0}\right)\dif x^p\\
			&=\frac 1{d_\beta}\left(\sum_{r=1}^\ell  \frac{1}{2\pi i}\theta_r\tangentvector{x^r}\left(\left.\pdf{t^\beta}{x^p}\right|_{\lambda = 0}\right)\right)\dif x^p.
		\end{align*}
		Here we used the relation \eqref{theta r} for the third equality. Note that $t^\beta\bigr|_{\lambda = 0}$ is a Fourier polynomial of degree $d_\beta$ which does not depend on $\lambda$, so $\left.\pdf{t^\beta}{x^p}\right|_{\lambda = 0}$ and 
		\begin{align}
			\sum_{r=1}^\ell  \frac 1{2\pi i}~\theta_r\tangentvector{x^r}\left(\left.\pdf{t^\beta}{x^p}\right|_{\lambda = 0}\right) =\left.d_\beta\pdf{t^\beta}{x^p}\right|_{\lambda = 0}.
		\end{align}
		Thus we arrive at
		\begin{align}
			\frac {\kappa}{d_\beta}\nabla^g_{g^\sharp(\omega_e)}\dif t^\beta\bigr|_{\lambda=0} = \left.\pdf{t^\beta}{x^p}\right|_{\lambda = 0}\dif x^p =\dif t^\beta\bigr|_{\lambda = 0}.
		\end{align}
		The proposition is proved.
	\end{proof}
	
	\begin{cor}
		Denote $\omega_e = e_\alpha\dif t^\alpha$, then the following relations hold true:
		\begin{align}
			\label{unit 1}
			e_{\alpha}g^{\alpha\beta} =\frac1{\kappa}d_\beta t^\beta,\quad
			e_{\alpha}\Gamma_\gamma^{\alpha\beta} = \frac1{\kappa}d_\beta\delta_\gamma^\beta.
		\end{align}
	\end{cor}
	\begin{proof}
		The second relation of \eqref{unit 1} follows immediately from Proposition \ref{zh-05-23}. To prove the first relation of \eqref{unit 1} we can assume, as in the proof of Proposition \ref{zh-05-23}, that the coordinates $t^1,\dots, t^\ell$ and $x^1,\dots,x^\ell$ are related via the null-transformation $\varphi_0$. Thus we have
		\begin{align*}
			e_\alpha g^{\alpha\beta} &=g(\omega_e, \dif t^\beta) = -\frac{2}{\kappa}\pi i g(\dif (\omega, \hat{\mathbf{x}}), \dif t^\beta)\\
			&=\frac {\kappa}{2\pi i}\sum_{s = 1}^\ell   m_s (\dif x^r, \dif x^s)\pdf{t^\beta}{x^r}
			=\frac1{\kappa}\sum_{s = 1}^\ell \frac{\theta_s}{2\pi i}\pdf{t^\beta}{x^s} =d_\beta t^\beta.
		\end{align*}
		The corollary is proved.
	\end{proof}
	\begin{prop}[\cite{2d-tft}, Appendix D]
		The Christoffel symbols $\Gamma_{\gamma}^{\alpha\beta}$ are uniquely determined by the metric $g$ via the relations
		\begin{align}
			\pdf{g^{\alpha\beta}}{t^\gamma} =\Gamma_{\gamma}^{\alpha\beta} + \Gamma_{\gamma}^{\beta\alpha},\quad
			g^{\alpha\epsilon}\Gamma_\epsilon^{\beta\gamma} =g^{\beta\epsilon}\Gamma_\epsilon^{\alpha\gamma}.\label{zh-05-23b}
		\end{align}
		Moreover, we have the following properties for the flat pencil of metrics:
		\begin{align}
			\eta^{\alpha\epsilon}\Gamma_\epsilon^{\beta\gamma} =\eta^{\beta\epsilon}\Gamma_\epsilon^{\alpha\gamma},\quad
			\pdf{\Gamma_{\gamma}^{\alpha\beta}}{t^\delta} =\pdf{\Gamma_{\delta}^{\alpha\beta}}{t^\gamma},\quad
			\Gamma_{\epsilon}^{\alpha\beta}\Gamma^{\epsilon\gamma}_{\delta} =\Gamma_{\epsilon}^{\alpha\gamma}\Gamma^{\epsilon\beta}_{\delta}.\label{zh-05-24}
		\end{align}
	\end{prop}
	It follows from the first two relations of \eqref{zh-05-24} that there exist functions $f^\beta$ such that
	\begin{align}
		\label{FP6}
		\eta^{\alpha\epsilon}\pdff{f^\beta}{t^\epsilon}{t^\gamma}=\Gamma^{\alpha\beta}_\gamma,
	\end{align}
	where $\deg f^\beta = \kappa + d_\beta$ . 
	
	\begin{lem}\label{zh-05-28b}
		The Christoffel symbols satisfy the following relations:
		\begin{align}
			\label{key formula}
			\Gamma_{\gamma}^{\alpha\beta} = \frac {d_\beta}{d_\alpha + d_\beta}\pdf{g^{\alpha\beta}}{t^\gamma}.
		\end{align}
		\begin{proof}
			By multiplying $e_\gamma$ on both sides of the second relation of \eqref{zh-05-23b} and by using \eqref{unit 1} we obtain
			\begin{align}
				\label{tmp E Gamma}
				\sum_{\epsilon} d_\epsilon t^\epsilon \Gamma_\epsilon^{\alpha\beta} = d_\beta g^{\alpha\beta},
			\end{align}
			so it  follows from \eqref{FP6} that
			\begin{align}
				\sum_{\epsilon} d_\epsilon t^\epsilon \tangentvector{t^\epsilon}\left(\eta^{\alpha\nu}\pdf{f^\beta}{t^{\nu}}\right) = (d_\beta + d_\alpha)\eta^{\alpha\nu}\pdf{f^\beta}{t^{\nu}} = d_\beta g^{\alpha\beta}.
			\end{align}
			Thus we have
			\begin{align}
				(d_\alpha + d_\beta)\eta^{\alpha\nu}\pdff{f^\beta}{t^{\nu}}{t^\gamma} =d_\beta \pdf{g^{\alpha\beta}}{t^\gamma},
			\end{align}
			which leads to the validity of \eqref{key formula}. The lemma is proved.
		\end{proof}
	\end{lem}
	
	\begin{thm}
		The algebra defined by the multiplication \eqref{Frobenius alg} on $T^*\mathcal{M}_D$ is a Frobenius algebra with the bilinear form $\langle\cdot\,,\cdot\rangle$ given by the flat metric $\eta$ and with the unity \eqref{zh-05-24b}.
	\end{thm}
	\begin{proof}
		We need to show that the algebra is commutative and associative, and the operation of multiplication is compatible with the bilinear form. From \eqref{key formula} it follows that
		\begin{align}
			\label{integrable 1}
			\frac1{d_\gamma}\Gamma^{\beta\gamma}_\alpha=\frac1{d_\beta}\Gamma^{\gamma\beta}_\alpha,
		\end{align}
		which leads to the commutativity of the algebra. Now by using the third relation of \eqref{zh-05-24} we have
		\begin{align*}
			&\left(\nd t^\al\cdot \nd t^\beta\right)\cdot \nd t^\gamma=
			\frac{\kappa}{d_\beta}\Gamma^{\al\beta}_\xi \nd t^\xi\cdot \nd t^\gamma
			=\frac{\kappa}{d_\beta}\frac{\kappa}{d_\gamma}\Gamma^{\al\beta}_\xi \Gamma^{\xi\gamma}_\zeta\nd t^\zeta\\
			&\quad =\frac{\kappa}{d_\al}\frac{\kappa}{d_\gamma}\Gamma^{\beta\al}_\xi \Gamma^{\xi\gamma}_\zeta\nd t^\zeta=\frac{\kappa}{d_\al}\frac{\kappa}{d_\gamma}\Gamma^{\beta\gamma}_\xi \Gamma^{\xi\al}_\zeta\nd t^\zeta\\
			&\quad =\left(\nd t^\beta\cdot\nd t^\gamma\right)\cdot\nd t^\al
			=\nd t^\al\cdot\left(\nd t^\beta\cdot\nd t^\gamma\right),
		\end{align*}
		so the associativity of the algebra holds true. Finally, from the first relation of \eqref{zh-05-24} we obtain 
		\begin{align*}
			\langle\nd t^\al\cdot\nd t^\beta,\nd t^\gamma\rangle
			=\frac{\kappa}{d_\beta}\eta^{\gamma\xi}\Gamma^{\al\beta}_\xi 
			=\frac{\kappa}{d_\beta}\eta^{\al\xi}\Gamma^{\gamma\beta}_\xi 
			=\frac{\kappa}{d_\gamma}\eta^{\al\xi}\Gamma^{\beta\gamma}_\xi 
			=\langle\nd t^\al,\nd t^\beta\cdot\nd t^\gamma\rangle,
		\end{align*}
		which verifies the compatibility of the multiplication with the bilinear form. 
		The proposition is proved.
	\end{proof}

	The Frobenius algebra on $T^*\mathcal{M}_D$ induces a Frobenius algebra structure on $T\mathcal M_D$ via the isomorphism
	\begin{align}
		\eta^\sharp\colon T^*\mathcal M_D \to  T\mathcal M_D,\quad \omega \mapsto \eta(\omega,\cdot),
	\end{align}
	where $D$ is define by \eqref{zh-05-25a}. 
	Explicitly, the Frobenius algebra structure $(\cdot, \langle\cdot\,,\cdot\rangle, e)$ on $T\mathcal M_D$ is defined by
	\begin{align}
		\label{Frobenius alg on TM}
		X \cdot Y = \eta^\sharp\left(\eta^\flat(X)\cdot\eta^\flat(Y)\right),\quad X, Y\in \rm{Vect}(\mathcal{M}_D),
	\end{align}
	with the bilinear form 
	\begin{align}
		\langle X, Y\rangle=\eta^{-1}(X,Y)=\eta(\eta^\flat(X), \eta^\flat(Y))
	\end{align}
	and the unity 
	\begin{align}
		\label{unit}
		e=\eta^\sharp(\omega_e) = -\frac 1\kappa\mathrm{grad}_\eta \sum_{r=1}^{\ell}m_r \log z^r,
	\end{align}
	where $\eta^\flat$ is the inverse map of $\eta^\sharp$. From \eqref{Euler},
	\eqref{zh-05-25b} and \eqref{zh-05-24b} it follows that
	\[\mathcal{L}_E e=-\kappa e.\]
	
	\subsection{The generalized Frobenius manifold structure on $\mathcal M_D$}
	A generalized Frobenius manifold extends Dubrovin's original definition of a Frobenius manifold \cite{2d-tft}; it satisfies all axioms of the definition of a Frobenius manifold but the flatness
	of the unit vector field \cite{gfm}. 
	
	\begin{defn}[\cite{2d-tft, gfm}]
		A (complex) generalized Frobenius manifold (GFM) is a complex manifold $M$ endowed on each of its tangent spaces $T_pM$ a Frobenius algebra $A_p = (\cdot,\langle \cdot\,,\cdot\rangle ,e)$ which depends holomorphically on $p$,
		here $\cdot$ is the operation of multiplication of the algebra, $\langle \cdot\,,\cdot\rangle$ is a symmetric non-degenerate bilinear form which is compatible with the operation of multiplication, and $e$ is the unity. It is required to satisfy the following axioms:
		\begin{enumerate}[FM1.]
			\item The bilinear form $\langle \cdot\,,\cdot\rangle $ yields a flat metric on $M$.
			\item Define a $3$-tensor $c$ by $c(u,v,w) = \langle u\cdot v,w\rangle $, then the $4$-tensor
			\begin{align}
				\nabla c:(W,X,Y,Z) \to (\nabla_W c)(X,Y,Z)
			\end{align}    
			is symmetric w.r.t. $W,X,Y,Z \in \rm{Vect}(M)$, where $\nabla$ is the Levi-Civita connection of the flat metric.
			\item There exists an Euler vector field $E$ on $M$ such that $\nabla\nabla E = 0$, $\nabla E$ is diagonalizable and
			\begin{align}
				\label{Euler 1}
				\mathcal L_E(X\cdot Y)&= (\mathcal L_E X)\cdot Y + X\cdot(\mathcal L_E Y) + X\cdot Y,\\
				\label{Euler 2}
				\mathcal L_E\langle X,Y\rangle &=\langle \mathcal L_E X, Y\rangle  + \langle X, \mathcal L_E Y\rangle  + (2-d) \langle X,Y\rangle.
			\end{align}
			here $d$ is a certain constant called the charge of $M$.
		\end{enumerate}
	\end{defn}
	
	\begin{lem}
		The Frobenius algebra on $T\mathcal{M}_D$ defined in the last subsection yields a generalized Frobenius manifold structure on $\mathcal{M}_D$
		with the Euler vector field 
		\begin{equation}\label{zh-05-28a}
			\tilde{E}=\frac1{\kappa}E=\sum_{\al=1}^\ell \frac{d_\al t^\al}{\kappa}\frac{\p}{\p t^\al}
		\end{equation}
		and with charge $d=1$.
	\end{lem}
	\begin{proof}
		The validity of the axiom \textit{FM1} follows from the flatness of the metric $\eta$, and that of the axiom \textit{FM2} follows from the relation \eqref{integrable 1} and the second relation of \eqref{zh-05-24}. To check the quasi-homogeneity of the Frobenius manifold structure required by the axiom \textit{FM3}, we need to verify equations \eqref{Euler 1} and \eqref{Euler 2}
		which can be represented in the flat coordinates $t^1,\dots, t^\ell$ as follows:
		\begin{align}
			\p_{\tilde{E}} c_{\al\beta}^\gamma&=\left(1-\frac{d_\al}{\kappa}-\frac{d_\beta}{\kappa}+\frac{d_\gamma}{\kappa}\right) c_{\al\beta}^\gamma,\label{zh-05-26a}\\
			0&=\left(1-\frac{d_\al}{\kappa}-\frac{d_\beta}{\kappa}\right)\eta_{\al\beta}.\label{zh-05-26b}
		\end{align}
		Here 
		\begin{equation}\label{zh-06-02e}
			c_{\al\beta}^\gamma=\frac{\kappa}{d_\rho}\eta_{\al\nu}
			\eta_{\beta\rho}\eta^{\gamma\zeta}\Gamma^{\nu\rho}_\zeta
		\end{equation}
		are the structure constants of the Frobenius algebra
		\begin{equation}
			\frac{\p}{\p t^\al}\cdot\frac{\p}{\p t^\beta}=c_{\al\beta}^\gamma\frac{\p}{\p t^\gamma},
		\end{equation}
		and $(\eta_{\al\beta})=(\eta^{\al\beta})^{-1}$ with $\eta^{\al\beta}=\eta(\nd t^\al,\nd t^\beta)$. The validity of the identities follows from \eqref{zh-05-25b} and the fact that
		$\deg\Gamma^{\al\beta}_\gamma=d_\al+d_\beta-d_\gamma$.
	\end{proof}
	Now we give the main theorem of this paper.
	
	
	\begin{thm}[Main]
		\label{main}
		Suppose $\{z^1,\dots,z^\ell\}$ be a set of pencil generators associated with an irreducible reduced root system $R$ and a fixed weight $\omega$, and the vector field $E$ given by \eqref{Euler} is diagonalizable, then there exists a generalized Frobenius manifold structure of charge $d=1$ on $\mathcal M_D$, of which the flat metric is given by $\eta$ and the multiplication is defined by \eqref{Frobenius alg on TM}. Moreover, 
		the unit vector field $e$ and the Euler vector field $\tilde E$ of this generalized Frobenius are given by \eqref{unit} and \eqref{zh-05-28a} respectively, and the intersection form coincides with $g$.
	\end{thm}
	\begin{proof}
		We only need to show that the intersection form of the generalized Frobenius manifold $\mathcal M_D$ coincides with $g$, i.e.,
		\begin{align}
			\label{intersection form}
			g(\nd t^\alpha, \nd t^\beta)= i_{\tilde{E}}(\nd t^\alpha\cdot\nd t^\beta).
		\end{align}
		Indeed, by using Lemma \ref{zh-05-28b} we obtain
		\begin{align}
			i_{\tilde E}(\dif t^\alpha\cdot\dif t^\beta)=\frac {\kappa \tilde E^\gamma}{d_\beta}\Gamma^{\alpha\beta}_\gamma=\sum_\gamma\frac{d_\gamma t^\gamma}{d_\beta}\Gamma^{\alpha\beta}_\gamma=\sum_\gamma \frac{d_\gamma t^\gamma}{d_\alpha+d_\beta}\pdf{g^{\alpha\beta}}{t^\gamma}=g^{\al\beta}.
		\end{align}
		Here we used the fact that $\deg g^{\alpha\beta} = d_\alpha+d_\beta$.
		The theorem is proved.
	\end{proof}
	
	Since the generalized Frobenius manifold structure of $\mathcal M_D$ is constructed from the root system $R$, the weight $\omega$ and the pencil generators $z^1,\dots,z^\ell $, we denote this generalized Frobenius manifold by
	\begin{align*}
		\mathcal M_D =\mathcal M(R, \omega,\{z^1,\dots,z^\ell \}). 
	\end{align*}
	If there is no ambiguity in the selection of the pencil generators, we simply write it as
	\begin{align*}
		\mathcal M_D =\mathcal M(R, \omega).
	\end{align*}
	
	\section{Polynomiality Property and the Monodromy Groups}
	\label{Polynomial Property and Monodromy Groups}
	In this section, we consider the monodromy group of the generalized Frobenius manifold $\mathcal{M}_D=\mathcal M(R, \omega)$. Note that the monodromy groups of the Frobenius manifold structures that are constructed in \cite{2d-tft, DZ1998} on the orbit spaces of the Coxeter groups and the extended affine Weyl groups coincide with the underlined groups. In the present case the monodromy group $\mathrm{Mono}(\mathcal{M}_D)$ can be embedded into the affine Weyl group $W_a(R)$. More precisely, we will show that the image of this embedding is
	\begin{align}
		\mathrm{Mono}(\mathcal{M}_D) ~\cong~\mathrm{Stab}_{W}(\omega) \ltimes \mathbb Z^\ell,
	\end{align}
	where 
	\begin{align}\label{stab}
		\mathrm{Stab}_{W}(\omega)=\{\sigma \in W(R)\mid\sigma(\omega) = \omega\}
	\end{align}
	is a parabolic subgroup of $W(R)$.
	Moreover, we will construct another generalized Frobenius manifold $\widehat{\mathcal M}_D$ as a covering space of $\mathcal M_D$, on which the flat coordinates $t^1,\dots,t^\ell$ are globally defined, and show that its monodromy group is a subgroup of $\mathrm{Mono}(\mathcal{M}_D)$. To this end, we first need to prove Theorem \ref{z is polynomial of t} by using properties of the periods of $\mathcal{M}_D$,  which states that for the set $\{z^1,\dots,z^\ell\}$ of pencil generators the functions $z^i$ are polynomials of the flat coordinates $t^1,\dots,t^\ell$. 
	
	\subsection{Periods of the generalized Frobenius manifolds}
	Let us consider flat coordinates $\tilde{v}(\lambda)$ of the metric $g+\lambda\eta$ of $\mathcal{M}_D$. They satisfy the equations
	\begin{equation}
		\label{extended flat 1}
		\nabla_{g+\lambda\eta}\dif \tilde{v}(\lambda)=0,
	\end{equation}
	and are called periods of the generalized Frobenius manifold \cite{2d-tft,gfm}.
	Denote
	\begin{align}
		\mathcal M_\Lambda=(\mathcal{M}_D \times \mathbb C^\times)\setminus{\Sigma_\Lambda},
	\end{align}
	where \[\Sigma_\Lambda = \{(z,\lambda)\in \mathcal M_D \times \mathbb C^\times\mid\det (g+\lambda\eta)(z) = 0\}.\] 
	Then solutions of the equation \eqref{extended flat 1} live on
	\begin{align*}
		\mathscr Q(\widetilde{\mathcal M}_\Lambda)= \mathscr O(\widetilde{\mathcal M}_\Lambda)/\mathbb C,
	\end{align*}
	where $\widetilde{\mathcal M}_\Lambda$ is the universal covering space of $\mathcal M_\Lambda$. 
	
	In the coordinates $z^1,\dots,z^\ell$, the equation \eqref{extended flat 1}
	reads
	\begin{equation}
		\label{concrete extended 1}
		\pdff{\tilde{v}}{z^i}{z^j} - \Gamma_{g_\lambda, ij}^k \pdf{\tilde{v}}{z^k}= 0,
	\end{equation}
	here $\Gamma_{g_\lambda, ij}^k$ are the Christoffel symbols of the Levi-Civita connection of $g+\lambda\eta$.
	Note that the $\varphi_\lambda$-transformation \eqref{zh-06-01a} transforms the constant metric $a/4\pi^2$ defined in \eqref{a=(.,.)}
	to $g+\lambda\eta$, hence $x^1,\dots,x^\ell$ are flat coordinates of $g+\lambda\eta$ via the $\varphi_\lambda$-transformation. By the definition of proper generators, there exist Laurant polynomials $Z_1,\dots,Z_\ell$ of $\mathrm{e}^{2\pi i(x^1 + \theta_1 c)},\dots,\mathrm{e}^{2\pi i(x^\ell  + \theta_\ell  c)}$ such that
	\begin{align}
		\label{zr = lambda Zr}
		z^r = \mathrm{e}^{-2\pi i\theta_rc}Z_r\left(x^1+\theta_1c,\dots,x^\ell+\theta_\ell c\right).
	\end{align}
	Therefore, the refined flat coordinates
	\begin{align}
		\check{x}^r(z,\lambda) := x^r + \theta_r c = x^r + \frac{\theta_r}{\kappa}\log\lambda,\quad r=1,\dots,\ell
	\end{align}
	depend only on $z^j \mathrm{e}^{2\pi i\theta_jc},\, j=1,\dots,\ell$. As functions of $z^1,\dots,z^\ell $ and $\lambda$, they are quasi-homogeneous of degree zero, i.e., they satisfy the equation
	\begin{equation}\label{extended flat 2}
		\kappa\lambda\pdf{\tilde{v}}{\lambda} + \sum_{j=1}^{\ell}\theta_j z^j\pdf{\tilde{v}}{z^j}=0.
	\end{equation}
	We denote by $\mathcal S_\Lambda$ the space of periods of 
	$\mathcal{M}_D$ that are quasi-homogeneous of degree zero modulo constant functions, then $\mathcal S_\Lambda$ consists of the solutions to equations \eqref{concrete extended 1} and \eqref{extended flat 2} up to  additive constants, hence
	$\dim_{\mathbb C} \mathcal{S}_{\Lambda}=\ell$, and $\check{x}^1(z,\lambda), \dots,\check{x}^\ell (z,\lambda)$ form a basis of $\mathcal S_\Lambda$.
	
	Now we proceed to find another basis of $\mathcal S_\Lambda$ associated with the flat coordinates $t^1(z),\dots,t^\ell(z)$ of the metric $\eta$. We denote $\rho=\frac1{\lambda}$.
	\begin{thm}
		\label{ext coord}
		For each quasi-homogeneous flat coordinate $t^\alpha = t^\alpha(z)$ of the metric $\eta$ with degree $d_\alpha$, $\alpha=1,\dots,\ell$, there exist a unique quasi-homogeneous period $\tilde{v} \in \mathcal S_\Lambda$ such that
		\begin{align}
			\tilde{v} = \tilde{v}^\alpha(z,\rho) = \rho^{d_\alpha/\kappa}h^\alpha(z,\rho),
		\end{align}
		where is holomorphic at $\rho =0$ with $h^\alpha(z,0) = t^\alpha(z)$, or equivalently, $h^{\alpha}(z, \rho)$ has the form
		\begin{align}
			h^\alpha(z,\rho) = t^\alpha(z) + \sum_{j=1}^\infty h^\alpha_{j}(z) \rho^j.
		\end{align}
	\end{thm}
	\begin{proof}
		Let us look for a solution of \eqref{concrete extended 1} and \eqref{extended flat 2} of the form $\tilde{v}(z, \rho) = \rho^{d_\alpha/\kappa}h^\alpha(z,\rho)$ satisfying the initial condition $h^\alpha(z,0)=t^\al(z)$. Denote $\xi_j = \pdf{h}{z^j}$, then 
		these equations take the form
		\begin{align}
			\label{xi 1}
			\pdf{\xi_j}{z^i}- \Gamma_{ij}^k(\rho) \xi_k&= 0,\\
			\label{xi 2}
			-\kappa\rho\pdf{\xi_j}{\rho} + \sum_{r=1}^{\ell}\theta_r z^r\pdf{\xi_j}{z^r}&=(d_\alpha - \theta_j) \xi_j,
		\end{align}
		where $\Gamma_{ij}^k(\rho)$ are the Christoffel symbols of the Levi-Civita connection of $\eta+\rho g$. 
		By using the equations \eqref{xi 1} we can rewrite the equations \eqref{xi 2} as
		\begin{align}
			\label{vec form}
			\kappa\rho\pdf{\vec{\xi}}{\rho} = \left(A(\rho) - d_\alpha\right)\vec{\xi},
		\end{align}
		where
		\begin{align*}
			\vec{\xi}=\vec{\xi}(\rho)=(\xi_1,\cdots,\xi_\ell)^T,\quad
			A_j^k(\rho) = \sum_{r=1}^\ell \theta_r z^r\Gamma_{rj}^k(\rho) + \theta_j\delta_j^k.
		\end{align*}
		Note that $\rho=0$ is a Fuchsian singularity of equation \eqref{vec form}, so we can solve it by expanding $A(\rho)$ and $\vec{\xi}(\rho)$ into power series of $\rho$ as follows:
		\begin{align*}
			A(\rho)=A + \rho A_1 + \rho^2 A_2 + \dots,\quad
			\vec{\xi}(\rho)=\vec{\xi}_{(0)} + \rho \vec{\xi}_{(1)} + \rho^2 \vec{\xi}_{(2)} + \dots,
		\end{align*}
		here $A=A(0)=(A^k_j(0))$ coincides with the matrix $A$ defined in \eqref{A matrix}, and 
		\begin{equation}\label{zh-06-02d}
			\vec{\xi}_{(0)}=\left(\pdf{t^\alpha}{z^1},\dots,\pdf{t^\alpha}{z^\ell}\right)^T.
		\end{equation}
		Note that $\rho = 0$ is a Fuchsian singularity of equation \eqref{vec form}, its solutions of formal power series always admit a positive radius of convergence. As a result, by expanding equation \eqref{vec form} into power series of $\rho$, we only need to solve the linear equations
		\begin{align}
			(A - d_\alpha)\vec{\xi}_{(0)}=&~\mathbf{0},\label{j-11-02}\\
			(A-d_\alpha-j\kappa)\vec{\xi}_{(j)}=&-\sum_{r=1}^{j}A_r\vec{\xi}_{(j-r)},\quad j=1,2,\dots.\label{zh-06-02c}
		\end{align}
		From \eqref{Axidxi} we know that \eqref{j-11-02} automatically holds. Recall from Lemma \ref{zh-06-02b} that for each $\beta=1,\dots,\ell$, $\mathrm{Re}(d_\beta) > 0$, so from \eqref{zh-05-25b} we have
		$0 < \mathrm{Re}(d_\beta) < \kappa$. This indicates that all the eigenvalues of $A$ have real parts between $0$ and $\kappa$, hence we can solve \eqref{zh-06-02c}
		to obtain a unique solution $\vec{\xi}_{(1)}, \vec{\xi}_{(2)},\dots$ recursively. 
		From the fact that $\vec{\xi}_{(0)}$ gives a solution of the equations \eqref{xi 1} and the compatibility of the equations of \eqref{xi 1} and \eqref{xi 2}, we know that $\vec{\xi}_{(1)}, \vec{\xi}_{(2)},\dots$ yield a solution of 
		the equations \eqref{xi 1} and \eqref{xi 2}, thus we proved the proposition.
	\end{proof}
	
	\subsection{The polynomiality property}
	Since periods $v^1(z, \rho),\cdots, v^\ell(z, \rho)$ of $\mathcal{M}_D$ given by Theorem \ref{ext coord} are linearly independent (indeed, they are linearly independent near $\rho=0$), they form another basis of $\mathcal S_\Lambda$. 
	Thus there exist constants $c_\alpha^\beta$ and $s^\beta$ such that
	\begin{align}
		\label{transition}
		\check{x}^\beta\left(z, 1/\rho\right)=c_\alpha^\beta v^\alpha(z,\rho) + s^\beta.
	\end{align}
	Since the functions $Z_r = Z_r(\check{x}^1,\dots,\check{x}^\ell)$ given by \eqref{zr = lambda Zr} are holomorphic functions of $\check{x}^1,\dots,\check{x}^\ell$, they are also holomorphic functions of $v^1(z,\rho),\dots,v^\ell(z,\rho)$. Hence we can expand them to get 
	\begin{align}
		\label{zr representation}
		z^r = \rho^{-\theta_r/\kappa}Z_r=\rho^{-\theta_r/\kappa}\sum_{n_1,\dots,n_\ell \in \mathbb{Z}_{\ge 0}}a_{n_1,\dots,n_\ell}(v^1(z,\rho))^{n_1}\cdots (v^\ell(z,\rho))^{n_\ell},
	\end{align}
	where $a_{n_1,\dots,n_\ell}$ are some constants. 
	
	In order to prove the polynomial dependence of $z^1,\dots, z^\ell$ on $t^1,\dots, t^\ell$, we need the following lemma.
	\begin{lem}
		\label{math analys exercise}
		Suppose $N$ be a positive integer, $b_1,\dots,b_N$ be distinct real numbers, and $p_1,\dots,p_N$ be complex numbers. If
		\begin{align}
			\label{simple fact}
			\lim_{\rho\to 0}\sum_{j=1}^{N} \rho^{\sqrt{-1}b_j}p_j =0,
		\end{align}
		then $p_1=\cdots=p_N = 0$.
		\begin{proof}	
			Let us define the function $h\colon \mathbb{R}\to\mathbb{C}$ as follows:
			\[h(s)=\sum_{j=1}^N p_j e^{-\sqrt{-1}\,b_j s}.\]
			Then from the assumption of the lemma we get
			\[\lim_{s\to+\infty} h(s)=0,\]
			from which it follows that
			\[\lim_{T\to+\infty}\frac1{T}\int_0^T h(s) e^{\sqrt{-1}\, \al s}\nd s =0,\quad \forall \al\in\mathbb{R}.\]
			By taking $\al=b_j$ for $j=1,\dots, N$, we arrive at $p_j=0$. The lemma is proved.
		\end{proof}
	\end{lem}
	
	\begin{thm}
		\label{z is polynomial of t}
		For each $r = 1,\cdots,\ell$, $z^r$ is a polynomial of $t^1,\cdots,t^\ell $.
		\begin{proof}
			It follows from Theorem \ref{ext coord} that
			\[v^\alpha(z,\rho) = \rho^{d_\alpha/\kappa}t^\alpha(z) + o(\rho^{d_\alpha/\kappa}),\quad \textrm{when}\ \rho \to 0.\]
			Here and in what follows, the limit $\rho \to 0$ is taken pointwise for each $z \in \mathcal M_D$. Thus we have
			\begin{align}
				\label{zr expand}
				z^r=\rho^{-\theta_r/\kappa}\sum_{n_1,\dots,n_\ell \in \mathbb{Z}_{\ge 0}}a_{n_1,\dots,n_\ell}\left(\rho^{(n_1d_1 + \dots + n_\ell d_\ell)/\kappa}(t^1)^{n_1}\cdots (t^\ell )^{n_\ell} + o(\rho^{(n_1d_1 + \dots + n_\ell d_\ell)/\kappa})\right).
			\end{align}
			Let us denote
			\begin{align}
				\label{inf}
				\theta=\inf\{\mathrm{Re}(n_1d_1 + \dots + n_\ell d_\ell)~\big|~a_{n_1, \dots,n_\ell}\not= 0\}.
			\end{align}
			We now show that $\theta = \theta_r$. Suppose $\theta >\theta_r$, then it follows from \eqref{zr expand} that
			\begin{align}
				z^r=\rho^{-\theta_r/\kappa}o(\rho^{\theta_r/\kappa})=o(1),\quad \rho \to 0,
			\end{align}
			However, this is not true if $z^r \not= 0$. On the other hand, if $\theta < \theta_r$, then from \eqref{zr expand} we know that
			\begin{align*}
				z^r=~\rho^{(\theta-\theta_r)/\kappa}\sum_{\mathrm{Re}(n_1d_1 + \dots + n_\ell d_\ell) = \theta}\rho^{\sqrt{-1}\,\mathrm{Im}(n_1d_1 + \dots + n_\ell d_\ell)} a_{n_1,\dots,n_\ell}(t^1)^{n_1}\cdots (t^\ell )^{n_\ell} + o(\rho^{(\theta-\theta_r)/\kappa}),
			\end{align*}
			when $\rho \to 0$. Since each $d_\alpha$ admits positive real part, we know that only finitely many vectors $(n_1,\dots,n_\ell)\in \mathbb Z_{\ge 0}^\ell$ satisfy the condition $\mathrm{Re}(n_1d_1 + \dots + n_\ell d_\ell) = \theta$. Suppose $b_1,\dots,b_N$ be all possible distinct values of $\mathrm{Im}(n_1d_1+\cdots+n_\ell d_\ell)$ for those $(n_1,\dots,n_\ell)$, then we have
			\begin{align*}
				z^r=\rho^{(\theta-\theta_r)/\kappa}\sum_{j=1}^N \rho^{\sqrt{-1}b_j}p_j(t^1,\dots,t^\ell) + o(\rho^{(\theta-\theta_r)/\kappa}), \quad\textrm{when}\  \rho \to 0,
			\end{align*}
			where $p_1,\dots,p_N$ are some polynomials of $t^1,\dots,t^\ell$. Since $(\theta-\theta_r)/\kappa < 0$, we must have
			\begin{align}
				\sum_{j=1}^N \rho^{\sqrt{-1}\,b_j}p_j(t^1,\dots,t^\ell) \to 0, \quad\textrm{when}\  \rho \to 0.
			\end{align}
			Thus, by using Lemma \ref{math analys exercise} we arrive at $p_j=0$ for $j=1,\dots,N$. Since $t^1,\dots,t^\ell$ are functionally independent, all the coefficients $a_{n_1,\dots,n_\ell}$ of the polynomials $p_j$ must vanish. This contradicts with the definition \eqref{inf} of $\theta$, so we have $\theta = \theta_r$. Now, we can represent $z^r$ in the form
			\begin{align}
				z^r &=\sum_{\mathrm{Re}(n_1d_1 + \cdots + n_\ell d_\ell) = \theta_r}\rho^{\sqrt{-1}\,\mathrm{Im}(n_1d_1 + \cdots + n_\ell d_\ell)}a_{n_1,\dots,n_\ell}(t^1)^{n_1}\cdots (t^\ell )^{n_\ell} + o(1),\\
				&=p_0(t^1,\dots, t^\ell) + \sum_{j=1}^N \rho^{\sqrt{-1}b_j}p_j(t^1,\dots,t^\ell) + o(1), \quad \textrm{when}\ \rho \to 0.
			\end{align}
			Here $b_1,\dots,b_N$ are all distinct nonzero possible values of $\mathrm{Im}(n_1d_1+\cdots+n_\ell d_\ell)$, and
			$p_0,p_1,\dots,p_N$ are polynomials of $t^1,\dots,t^\ell$. By usingLemma \ref{math analys exercise} again, we get $p_j=0$ for $j=1,\dots,N$ and $z^r = p_0$,
			so $z^r$ is a polynomial of $t^1,\cdots,t^\ell$.
			The theorem is proved.
		\end{proof}
	\end{thm}
	
	\begin{rem}
		The above theorem shows that the pencil generators are polynomials of the flat coordinates of $\eta$. This property is shared by the Frobenius manifold structures defined on the orbit spaces of Coxeter groups \cite{2d-tft} and of the extended affine Weyl groups  \cite{DZ1998, DSZZ}. The key distinction lies in the inverse relationship: while the flat coordinates $t^\alpha$ are also polynomials of the invariant polynomials $z^j$'s for the Frobenius manifold structures constructed in  \cite{2d-tft} and  \cite{DZ1998}, they may involve radical expressions of $z^j$'s
		for the ones constructed in \cite{DSZZ}. For the generalized Frobenius manifold structures obtained by the construction of the present paper, the flat coordinates $t^\alpha$ are usually complicated functions of $z^j$'s, and even not solvable by radicals for some cases, such as $\mathcal M(C_\ell, \omega_\ell,\{y_1,\dots,y_\ell\})$ for $\ell \ge 5$. 
	\end{rem}
	
	\begin{cor}
		\label{deg rational}
		All degrees $d_\alpha$ are positive rational numbers.
		\begin{proof}
			By Lemma \ref{zh-06-02b}, it is sufficient to prove that $d_\alpha \in \mathbb Q$. From Theorem \ref{z is polynomial of t} we know that each $z^r=z^r(t), r=1,\dots\ell$ is a quasi-homogeneous polynomial of $t^1,\dots, t^\ell$ of degree $\theta_r$, so for any of its monomial term of the form $c (t^1)^{n_1}\cdots (t^\ell)^{n_\ell}$ we have
			\[n_1 d_1+\cdots+ n_\ell d_\ell = \theta_r.
			\]
			Since $z^1(t),\dots, z^\ell(t)$ are functionally independent, among their monomial terms we can find $\ell$ terms 
			\[f_k=c_k (t^1)^{n_{k1}}\cdots (t^\ell)^{n_{k\ell}},\quad k=1,\dots,\ell\]
			which are functionally independent, and have degrees $\theta_{m_1},\dots,\theta_{m_\ell}$. Thus we obtain the following system of linear equations for $d_1,\dots,d_\ell$:
			\begin{align}
				\label{solve-degrees}
				n_{k1}d_1+\cdots+n_{k\ell}d_\ell = \theta_{m_k},\quad k= 1,\dots,\ell.
			\end{align}
			Due to the functional independence of $f_1,\dots,f_\ell$, the Jacobian determinant 
			\[
			\det\begin{pmatrix}
				\pdf{f_1}{t^1}&\cdots&\pdf{f_1}{t^\ell}\\
				\vdots & \ddots & \vdots\\
				\pdf{f_\ell}{t^1}&\cdots & \pdf{f_\ell}{t^\ell}
			\end{pmatrix} = \frac{f_1\cdots f_\ell}{t^1\cdots t^\ell}\det\begin{pmatrix}
				n_{11} &\cdots&n_{1\ell}\\
				\vdots & \ddots & \vdots\\
				n_{\ell 1}&\cdots & n_{\ell\ell}
			\end{pmatrix}
			\]
			does not vanish identically. It implies that the system of equations \eqref{solve-degrees} has non-degenerate coefficient matrix, hence it determines $d_1,\dots,d_\ell$ uniquely in $\mathbb Q$. The corollary is proved.
		\end{proof}
	\end{cor}
	
	
	Note that the flat coordinates $t^1,\dots,t^\ell $ as functions of $z^1,\dots,z^\ell$ live on the universal covering $\widetilde{\mathcal M}_D$ of $\mathcal M_D$, hence $\pi_1(\mathcal M _D)$ acts on those functions. Denote the stable subgroup of this action by 
	\begin{align}
		H=\{\gamma \in \pi_1(\mathcal M_D)\mid\gamma^*t^{\alpha}(p)=t^{\alpha}(p),\ \forall p \in\widetilde{\mathcal M}_D,~\alpha=1,\dots, \ell\},
	\end{align}
	then $H$ corresponds to a covering space $\widehat{\mathcal M}_D$ of $\mathcal M_D$. Concretely, we may define an equivalence relation $\sim$ on $\widetilde{\mathcal M}_D$ by
	\begin{align}
		p\sim p' \iff \pi(p) =  \pi(p'),\quad t^\alpha(p) = t^\alpha(p'),\quad \alpha = 1,\dots,\ell.
	\end{align}
	Here $\pi\colon \widetilde{\mathcal M}_D \to \mathcal M _D$ is the canonical projection. Then we have
	\begin{align}
		\label{MH}
		\widehat{\mathcal M}_D =\widetilde{\mathcal M}_D/\sim
	\end{align}
	with covering map
	\[\widehat{\pi}:~\widehat{\mathcal M}_D \to \mathcal M_D.\]
	
	\begin{cor}
		\label{ti are coordinates}
		Define the covering space $\widehat{\mathcal M}_D$ of $\mathcal M_D$ as \eqref{MH}. Then $(t^1,\cdots,t^\ell)$ serves as global coordinates of $\widehat{\mathcal M}_D$.
	\end{cor}
	\begin{proof}
		To prove that $(t^1,\cdots,t^\ell)$ are global coordinates, it suffices to show that
		\begin{align}
			\label{tmp t}
			(t^1,\cdots,t^\ell)\colon \widehat{\mathcal M}_D \to \mathbb {C}^\ell
		\end{align}
		is injective. Suppose $p, p' \in \widehat{\mathcal M}_D$ satisfy $t^\alpha(p) = t^\alpha(p'),~ \alpha = 1,\cdots,\ell$, then by using Theorem \ref{z is polynomial of t} we know that
		\[z^r(t^1(p),\cdots,t^\ell(p)) = z^r(t^1(p'),\cdots,t^\ell(p')).\]
		This implies that $\widehat{\pi}(p) = \widehat{\pi}(p')$, hence $p \sim p'$, i.e., they are identified on $\widehat{\mathcal M}_D$. Therefore, \eqref{tmp t} is injective.
	\end{proof}
	
	\begin{cor}
		Suppose $F$ is the prepotential of $\widehat{\mathcal M}_D$, then all the second derivatives of $F$ are rational functions of $t^1,\cdots,t^\ell$.
	\end{cor}
	
	\begin{proof}
		From Theorem \ref{z is polynomial of t}, one can verify that each $g^{\alpha\beta} = g(\dif t^\alpha, \dif t^\beta)$ is a rational function of $t^1,\cdots,t^\ell$. Now let us look at properties of the potential $F$ of the generalized Frobenius manifold $\widehat{\mathcal{M}}_D$, which is a function of the flat coordinates $t^1,\dots,t^\ell$ and is defined by the relations
		\[
		\eta^{\gamma\nu}\frac{\p^3 F}{\p t^\al\p t^\beta\p t^\nu}=c^\gamma_{\al\beta},\quad
		\mathcal{L}_{\tilde{E}} F=2 F+\textrm{quadratic functions of}\ t^1,\dots, t^\ell,
		\]
		up to the addition of quadratic functions of $t^1,\dots, t^\ell$. Here the structure constants $c^\gamma_{\al\beta}$ of the Frobenius algebra are defined in \eqref{zh-06-02e}. From the relations \eqref{key formula} and \eqref{zh-06-02e}, we can choose the potential $F$ such that
		\begin{equation}
			\eta^{\al\nu}\eta^{\beta\xi}\frac{\p^2 F}{\p t^\nu\p t^\xi}=\frac{\kappa}{d_\al+d_\beta} g^{\al\beta}.
		\end{equation}
		Therefore, all second derivatives of $F$ are rational functions of $t^1,\dots, t^\ell$.
	\end{proof}
	
	\subsection{The Monodromy group of $\mathcal M(R,\omega)$}
	The definition of the monodromy group of a Frobenius manifold was given by Dubrovin in \cite{2d-tft}, and can be extended directly to the case of generalized Frobenius manifolds. Let us first recall this defintion. 
	
	Let $M$ be an $\ell$-dimensional generalized Frobenius manifold with flat metric $\eta$ and intersection form $g$. Consider the Gauss-Manin equation
	\begin{align}
		\label{flat of g}
		\nabla_g\dif v=0.
	\end{align}
	for the flat coordinates of $g$. Denote 
	\[\Sigma = \{z \in  M\mid \det g(z) = 0\},\] 
	then the solutions of \eqref{flat of g} live on the universal covering $\mathcal{V}$ of $ M\setminus\Sigma$. Suppose $v^1,\dots,v^\ell$ be flat coordinates of $g$, then they span, together with constant functions, the solution space of \eqref{flat of g}. 
	For any loop $\gamma \in \pi_1( M\setminus\Sigma)$, $\gamma$ acts on each $v^\alpha$ by analytic continuation. Since this action preserves solutions of \eqref{flat of g}, $\gamma^* v^\alpha$ is still a flat coordinate.
	Therefore, $\gamma$ yields an affine transformation 
	\begin{equation}
		\label{act on S}
		\gamma^*v^\alpha=A_\beta^\alpha(\gamma)v^\beta + B^\alpha(\gamma),\quad \al=1,\dots,\ell
	\end{equation}
	on the solution space of \eqref{flat of g}. Here the linear part $A^\alpha_\beta$ is orthogonal w.r.t. the metric $g$. Let
	\begin{align}
		\mathcal S=\left\{ v\in\mathscr O(\mathcal{V})/\mathbb C \mid v ~\text{solves the equations}~\eqref{flat of g}\right\},
	\end{align}
	and $E^\ell$ be the dual space of $\mathcal S$. Then $v^1,\dots,v^\ell$ span $\mathcal S$, and they also give a  coordinate system on $E^\ell$. 
	The transformation \eqref{act on S} yields an affine transformation on $E^\ell$, which we also denote by $\gamma^*$. This transformation on $E^\ell$ is an isometry w.r.t. the metric $g$, hence we obtain a representation
	\begin{align}
		\label{pi1 represent}
		r\colon\pi_1( M\setminus\Sigma) \to \mathrm{Iso}(E^\ell) \cong O(\ell,\mathbb{C}) \ltimes \mathbb C^\ell,\quad \gamma \mapsto \gamma^*
	\end{align}
	of the fundamental group of $M\setminus\Sigma$.
	\begin{defn}[\cite{2d-tft}, Appendix G.]
		The image of $r$ given in \eqref{pi1 represent} is called the monodromy group of $ M$, which is denoted by
		\begin{align}
			\mathrm{Mono}( M)=r(\pi_1(\mathcal M\setminus\Sigma)) \le O(\ell,\mathbb{C}) \ltimes \mathbb C^\ell.
		\end{align}
	\end{defn}
	
	In our construction, the inner product space $E^\ell$ can be identified with $V\otimes\mathbb C$. Let
	\[
	\Sigma = \left\{z \in \mathcal M_D\mid \det \left(g^{ij}(z)\right) = 0\right\},
	\]
	and $\mathcal U(\mathcal M_D)$ be the universal covering space of $\mathcal M_D \setminus \Sigma$. Then the flat coordinates of $g$ are $x^1\bigr|_{\lambda=0},\cdots,x^\ell \bigr|_{\lambda=0}$, which are functions on $\mathcal U(\mathcal M_D)$. Moreover, $x^1\bigr|_{\lambda=0},\cdots,x^\ell \bigr|_{\lambda=0}$ are also coordinates of $V\otimes \mathbb C$, so the complexified Euclidean space $E^\ell$ is canonically isomorphic to $V\otimes\mathbb C$. Hence we obtain a representation
	\begin{align}
		\label{pi1 representation Weyl case}
		r\colon \pi_1(\mathcal M_D \setminus \Sigma) \to \mathrm{Iso}(V\otimes\mathbb C) \cong O(\ell,\mathbb{C}) \ltimes \mathbb C^\ell,\quad \gamma \mapsto \gamma^*,
	\end{align}
	where
	\begin{align}
		\gamma^*x^i\bigr|_{\lambda=0}=A_j^i(\gamma)x^j\bigr|_{\lambda=0} + B^i(\gamma).
	\end{align}
	Here the image $r(\pi_1(\mathcal M_D \setminus \Sigma))$ is the monodromy group of $\mathcal M_D$. Now we proceed to compute the monodromy group of the generalized Frobenius manifolds $\mathcal M_D$ and $\widehat{\mathcal M}_D$. We will end this subsection by proving that
	\[
	\mathrm{Mono}(\mathcal M_D) \cong \mathrm{Stab}_{W}(\omega)\ltimes \mathbb Z^\ell
	\]
	and $\mathrm{Mono}(\widehat{\mathcal M}_D)$ is a subgroup of $\mathrm{Mono}(\mathcal M_D)$.
	\begin{prop}
		\label{grp embed}
		There exists a natural embedding $\mathrm{Mono}(\widehat{\mathcal M}_D) \hookrightarrow \mathrm{Mono}(\mathcal M_D)$.
	\end{prop}
	
	\begin{proof}
		Recall that the covering map
		\begin{align}
			\widehat{\pi}\colon \widehat{\mathcal M}_D \to \mathcal M _D
		\end{align}
		is nondegerate, so from
		\begin{align}
			\det \left(g^{\alpha\beta}(t)\right) = \det \left(g^{ij}(z)\right)\det\left(\pdf{t^\alpha}{z^j}\right)^2
		\end{align}
		it follows that $\left(g^{\alpha\beta}(t)\right)$ is degenerate if and only if $\left(g^{ij}(z(t))\right)$ is degenerate. In other words, 
		\[
		(\widehat{\pi})^{-1}(\Sigma)=\widehat{\Sigma}:= \left\{t \in \widehat{\mathcal M}_D\mid \det 
		\bigl(g^{\alpha\beta}(t)\bigr) = 0\right\},
		\]
		and the map
		\begin{align}
			\widehat{\pi}: \widehat{\mathcal M}_D \setminus\widehat{\Sigma} \to \mathcal M_D \setminus \Sigma
		\end{align}
		is also a covering map, hence $\pi_1(\widehat{\mathcal M}_D \setminus\widehat{\Sigma})$ is isomorphic to a subgroup of 
		$\pi_1(\mathcal M_D \setminus \Sigma)$.
		Thus, from the above-mentioned definition of the monodromy groups of generalized Frobenius manifolds, it follows the existence of a natural embedding
		\begin{align}
			\label{bigger group}
			\mathrm{Mono}(\widehat{\mathcal M}_D) \hookrightarrow \mathrm{Mono}(\mathcal M_D)
		\end{align}
		as subgroups of $\mathrm{Iso}(V\otimes\mathbb C)$. The proposition is proved.
	\end{proof}
	Now we proceed to show that $\mathrm{Mono}(\mathcal M_D)$ is isomorphic to $\mathrm{Stab}_{W}(\omega) \ltimes\mathbb Z^\ell$.
	%
	Recall that
	\begin{align}
		\varphi_0\colon V\otimes\mathbb C \to \mathcal M
	\end{align}
	is a branched covering, which induces an unbranched covering map
	\begin{align}
		\varphi_0\colon V\otimes\mathbb C\setminus \varphi_0^{-1}(D\cup\Sigma) \to \mathcal M_D \setminus \Sigma.
	\end{align}
	Since $\mathcal U(\mathcal M_D)$ is the universal covering space of $\mathcal M_D \setminus \Sigma$, it is also the universal covering space of $V\otimes\mathbb C\setminus \varphi_0^{-1}(D\cup\Sigma)$. We denote the corresponding covering map by
	\begin{align}
		j_0\colon \mathcal U(\mathcal M_D) \to V\otimes\mathbb C\setminus \varphi_0^{-1}(D\cup\Sigma),\quad s \mapsto (x^1\bigr|_{\lambda=0}(s),\cdots,x^\ell \bigr|_{\lambda=0}(s)).
	\end{align}
	We have the following commutative diagram of covering spaces:
	\begin{center}
		\begin{tikzpicture}[
			>=Stealth,
			scale=0.8,
			every node/.style={align=center},
			node distance=1.5cm
			]
			
			\def\baseLength{5}
			\coordinate (A) at (0, 0);
			\coordinate (B) at (-\baseLength/1.8, -\baseLength/2);
			\coordinate (C) at (0,-\baseLength);
			\coordinate (D) at (\baseLength/1.3, -\baseLength/2);
			
			\tikzset{
				formula/.style={
					rectangle,
					minimum width=3cm,
					inner sep=2pt
				}
			}
			
			\node[formula] (A) at (A) {$\mathcal U(\mathcal M_D)$};
			\node[formula] (B) at (B) {$\widehat{\mathcal M}_D \setminus\widehat{\Sigma}$};
			\node[formula] (C) at (C) {$\mathcal M_D \setminus \Sigma$.};
			\node[formula] (D) at (D) {$V\otimes\mathbb C\setminus\varphi_0^{-1}(D\cup\Sigma)$};
			
			\draw[->] (A) -- node[above,pos=0.5,xshift=-2pt,yshift=2pt] {} (B);
			\draw[->] (B) -- node[left,pos=0.5] {$\widehat{\pi}$} (C);
			\draw[->] (A) -- node[right,pos=0.5,xshift=-1pt] {} (C);
			\draw[->] (A) -- node[right,pos=0.4,xshift=5pt,yshift=5pt] {$j_0$} (D);
			\draw[->] (D) -- node[right,pos=0.5,xshift=1pt,yshift=-2pt] {$\varphi_0$} (C);
		\end{tikzpicture}
	\end{center}
	All the maps in the diagram are covering maps. 
	
	To determine the image of $\pi_1(\mathcal M_D\setminus\Sigma)$ under the map $r$ that is introduced in \eqref{pi1 representation Weyl case}, we are to use the isomorphism
	\begin{align}
		\label{pi1=deck}
		\pi_1(\mathcal M_D \setminus \Sigma)\cong\mathrm{Deck}\left(\frac{\mathcal U(\mathcal M_D)}{\mathcal M_D\setminus \Sigma}\right),
	\end{align}
	here $\mathrm{Deck}(X/Y)$ is the deck transformation group on the covering space $X$ over $Y$. We will also denote by $\gamma$ the deck transformation on $\mathcal U(\mathcal M_D)$ induced by $\gamma \in \pi_1(\mathcal M_D \setminus \Sigma)$.
	\begin{lem}
		\label{induced deck transformation}
		Let $\gamma \in \pi_1(\mathcal M_D \setminus \Sigma)$, then $p=r(\gamma)$ is the unique deck transformation on $V\otimes\mathbb C\setminus \varphi_0^{-1}(D\cup\Sigma)$ over $\mathcal M_D \setminus \Sigma$ such that
		\begin{align}
			\label{com}
			j_0\circ \gamma = p\circ j_0.
		\end{align}
		In other words, we have the following commutative diagram:
		\begin{center}
			\begin{tikzpicture}[
				>=Stealth,
				scale=0.8,
				every node/.style={align=center},
				node distance=1.5cm
				]
				
				\def\baseLength{5}
				\coordinate (A) at (-\baseLength/1.5, 0);
				\coordinate (B) at (\baseLength/1.5, 0);
				\coordinate (C) at (0, -\baseLength/2.5);
				
				\coordinate (D) at (-\baseLength/3, \baseLength/2.5);
				\coordinate (E) at (\baseLength/3, \baseLength/2.5);
				
				\tikzset{
					formula/.style={
						rectangle,
						minimum width=3cm,
						inner sep=2pt
					}
				}
				
				\node[formula] (A) at (A) {$V\otimes\mathbb C\setminus \varphi_0^{-1}(D\cup\Sigma)$};
				\node[formula] (B) at (B) {$V\otimes\mathbb C\setminus \varphi_0^{-1}(D\cup\Sigma)$};
				\node[formula] (C) at (C) {$\mathcal M_D \setminus \Sigma$.};
				\node[formula] (D) at (D) {$\mathcal U(\mathcal M_D)$};
				\node[formula] (E) at (E) {$\mathcal U(\mathcal M_D)$};
				
				\draw[->,dashed] (A) -- node[above,pos=0.5] {$\exists !~ p=r(\gamma)$} (B);
				\draw[->] (B) -- node[right,pos=0.5,xshift=2pt,yshift=-2pt] {$\varphi_0$} (C);
				\draw[->] (A) -- node[left,pos=0.5,xshift=-1pt,yshift=-2pt] {$\varphi_0$} (C);
				
				\draw[->] (-\baseLength/7, \baseLength/2.5) -- node[above,pos=0.5] {$\gamma$} (\baseLength/7, \baseLength/2.5); 
				\draw[->] (D) -- node[left,pos=0.5,xshift=-5pt,yshift=5pt] {$j_0$} (A);
				\draw[->] (E) -- node[right,pos=0.5,yshift=5pt] {$j_0$} (B);
				
			\end{tikzpicture}.
		\end{center}
		\begin{proof}
			The uniqueness follows from the fact that the covering map $j_0$ is surjective. Now we show that $p=r(\gamma)$ satisfies the commutative diagram. Since $r(\gamma) = \gamma^*$ is defined as analytic continuation on $\mathcal U(\mathcal M_D)$, $p=r(\gamma)$ satisfies \eqref{com}. We then only need to verify that $r(\gamma)$ is a deck transformation on $V\otimes\mathbb C\setminus \varphi_0^{-1}(D\cup\Sigma)$ over $\mathcal M_D \setminus \Sigma$, i.e., $\varphi_0\circ r(\gamma) = \varphi_0$. Since $\gamma$ is a deck transformation over $\mathcal M_D \setminus \Sigma$, we have
			\begin{align}
				\varphi_0\circ j_0 =\varphi_0\circ j_0\circ \gamma=\varphi_0\circ r(\gamma)\circ j_0.
			\end{align}
			Since $j_0$ is surjective, it follows that
			\begin{align*}
				\varphi_0 = \varphi_0\circ r(\gamma).
			\end{align*}
			The lemma is proved.
		\end{proof}
	\end{lem}
	
	\begin{prop}
		\label{describe}
		The monodromy group $r(\pi_1(\mathcal M_D \setminus \Sigma))$ can be represented as
		\begin{align}
			r(\pi_1(\mathcal M_D \setminus \Sigma))= \mathrm{Deck}\bigg(\frac{V\otimes\mathbb C\setminus \varphi_0^{-1}(D\cup\Sigma)}{\mathcal M_D \setminus \Sigma}\bigg).
		\end{align}
		In particular, every deck transformation on $V\otimes\mathbb C\setminus \varphi_0^{-1}(D\cup\Sigma)$ over $\mathcal M_D \setminus \Sigma$ can be extended to an isometry on $V\otimes\mathbb C$.
	\end{prop}
	\begin{proof}
		For any $\gamma \in \pi_1(\mathcal M_D \setminus \Sigma)$, we know from Lemma \ref{induced deck transformation} that $r(\gamma)$ is a deck transformation on $V\otimes\mathbb C\setminus \varphi_0^{-1}(D\cup\Sigma)$ over $\mathcal M_D \setminus \Sigma$.
		
		On the other hand, suppose $p$ is a deck transformation on $V\otimes\mathbb C\setminus \varphi_0^{-1}(D\cup\Sigma)$ over $\mathcal M_D \setminus \Sigma$. Since $\mathcal U(\mathcal M_D)$ is their universal covering space, we know that the deck transformation $p$ can be lifted to a deck transformation $\sigma: \mathcal U(\mathcal M_D) \to \mathcal U(\mathcal M_D)$.
		Therefore, we have the following commutative diagram:
		\begin{center}
			\begin{tikzpicture}[
				>=Stealth,
				scale=0.8,
				every node/.style={align=center},
				node distance=1.5cm
				]
				
				\def\baseLength{5}
				\coordinate (A) at (-\baseLength/1.5, 0);
				\coordinate (B) at (\baseLength/1.5, 0);
				\coordinate (C) at (0, -\baseLength/2.5);
				
				\coordinate (D) at (-\baseLength/3, \baseLength/2.5);
				\coordinate (E) at (\baseLength/3, \baseLength/2.5);
				
				\tikzset{
					formula/.style={
						rectangle,
						minimum width=3cm,
						inner sep=2pt
					}
				}
				
				\node[formula] (A) at (A) {$V\otimes\mathbb C\setminus \varphi_0^{-1}(D\cup\Sigma)$};
				\node[formula] (B) at (B) {$V\otimes\mathbb C\setminus \varphi_0^{-1}(D\cup\Sigma)$};
				\node[formula] (C) at (C) {$\mathcal M_D \setminus \Sigma$.};
				\node[formula] (D) at (D) {$\mathcal U(\mathcal M_D)$};
				\node[formula] (E) at (E) {$\mathcal U(\mathcal M_D)$};
				
				\draw[->] (A) -- node[above,pos=0.5] {$p$} (B);
				\draw[->] (B) -- node[right,pos=0.5,xshift=2pt,yshift=-2pt] {$\varphi_0$} (C);
				\draw[->] (A) -- node[left,pos=0.5,xshift=-1pt,yshift=-2pt] {$\varphi_0$} (C);
				
				\draw[->] (-\baseLength/7, \baseLength/2.5) -- node[above,pos=0.5] {$\sigma$} (\baseLength/7, \baseLength/2.5); 
				\draw[->] (D) -- node[left,pos=0.5,xshift=-5pt,yshift=5pt] {$j_0$} (A);
				\draw[->] (E) -- node[right,pos=0.5,yshift=5pt] {$j_0$} (B);
				
			\end{tikzpicture}
		\end{center}
		Thus by using Lemma \ref{induced deck transformation}, we have $p = r(\sigma)$. The proposition is proved.
	\end{proof}
	Now we are ready to compute the monodromy group $\mathrm{Mono}(\mathcal M_D)$.
	\begin{thm}
		\label{compute mono}
		For the generalized Frobenius manifold $\mathcal M_D = \mathcal M(R,\omega)$, its monodromy group $\mathrm{Mono}(\mathcal M_D)$ is isomorphic to a subgroup of $W_a(R)\cong W(R) \ltimes \mathbb Z^\ell $. More precisely,
		\begin{align}
			\mathrm{Mono}(\mathcal M_D)\cong\mathrm{Stab}_{W}(\omega) \ltimes \mathbb Z^\ell .
		\end{align}
		Here $\mathrm{Stab}_{W}(\omega)$ stands for the stabilizer of $\omega$ in $W(R)$. 
	\end{thm}
	We first prove three lemmas.
	\begin{lem}
		\label{abstract alg fact}
		Let $\mathbb F \hookrightarrow \mathbb E$ be a field extension, $x_1,\dots,x_m \in \mathbb E$ be some nonzero algebraic elements over $\mathbb F$. Suppose $a_1,\dots,a_m \in \mathbb E$ such that $a_1x_1^n+\cdots + a_mx_m^n \in \mathbb F$ holds for any $n \in \mathbb Z_{>0}$, then it also holds for $n=0$, i.e., we have $a_1+\cdots + a_m \in \mathbb F$.
	\end{lem}
	\begin{proof}
		Since $x_i \not= 0$, the minimal polynomial $p_i$ of $x_i$ over $\mathbb F$ has a nonzero constant term (even if $x_i \in \mathbb F$). Let $p = p_1\cdots p_m$, then $a_1p(x_1) + \cdots + a_m p(x_m) = 0$. Since $p$ has a nonzero constant term, we have $a_1+\cdots + a_m \in \mathbb F$.
	\end{proof}
	
	\begin{lem}
		\label{key lem}
		For any $\rho \in \bigoplus_{j=1}^\ell\mathbb Z\omega_j$, we have
		\begin{align}
			\label{target in key lem}
			Q_\rho(\hat{\bfx}) = \sum_{\sigma \in \mathrm{Stab}_W(\omega)} \mathrm{e}^{2\pi i(\rho, \sigma(\hat{\bfx}))} \in \mathbb C\left(z^1\big|_{\lambda=0},\dots,z^\ell\big|_{\lambda=0}\right).
		\end{align}
	\end{lem}
	\begin{proof}
		For any $\rho \in \bigoplus_{j=1}^\ell\mathbb Z\omega_j$, we  denote
		\begin{align}
			\rho =\rho_1\omega_1+\cdots + \rho_\ell\omega_\ell,\quad \rho_j \in \mathbb Z.
		\end{align}
		We first consider the case when each $\rho_j >0$. In what follows, we briefly call this case by $\rho > 0$. Note that
		\begin{equation}
			\label{zh-01-06-1}
			\sum_{\sigma \in W(R)} \mathrm{e}^{2\pi i(\rho, \sigma(\bfx) - c\omega)}= \sum_{\sigma \in W(R)} \mathrm{e}^{2\pi i(\rho, \sigma(\hat{\bfx}) + c\sigma(\omega) - c\omega)} \in \mathscr A^W[\lambda^{\varepsilon/\kappa}],
		\end{equation}
		where $\varepsilon = \mathrm{gcd}\{(\omega_r,\alpha_r)\mid r=1,\dots,\ell\}$. Indeed, \eqref{zh-01-06-1} is $W_a(R)$-invariant, and equation \eqref{sigma(weight)} implies that
		\[\sigma(\omega) - \omega= -\sum_{r=1}^{\ell} n_r\alpha_r\] for some $n_r \in \mathbb Z_{\ge 0}$, which indicates that $(\rho, \sigma(\omega) - \omega) \in \varepsilon\mathbb Z_{\le 0}$. Since $\mathscr A^W = \mathbb C[z^1,\dots,z^\ell;\lambda]$, by putting $\lambda = 0$ in \eqref{zh-01-06-1}, we obtain
		\[
		\sum_{\substack{\sigma \in W(R)\\(\rho, \sigma(\omega) - \omega) = 0}} \mathrm{e}^{2\pi i(\rho, \sigma(\hat\bfx))} \in \mathbb C\left[z^1\big|_{\lambda=0},\dots,z^\ell\big|_{\lambda=0}\right].
		\]
		Therefore, we only need to show that for any $\sigma \in W(R)$, $(\rho, \sigma(\omega) - \omega) = 0$ implies $\sigma(\omega) = \omega$. In fact, if $(\rho, \sigma(\omega) - \omega) = 0$, we have
		\[
		(\rho, \sigma(\omega) - \omega) = (\rho, -\sum_{r=1}^{\ell}n_r\alpha_r) = \sum_{r=1}^{\ell} \rho_r n_r(\omega_r,\alpha_r) = 0.
		\]
		Since $\rho_r > 0$, $n_r \ge 0$ and $(\omega_r,\alpha_r) > 0$, $r=1,\dots,\ell$, we have $n_r = 0$, $r=1,\dots,\ell$, hence $\sigma(\omega) = \omega$. Therefore, $Q_\rho \in \mathbb C(z^1\big|_{\lambda=0},\dots,z^\ell\big|_{\lambda=0})$ for any $\rho > 0$.
		
		In particular, we know that for any $\rho > 0$ and $\sigma \in \mathrm{Stab}_W(\omega)$, $\mathrm{e}^{2\pi i(\rho, \sigma(\hat{\bfx}))}$ is algebraic over $\mathbb C(z^1\big|_{\lambda=0},\dots,z^\ell\big|_{\lambda=0})$. In fact, since for any $n \in\mathbb Z_{>0}$, we have $n\rho > 0$, hence
		\[
		Q_{n\rho}(\hat{\bfx}) = \sum_{\sigma \in \mathrm{Stab}_W(\omega)}\mathrm{e}^{2\pi i(n\rho, \sigma(\hat{\bfx}))} \in \mathbb C\left(z^1\big|_{\lambda=0},\dots,z^\ell\big|_{\lambda=0}\right).
		\]
		In other words, all the Newton sums of $\{\mathrm{e}^{2\pi i (\rho, \sigma(\hat{\bfx}))}\mid \sigma \in \mathrm{Stab}_W(\omega)\}$  belong to the field $\mathbb C\left(z^1\big|_{\lambda=0},\dots,z^\ell\big|_{\lambda=0}\right)$, hence so are their symmetric polynomials
		\[
		r_n = \sum_{\substack{I\subseteq \mathrm{Stab}_W(\omega)\\ \# I = n}}~\prod_{\sigma \in I}\mathrm{e}^{2\pi i (\rho, \sigma(\hat{\bfx}))} \in \mathbb C\left(z^1\big|_{\lambda=0},\dots,z^\ell\big|_{\lambda=0}\right) .
		\]
		This implies that they are algebraic over $\mathbb C\left(z^1\big|_{\lambda=0},\dots,z^\ell\big|_{\lambda=0}\right)$.
		
		Now let us consider the general case. For any $\rho \in \bigoplus_{j=1}^\ell\mathbb Z\omega_j$, we take $\xi \in \bigoplus_{j=1}^\ell\mathbb Z\omega_j$ and $\xi>0$ such that $\rho + \xi > 0$. Then for any $n \in \mathbb Z_{>0}$, we have $\rho + n\xi > 0$, hence
		\[
		Q_{\rho + n\xi}(\hat{\bfx}) = \sum_{\sigma \in \mathrm{Stab}_W(\omega)} \mathrm{e}^{2\pi i(\rho + n\xi, \sigma(\hat{\bfx}))} \in \mathbb C\left(z^1\big|_{\lambda=0},\dots,z^\ell\big|_{\lambda=0}\right).
		\]
		Since $\mathrm{e}^{2\pi i(\xi, \sigma(\hat{\bfx}))}$ is algebraic over $\mathbb C\left(z^1\big|_{\lambda=0},\dots,z^\ell\big|_{\lambda=0}\right)$, it follows from Lemma \ref{abstract alg fact} that
		\[
		Q_{\rho}(\hat{\bfx}) = \sum_{\sigma \in \mathrm{Stab}_W(\omega)} \mathrm{e}^{2\pi i(\rho, \sigma(\hat{\bfx}))} \in \mathbb C\left(z^1\big|_{\lambda=0},\dots,z^\ell\big|_{\lambda=0}\right).
		\]
		Therefore, we have proved the lemma.
	\end{proof}
	\begin{lem}
		\label{Galois ext}
		Let $R$ be a root system, $\omega$ be a weight and $\mathscr A$ be the corresponding $\lambda$-Fourier polynomial ring. Denote by $\mathscr K = \mathrm{Frac}(\mathscr A)$ the fraction field of $\mathscr A$, then $W(R)$ acts on $\mathscr K$, and
		\[
		\mathrm{Inv}_{\mathscr K}\left(\mathrm{Stab}_W(\omega)\right) =\mathbb C\left(z^1\big|_{\lambda=0},\dots,z^\ell\big|_{\lambda=0};\lambda\right).
		\]
		Here $\mathrm{Inv}_{\mathscr K}(\mathrm{Stab}_W(\omega))$ refers to the fixed subfield of $\mathscr K$ with respect to the action of $\mathrm{Stab}_W(\omega)$. In particular, the field extension $\mathbb C(z^1\big|_{\lambda=0},\dots,z^\ell\big|_{\lambda=0};\lambda) \hookrightarrow \mathscr K$ is Galois.
	\end{lem}
	\begin{proof}
		From \eqref{phi 0} and \eqref{proper gen defn} we know that
		\begin{align}
			z^j\bigr|_{\lambda=0}= \frac 1{N_j}\sum_{\sigma \in H_j} \mathrm{e}^{2\pi i(\sigma^{-1}(\omega_j), \hat{\mathbf{x}})},
		\end{align}
		where 
		\begin{align}
			\label{define Hj}
			H_j =\{\sigma \in W(R)\mid (\omega_j, \omega - \sigma(\omega)) = 0\}.
		\end{align}
		Since for any $\phi \in \mathrm{Stab}_W(\omega)$, $\sigma \mapsto \sigma\phi$ induces a  permutation on $H_j$, we have 
		\[z^j\big|_{\lambda=0} \in \mathrm{Inv}_{\mathscr K}(\mathrm{Stab}_W(\omega)).\]
		Thus we only need to prove that 
		\[\mathrm{Inv}_{\mathscr K}(\mathrm{Stab}_W(\omega))\subseteq\mathbb C\left(z^1\big|_{\lambda=0},\dots,z^\ell\big|_{\lambda=0};\lambda\right).\]

		We first show that 
		\begin{align}
			\label{zh-01-06-2}
			\mathbb C[\mathrm{e}^{\pm2\pi ix^j}\mid j=1,\dots,\ell]^{\mathrm{Stab}_W(\omega)}\subseteq\mathbb C\left(z^1\big|_{\lambda=0},\dots,z^\ell\big|_{\lambda=0}\right).
		\end{align}
		In fact, for any $q \in\mathbb C[\mathrm{e}^{\pm2\pi ix^j}\mid j=1,\dots,\ell]^{\mathrm{Stab}_W(\omega)}$, we denote \[q= \sum_{\rho\in\bigoplus_{j=1}^\ell\mathbb Z\omega_j} A_\rho \mathrm{e}^{2\pi i(\rho, \hat{\bfx})},\] where only finitely many $A_\rho$ is not zero. Since $q$ is $\mathrm{Stab}_W(\omega)$-invariant, we have
		\[
		q =  \frac1{\#\mathrm{Stab}_W(\omega)}\sum_{\rho\in\bigoplus_{j=1}^\ell\mathbb Z\omega_j}A_\rho Q_\rho(\hat{\bfx}).
		\]
		By using Lemma \ref{key lem}, we arrive at \eqref{zh-01-06-2}.
		
		We next show that 
		\begin{align}
			\label{zh-01-06-3}
			\mathscr A^{\mathrm{Stab}_W(\omega)} \subseteq \mathbb C(z^1\big|_{\lambda=0},\dots,z^\ell\big|_{\lambda=0};\lambda).
		\end{align}
		Recall that $W(R)$ does not act on $\mathscr A$, since its action may introduce terms with negative power of $\lambda$. However, its subgroup $\mathrm{Stab}_W(\omega)$ acts on $\mathscr A$. In fact, for each $q \in \mathscr A$, we can write it as
		\[
		q(\hat{\bfx};\lambda) = \sum_r \lambda^r q_r(\hat{\bfx}),\quad q_r \in \mathbb C[\mathrm{e}^{\pm2\pi ix^j}\mid j=1,\dots,\ell].
		\]
		Then an element $\phi \in \mathrm{Stab}_W(\omega)$ acts on $q$ by
		\begin{align}
			\label{stab acts on A}
			(\phi^* q)(\hat{\bfx};\lambda)= \sum_r \lambda^r q_r(\phi(\bfx) - c\omega) = \sum_r \lambda^r q_r(\phi(\hat{\bfx})),
		\end{align}
		where we used the fact that $\phi(\omega) = \omega$. Therefore, $\mathrm{Stab}_W(\omega)$ indeed acts on $\mathscr A$. Note that \eqref{stab acts on A} also implies that if $q \in \mathscr A^{\mathrm{Stab}_W(\omega)}$, then each $q_r$ is also $\mathrm{Stab}_W(\omega)$-invariant, hence it follows from \eqref{zh-01-06-2} that \[q_r \in \mathbb C(z^1\big|_{\lambda=0},\dots,z^\ell\big|_{\lambda=0}),\] 
		and $q \in\mathbb C(z^1\big|_{\lambda=0},\dots,z^\ell\big|_{\lambda=0}; \lambda)$.
		
		Finally we show that for any $p/q \in \mathrm{Inv}_{\mathscr K}(\mathrm{Stab}_W(\omega))$, where $p,q \in \mathscr A$, we have \[\frac pq \in\mathbb C(z^1\big|_{\lambda=0},\dots,z^\ell\big|_{\lambda=0}; \lambda).\] In fact, we have
		\[
		\frac pq = \frac{p\cdot\prod_{\phi \in \mathrm{Stab}_W(\omega) \setminus\{\mathrm{id}\}}\phi^* q}{\prod_{\phi \in \mathrm{Stab}_W(\omega) }\phi^* q} =: \frac{\tilde p}{\tilde q}.
		\]
		Here $\tilde q \in \mathscr A^{\mathrm{Stab}_{W}(\omega)}$, hence $\tilde p\in \mathscr A^{\mathrm{Stab}_{W}(\omega)}$. Therefore, by using \eqref{zh-01-06-3} we have \[\frac pq \in\mathbb C(z^1\big|_{\lambda=0},\dots,z^\ell\big|_{\lambda=0}; \lambda).\]
		The lemma is proved.
	\end{proof}
	\begin{proof}[Proof of Theorem \ref{compute mono}]
		Due to Proposition \ref{describe}, we only need to determine all the isometric actions on $V\otimes\mathbb C\setminus \varphi_0^{-1}(D\cup\Sigma)$ that preserve $\mathcal M_D \setminus \Sigma$, i.e. actions on coordinates $x^i$ that preserves $z^j\bigr|_{\lambda=0}$. Recall that
		\begin{align}
			z^j\bigr|_{\lambda=0}=\frac 1{N_j}\sum_{\sigma \in H_j} \mathrm{e}^{2\pi i(\sigma^{-1}(\omega_j), \hat{\mathbf{x}})}.
		\end{align}
		For any given element $\tilde\phi \in \mathrm{Iso}(V\otimes\mathbb C )$, we decompose it into linear and the affine parts
		as follows:
		\begin{align}
			\tilde\phi(\hat{\mathbf{x}})= \phi(\hat{\mathbf{x}}) + v^1\alpha_1^\vee + \dots + v^\ell \alpha_\ell ^\vee.
		\end{align}
		Then $\tilde{\phi}$ acts on the the Fourier polynomials $z^j$ by
		\begin{align}
			(\tilde\phi^* z^j)\big|_{\lambda=0}&=\frac 1{N_j}\sum_{\sigma \in H_j} \mathrm{e}^{2\pi i(\sigma^{-1}(\omega_j), \tilde\phi(\hat{\mathbf{x}}))}\\
			&=\frac 1{N_j}\sum_{\sigma \in H_j} \mathrm{e}^{2\pi i(\sigma^{-1}(\omega_j), \phi(\hat{\mathbf{x}}) + v^1\alpha_1^\vee + \dots + v^\ell \alpha_\ell ^\vee)},
		\end{align}
		where $H_j$ is defined in \eqref{define Hj}. Therefore, the action preserves all $z^j\bigr|_{\lambda=0}$, $j = 1,\dots,\ell$ if and only if both of the following conditions hold:
		\begin{enumerate}
			\item\label{item1} For each $j=1,\dots,\ell$, $\phi$ induces a permutation on $\{\sigma^{-1}(\omega_j)\mid \sigma\in H_j\}$ by right action $\sigma \mapsto \sigma\phi$.
			\item\label{item2} for each $j=1,\dots,\ell$, $\sigma \in H_j$,
			\[(\sigma^{-1}(\omega_j), v^1\alpha_1^\vee + \cdots + v^\ell \alpha_\ell ^\vee)\]
			is an integer.
		\end{enumerate}
		We now proceed to respectively show that condition \eqref{item1} is equivalent to $\phi \in \mathrm{Stab}_W(\omega)$, and condition \eqref{item2} is equivalent to $v^1,\dots,v^\ell  \in \mathbb Z$.
		
		Suppose $\phi \in \mathrm{Stab}_W(\omega)$, then $\sigma \mapsto \sigma\phi$ induces a permutation on $H_j$, hence on $\{\sigma^{-1}(\omega_j)\mid \sigma\in H_j\}$. On the other hand, if condition \eqref{item1} holds, we are to show that $\phi \in \mathrm{Stab}_W(\omega)$. We consider the canonical action of $\phi$ on $\mathscr K$ given by $\phi^* q(\bfx, \lambda) = q(\phi(\bfx), \lambda)$. According to Lemma \ref{Galois ext}, we know that
		\[
		\mathrm{Gal}\left(\mathscr K/\mathbb C\left(z^1|_{\lambda=0},\dots,z^\ell|_{\lambda=0};\lambda\right)\right) = \mathrm{Stab}_W(\omega).
		\]
		Condition \eqref{item1} implies that $\phi^*$ fixes $\mathbb C\left(z^1|_{\lambda=0},\dots,z^\ell|_{\lambda=0};\lambda\right)$, hence we only need to show that $\phi^*(\mathscr K) \subseteq \mathscr K$. Note that $\mathscr K$ is generated by $\mathrm{e}^{2\pi i x_1},\dots,\mathrm{e}^{2\pi i x_\ell}$ and $\lambda$ as $\mathbb C$-algebra, and we have $\phi^*\lambda = \lambda$, it suffices to show that $\phi^*\mathrm{e}^{2\pi i x_1} \in \mathscr K$. In fact, by considering $\mathrm{id} \in H_j$, $j=1,\dots,\ell$, we know from condition \eqref{item1} that 
		\[
		\phi^*(\omega_j) = \sigma_j(\omega_j),\quad \exists~\sigma_j \in H_j.
		\]
		Then in particular, for each $j=1,\dots,\ell$, we have
		\[
		\phi^*\mathrm{e}^{2\pi ix^j} =\phi^*\mathrm{e}^{2\pi i(\omega_j, \bfx - c\omega)} =  \mathrm e^{2\pi i(\omega_j,\phi(\bfx) - c\omega)} = \mathrm e^{2\pi i(\sigma_j^{-1}(\omega_j), \hat{\bfx})}.
		\]
		Here we used $(\sigma_j^{-1}(\omega_j) - \omega_j, \omega) = 0$. From \eqref{too long to find a name} we know that for each $r=1,\dots,\ell$, $(\sigma_j^{-1}(\omega_j),\alpha_r^\vee)$ is an integer, so
		\[
		\phi^*\mathrm{e}^{2\pi ix^j} = \mathrm e^{2\pi i\sum_r x^r(\sigma_j^{-1}(\omega_j), \alpha_r^\vee)} \in \mathscr A,\quad j=1,\dots,\ell.
		\]
		It follows immediately that $\phi^*(\mathscr K) \subseteq \mathscr K$. Hence $\phi \in \mathrm{Stab}_W(\omega)$.
		
		Suppose $v^1,\dots,v^\ell \in \mathbb Z$, then for any $\sigma \in H_j$, we know from \eqref{too long to find a name} that for each $i=1,\dots,\ell$ $(\sigma^{-1}(\omega_j),\alpha_i^\vee)$ is an integer, so is their $\mathbb Z$-combination $(\sigma^{-1}(\omega_j), v^1\alpha_1^\vee + \cdots + v^\ell \alpha_\ell ^\vee)$. On the other hand, if condition \eqref{item2} holds, then for any $j=1,\dots,\ell$, by considering $\mathrm{id} \in H_j$, we obtain that $v^j \in \mathbb Z$.
		
		Therefore, we have $\mathrm{Mono}(\mathcal M_D) = \mathrm{Stab}_W(\omega)\ltimes \mathbb Z^\ell$.
	\end{proof}
	By using Proposition \ref{grp embed}, we have the following corollary.
	\begin{cor}
		\label{mono le stab}
		The monodromy group of $\widehat{\mathcal{M}}_D$ is isomorphic to a subgroup of $\mathrm{Stab}_{W}(\omega)\ltimes\mathbb Z^\ell$.
	\end{cor}
	
	In general, the monodromy group of $\widehat{\mathcal{M}}_D$ is not isomorphic to that of $\mathcal M_D$, and we will give a counterexample in Sect.\,\ref{examples}. However, in some cases we have $\widehat{\mathcal{M}}_D=\mathcal M_D$, so they have isomorphic monodromy groups. We will also see such examples in Sect.\,\ref{examples}.
	\begin{rem}
		The group $\mathrm{Stab}_{W}(\omega)$ is a parabolic subgroup of $W(R)$ \cite{bourbaki}. Explicitly, we have
		\begin{align}
			\mathrm{Stab}_{W}(\omega) = W_{S^c} = \langle\sigma_i\mid i \in S^c\rangle,
		\end{align}
		where $S$ is given by \eqref{zzh-4} and $S^c$ is its complement in $\{1,\cdots,\ell\}$.
	\end{rem}
	
	\section{Examples}
	\label{examples}
	In this section, we give some typical examples for our construction. We assume that $e_1,\dots, e_n$ give an orthonormal basis of the Euclidean space $\mathbb{R}^n$ for $n\in\mathbb{N}$ with inner product $(\cdot\,,\cdot)$, and the $\ell$-dimensional space $V$ will be $\mathbb{R}^\ell$ or a subspace of $\mathbb{R}^n$ for a certain $n>\ell$ endowed with the induced inner product $(\cdot\,,\cdot)$.
	For convenience of the presentation, we will use $x_i, y_i, z_i, t_\al$ to denote $x^i, y^i, z^i$ and $t^\al$ respectively. We will also use $E$ to denote the Euler vector $\tilde{E}$ of the generalized Frobenius manifolds given in \eqref{zh-05-28a}.
	
	We start with the simplest example.
	\begin{ex}[$(R,\omega)=(A_1, \omega_1)$] Let $R$ be the root system of type $A_1$, and $V$ be the hyperplane of $\mathbb{R}^2$ spanned by 		
		\[\al_1=e_1-e_2.\]
		We take $\omega$ to be the fundamental weight $\omega_1=\frac12\av_1=\frac12\al_1$.
		Let us introduce the coordinate $x_1$ on $V$ by
		\[\mathbf{x}=c\omega_1+x_1 \av_1,\quad \mathbf{x}\in V.\]
		In this case $\kappa = (\omega_1,\alpha_1) = 1$, hence $\lambda = \mathrm{e}^{-2\pi ic}$.	The Weyl group $W(R)$ acting on $V$ is generated by 
		\[\sigma_1(\mathbf{x})=c\omega_1+(-c-x_1)\av_1,\quad x_1\mapsto -c-x_1.\]
		The invariant Fourier polynomial has the form
		\[Y_1 = \mathrm{e}^{2\pi i \left(\omega_1,\bfx\right)}+\mathrm{e}^{2\pi i \left(\omega_1,\sigma_1(\bfx)\right)} = \mathrm{e}^{2\pi i \left(x_1+\frac12 c\right)}+\mathrm{e}^{-2\pi i \left(x_1+\frac12 c\right)},\]
		hence
		\[y_1=\mathrm{e}^{-\pi i c}Y_1=\mathrm{e}^{2\pi i x_1}+\lambda \mathrm{e}^{-2\pi i x_1}.\]
		The metric $(\cdot\,,\cdot)$ on $V$ is given by $(\nd x_1,\nd x_1)=(\av_1,\av_1)^{-1}=\frac12$, hence
		\[g^{11}_\lambda=\frac1{4\pi^2}(\nd y_1,\nd y_1)= \frac1{4\pi^2}\frac{\p y_1}{\p x_1}(\nd x_1,\nd x_1)\frac{\p y_1}{\p x_1}=-\frac12 y_1^2+2\lambda.\]
		Therefore, $y_1$ is a pencil generator, and we obtain a flat pencil
		\[
		\eta^{11}=2,\quad
		g^{11}=-\frac12 y_1^2.
		\]
		The metric $\eta$ has the flat coordinate $t_1=y_1/\sqrt{2}$. In this new coordinate the flat pencil has the form
		\[\eta^{11}=1,\quad g^{11}=-\frac12 t_1^2.\]
		Thus we obtain a generalized Frobenius manifold $\mathcal M(A_1, \omega_1)$ with potential 
		\[F=-\frac1{24} t_1^4.\]
		The Euler vector field and the unit vector field are given by
		\[E=\frac12 t_1\frac{\p}{\p t_1},\quad
		e=-\frac1{t_1}\frac{\p}{\p t_1}.\]
		Moreover, the flat coordinate of $g$ is given by
		\[
		x^1\big|_{\lambda=0} = \frac 1{2\pi i}\log y_1 = \frac 1{2\pi i}\log (\sqrt 2\,t_1).
		\]
		It follows that only translations $x_1 \to x_1 + n$, $n \in \mathbb Z$ preserve $y_1$, and equivalently $t_1$, hence the monodromy groups of $\mathcal M_D$ and $\widehat{\mathcal M}_D$ are given by
		\begin{align*}
			\mathrm{Mono}(\mathcal M_D) = \mathrm{Mono}(\widehat{\mathcal M}_D) =\mathbb Z.
		\end{align*}
		Indeed, $\mathrm{Stab}_{W}(\omega_1)$ is trivial and we have $\mathrm{Mono}(\mathcal M(A_1,\omega_1)) = \mathrm{Stab}_{W}(\omega_1)\ltimes\mathbb Z$.
	\end{ex}
	
	We have applied the above-mentioned construction of generalized Frobenius manifolds to different types of affine Weyl groups with appropriately chosen weights.
	In what follows we will give some representative examples, which are chosen based on the following points:
	\begin{enumerate}
		\item[1.] In Example \ref{A2,omega2}, the basic generators $y_1,y_2$ are pencil generators, and
		\begin{align*}
			\mathrm{Mono}(\widehat{\mathcal M}_D) =\mathrm{Mono}(\mathcal M_D) ~\cong~\mathrm{Stab}_{W}(\omega) \ltimes \mathbb Z^\ell .
		\end{align*}
		\item[2.] In Example \ref{C3,omega3}, the basic generators $y_1,y_2,y_3$ are also pencil generators, while
		$\mathrm{Mono}(\widehat{\mathcal M}_D)$ is a proper subgroup of $\mathrm{Stab}_{W}(\omega) \ltimes \mathbb Z^\ell $.
		\item[3.] In Example \ref{Be,omega1}, the basic generators $y_1,y_2,y_3$ are not pencil generators, and we need to construct pencil generators by simple translations on $y_1, y_2$ and $y_3$, i.e.,
		\begin{align*}
			z_1= y_1 + a_1\lambda,\quad z_2=y_2 + a_2\lambda,\quad z_3 = y_3.
		\end{align*}
		Here $a_1, a_2$ are some constants, and there are three possible values for them.
		\item[4.] In Example \ref{G2, omega2}, $y_1,y_2$ are not pencil generators. The pencil generators are constructed in a more complicated way as follows:
		\[
		z_1=y_1-6\lambda,\quad
		z_2=y_2-3\lambda y_1+12\lambda^2.
		\]
		\item[5.] In Example \ref{A3,omega1+omega3}, we select the weight $\omega=\omega_1 + \omega_3$ instead of a fundamental weight. This complicates the construction of pencil generators.
		\item[6.] In Example \ref{(D4,omega2)}, we present a case with much more complicated construction of pencil generators. Compared with the former ones, the resulting prepotential in this case is not a Laurant polynomial in flat coordinates.
	\end{enumerate}
	More examples will be presented in our subsequent publications \cite{JLTZ-2}.
	\begin{ex}[$(R,\omega)=(A_2, \omega_2)$]\label{A2,omega2}
		We take the root system $R$ of type $A_2$. Let $V$ be the hyperplane of $\mathbb{R}^3$ spanned by 
		\[\al_1=e_1-e_2,\quad \al_2=e_2-e_3.\]
		We have the coroots and fundamental weights
		\[\av_1=\al_1,\quad \av_2=\al_2,\quad
		\omega_1=\frac23\av_1+\frac13\av_2,\quad \omega_2=\frac13\av_1+\frac23\av_2.\]
		Choose $\omega=\omega_2$, and
		introduce the coordinates $x_1, x_2$ of $V$ by
		\[\mathbf{x}=c\omega_2+x_1 \av_1+x_2 \av_2,\quad \mathbf{x}\in V.\]
		We have the invariant Fourier polynomials
		\begin{align*}
			Y_1&=\mathrm{e}^{2\pi i \left(x_1+\frac13 c\right)}+\mathrm{e}^{2\pi i \left(x_2-x_1+\frac13 c\right)}+ \mathrm{e}^{-2\pi i \left(x_2+\frac23 c\right)},\\
			Y_2&=\mathrm{e}^{2\pi i \left(x_2+\frac23c\right)}+ \mathrm{e}^{-2\pi i \left(x_1+\frac13c\right)}+\mathrm{e}^{2\pi i \left(x_1-x_2-\frac13c\right)}.
		\end{align*}
		Since $\kappa = 1$, we have $\lambda = \mathrm{e}^{-2\pi i c}$ and the invariant $\lambda$-Fourier polynomials
		\begin{align*}
			y_1&=\lambda^{\frac13} Y_1=\mathrm{e}^{2\pi i x_1}+\mathrm{e}^{2\pi i (x_2-x_1)}+\lambda \mathrm{e}^{-2\pi i x_2},\\
			y_2&=\lambda^{\frac23} Y_2=\mathrm{e}^{2\pi i x_2}+\lambda \mathrm{e}^{-2\pi i x_1}+\lambda \mathrm{e}^{2\pi i (x_1-x_2)}.
		\end{align*}
		The metric on $V$ is given by
		\[\left((\av_i, \av_j)\right)=\begin{pmatrix}2 &-1\\ -1 &2\end{pmatrix},\quad
		(\nd x_i,\nd x_j)={\begin{pmatrix}2 &-1\\ -1 &2\end{pmatrix}}^{-1}=\frac13\begin{pmatrix}2 &1\\ 1 &2\end{pmatrix}.\]
		In the coordinates $y_1, y_2$ we have
		\[\left(g_\lambda^{ij}\right)=\frac1{4\pi^2}\left((\nd y_i, \nd y_j)\right)=\frac1{4\pi^2}\sum_{k,r=1}^2 \frac{\p y_i}{\p x_k}\frac{\p y_j}{\p x_r}(\nd x_k,\nd x_r)=\begin{pmatrix}-\frac23 y_1^2+2 y_2 & 3\lambda -\frac13 y_1 y_2\\[3pt] 3\lambda -\frac13 y_1 y_2& 2 \lambda y_1-\frac23 y_2^2\end{pmatrix}.\]
		Then $\{y_1, y_2\}$ is a set of pencil generators, and the corresponding flat pencil is
		given by
		\[
		\left(\eta^{ij}\right)=\begin{pmatrix}0 &3\\ 3&2 y_1\end{pmatrix},\quad
		\left(g^{ij}\right)=\begin{pmatrix}-\frac23 y_1^2+2 y_2 & -\frac13 y_1 y_2\\[3pt] -\frac13 y_1 y_2& -\frac23 y_2^2\end{pmatrix}.
		\]
		
		The metric $\eta$ has flat coordinates
		\[t_1=\frac 13 y_1,\quad t_2=y_2-\frac16 y_1^2,\quad \]
		in which the metrics $\eta$ and $g$ have the form
		\[
		\left(\eta^{ij}\right)=\begin{pmatrix}0 &1\\ 1&0\end{pmatrix},\quad
		\left(g^{ij}\right)=\begin{pmatrix} \frac19\left(2 t_2-3 t_1^2\right)& -t_1 t_2+\frac 12 t_1^3\\[4pt]-t_1 t_2+\frac 12 t_1^3&-\frac16\left(2t_2-3 t_1^2\right)^2\end{pmatrix}.
		\]
		Thus we obtain a generalized Frobenius manifold $\mathcal M(A_2,\omega_2)$ with potential
		\begin{align}
			F=\frac1{18} t_2^3-\frac14 t_1^2 t_2^2+\frac18  t_1^4t_2-\frac3{80} t_1^6,
		\end{align}
		the Euler vector field 
		\[E=\frac13 t_1\frac{\p}{\p t_1}+\frac23 t_2\frac{\p}{\p t_2},\]
		and the unit vector field
		\[e=-\frac{2}{2 t_2+3 t_1^2}\frac{\p}{\p t_1}-\frac{6 t_1}{2 t_2+3 t_1^2}\frac{\p}{\p t_2}.\]
		We also have $e = -\mathrm{grad}_\eta \log y_2$. Note that in this example, not only are $y_1, y_2$ polynomials of $t_1$ and $t_2$, but $t_1$ and $t_2$ are also polynomials of $y_1$ and $y_2$. This implies that the monodromy between $t_1$, $t_2$ and $y_1$, $y_2$ are trivial, i.e., $\widehat{\mathcal M}_D = \mathcal M_D$. Therefore,
		\begin{align*}
			\mathrm{Mono}(\widehat{\mathcal M}_D) = \mathrm{Mono}(\mathcal M_D)=\mathrm{Stab}_{W}(\omega_2)\ltimes \mathbb Z^2 =\langle\sigma_1\rangle \ltimes \mathbb Z^2\cong \mathbb Z_2\ltimes\mathbb Z^2.
		\end{align*}
		More explicitly, $\mathrm{Mono}(\mathcal M(A_2, \omega_2))$ is generated by
		\begin{align*}
			\sigma_1\colon &x_1 \mapsto x_2 - x_1,\quad x_2 \mapsto x_2,\\
			T_1\colon&x_1 \mapsto x_1 + 1,\quad x_2 \mapsto x_2,\\
			T_2\colon&x_1 \mapsto x_1,\quad x_2 \mapsto x_2+1.
		\end{align*}
	\end{ex}
	
	\begin{ex}[$(R,\omega)=(C_3, \omega_3)$]
		\label{C3,omega3}
		Let $R$ be the root system of type $C_3$, and $V=\mathbb{R}^3$. We take the simple roots
		\[\al_1=e_1-e_2, \quad \al_2=e_2-e_3,\quad \al_3=2e_3.\]
		Then we have the coroots
		\[\av_1=e_1-e_2,\quad \av_2=e_2-e_3,\quad \av_3=e_3,\]
		and the fundamental weights
		\[\omega_1=\av_1+\av_2+\av_3,\quad \omega_2=\av_1+2\av_2+2\av_3,\quad \omega_3=\av_1+2\av_2+3\av_3.\]
		Let us take $\omega=\omega_3$, and introduce coordinates $x_1, x_2, x_3$ of $V$ by
		\[\mathbf{x}= c\omega_3+x_1 \av_1+x_2 \av_2+x_3\av_3.\]
		Since $\kappa = 2$, we have $\lambda=\mathrm{e}^{-4\pi i c}$.
		The invariant $\lambda$-Fourier polynomials are given by
		\begin{align*}
			y_1=\lambda^{\frac12}(\xi_1+\xi_2+\xi_3),\quad
			y_2=\lambda(\xi_1\xi_2+\xi_1\xi_3+\xi_2\xi_3),\quad
			y_3=\lambda^{\frac32}\xi_1\xi_2\xi_3,
		\end{align*}
		where 
		\begin{align*}
			\xi_1&=\lambda^{-\frac12}\mathrm{e}^{2\pi i x_1}+\lambda^{\frac12} \mathrm{e}^{-2\pi i x_1},\quad
			\xi_2=\lambda^{-\frac12} \mathrm{e}^{2\pi i (x_2-x_1)}+\lambda^{\frac12} \mathrm{e}^{-2\pi i(x_2-x_1)},\\
			\xi_3&=\lambda^{-\frac12} \mathrm{e}^{2\pi i (x_3-x_2)}+\lambda^{\frac12} \mathrm{e}^{-2\pi i (x_3-x_2)}.
		\end{align*}
		The metric on $V$ has the form 
		$$ 
		\left((\nd x_i,\nd x_j)\right)= \begin{pmatrix}2&-1&0\\-1&2&-1\\ 0&-1&1\end{pmatrix}^{-1}=\begin{pmatrix}
			1&1&1\\ 1&2&2\\
			1&2&3
		\end{pmatrix}.
		$$
		In the coordinates $y_1, y_2, y_3$ we have
		\begin{align*}
			\left(g_\lambda^{ij}\right)&=\frac1{4\pi^2}\left((\nd y_i, \nd y_j)\right)=\frac1{4\pi^2}\sum_{k,r=1}^3 \frac{\p y_i}{\p x_k}\frac{\p y_j}{\p x_r}(\nd x_k,\nd x_r)\\
			&=\begin{pmatrix}
				-y_1^2+2 y_2+12 \lambda& -y_1 y_2+3 y_3+8 \lambda y_1 &-y_1 y_3+4 \lambda y_2\\
				-y_1 y_2+3 y_3+8 \lambda y_1&-2 y_2^2+2 y_1 y_3-8\lambda y_2+8 \lambda y_1^2 &-2y_2 y_3+4\lambda y_1 y_2-12 \lambda y_3\\
				-y_1 y_3+4 \lambda y_2&-2y_2 y_3+4\lambda y_1 y_2-12 \lambda y_3&-3 y_3^2-8 \lambda y_1 y_3+4 \lambda y_2^2
			\end{pmatrix}.
		\end{align*}
		It follows that $y_1,y_2,y_3$ are pencil generators, and the corresponding flat pencil is given by
		\begin{align*}
			\left(\eta^{ij}\right)&=\begin{pmatrix}
				12 & 8 y_1 &4 y_2\\ 8 y_1 &-8 y_2+8 y_1^2 & 4 y_1 y_2-12 y_3\\
				4 y_2 & 4 y_1 y_2-12 y_3 & -8 y_1 y_3+4 y_2^2
			\end{pmatrix},\\
			\left(g^{ij}\right)&=\begin{pmatrix}
				-y_1^2+2 y_2& -y_1 y_2+3 y_3 &-y_1 y_3\\
				-y_1 y_2+3 y_3&-2 y_2^2+2 y_1 y_3 &-2y_2 y_3\\
				-y_1 y_3&-2y_2 y_3&-3 y_3^2
			\end{pmatrix}.
		\end{align*}
		The flat coordinates $t_1, t_2, t_3$ of $\eta$ are given by the polynomials
		\begin{align*}
			y_1=2\sqrt{3}\, t_2,\quad
			y_2=4 t_2^2-4 t_1 t_3,\quad 
			y_3=\frac{8}{3\sqrt{3}}\left(t_1^3+t_2^3+t_3^3-3 t_1 t_2 t_3\right).
		\end{align*}
		In these flat coordinates the metric $\eta$ has the form
		\[\left(\eta^{\alpha\beta}\right)=\begin{pmatrix}0&0&1\\ 0&1&0\\ 1&0&0\end{pmatrix}.\]
		Thus we obtain the a generalized Frobenius manifold $\mathcal M(C_3,\omega_3)$ with potential
		\[F=-\frac1{9} t_1^3 t_2-\frac1{36} t_2^4-\frac13 t_1 t_2^2 t_3-\frac1{6} t_1^2 t_3^2-\frac1{9} t_2 t_3^3.\]
		The Euler vector field and the unit vector field are given by
		\begin{align*}
			E=&\,\frac 12t_1\frac{\p}{\p t_1}+\frac 12t_2\frac{\p}{\p t_2}+\frac 12t_3\frac{\p}{\p t_3},\\
			e=&\,\frac{3}{2\left(t_1^3+t_2^3+t_3^3-3 t_1 t_2 t_3\right)}\\&\, \times\left((t_1 t_2-t_3^2)\frac{\p}{\p t_1}+(t_1 t_3-t_2^2)\frac{\p}{\p t_2}+(t_2 t_3-t_1^2)\frac{\p}{\p t_3}\right).
		\end{align*}
		We also have $e = -\frac 12\mathrm{grad}_\eta \log y_3$. 
		
		Note that, if we take the flat coordinates $v_1, v_2, v_3$ of $\eta$ via a linear transformation
		\begin{align*}
			t_1 &= \frac 1{\sqrt 3}(v_1+\zeta v_2 + \zeta^2 v_3),\\
			t_2 &= \frac 1{\sqrt 3}(v_1 + v_2 + v_3),\\
			t_3 &= \frac 1{\sqrt 3}(v_1+\zeta^2 v_2 + \zeta v_3),
		\end{align*}
		where $\zeta= \frac{-1+\sqrt 3i}{2}$, then in these coordinates the metric $\eta$ has the form
		\[\left(\eta^{\alpha\beta}\right)=\begin{pmatrix}1&0&0\\ 0&1&0\\ 0&0&1\end{pmatrix},\]
		and the metric $g$ has the form
		\[\left(g^{\alpha\beta}\right)=\begin{pmatrix}-v_1^2&0&0\\ 0&-v_2^2&0\\ 0&0&-v_3^2\end{pmatrix}.\]
		The potential of the generalized Frobenius manifold is given by
		\begin{align*}
			F = -\frac 1{12}v_1^4 - \frac 1{12}v_2^4 - \frac 1{12}v_3^4,
		\end{align*}
		the Euler vector field and the unit vector field have the form 
		\[
		E=\frac 12v_1\frac{\p}{\p v_1}+ \frac 12v_2\frac{\p}{\p v_2}+\frac 12v_3\frac{\p}{\p v_3},\quad e = -\frac 1{2v_1}\tangentvector{v_1} - \frac 1{2v_2}\tangentvector{v_2} - \frac 1{2v_3}\tangentvector{v_3}.
		\]
		It is clear that, up to a factor, we have the isomorphism
		\begin{align*}
			\mathcal M(C_3, \omega_3) ~\cong~ \mathcal M(A_1,\omega_1)^{\oplus3}.
		\end{align*}
		Moreover, since the flat coordinates of $g$ has the form
		\begin{align*}
			x_1\big|_{\lambda=0} =& \frac 1{2\pi i}\log (2v_1), \quad x_2\big|_{\lambda=0} = \frac 1{2\pi i}\left(\log (2v_2) +\log (2v_1)\right) ,\\
			x_3\big|_{\lambda=0} =& \frac 1{4\pi i}\left(\log (2v_3)+\log(2v_2)+\log(2v_1)\right),
		\end{align*}
		we know that the monodromy group of $\widehat{\mathcal{M}}_D$ is composed of only those translations on each $x_i$ by integer numbers. Thus we have
		\begin{align}
			\mathrm{Mono}(\widehat{\mathcal M}_D)~\cong~ \{e\} \ltimes \mathbb Z^3 ~<~\mathrm{Stab}_{W}(\omega_3) \ltimes \mathbb Z^3.
		\end{align}
		Here $\mathrm{Stab}_{W}(\omega_3)=\langle\sigma_1, \sigma_2 \rangle \cong S_3$ is not trivial. In fact, in this case we have $\widehat{\mathcal M}_D/S_3 \cong \mathcal M_D$. This example shows that $\mathrm{Mono}(\widehat{\mathcal M}_D)$ is not isomorphic to $\mathrm{Mono}(\mathcal M_D)$ in general.
	\end{ex}
	
	\begin{ex}[$(R,\omega)=(B_3, \omega_1)$]
		\label{Be,omega1}
		Let $R$ be the root system of type $B_3$, and $V=\mathbb{R}^3$. We take the simple roots
		\[\al_1=e_1-e_2, \quad \al_2=e_2-e_3,\quad \al_3=e_3.\]
		Then we have the coroots
		\[\av_1=e_1-e_2,\quad \av_2=e_2-e_3,\quad \av_3=2e_3.\]
		The fundamental weights are given by
		\[\omega_1=\av_1+\av_2+\frac12\av_3,\quad \omega_2=\av_1+2\av_2+\av_3,\quad \omega_3=\frac12\av_1+\av_2+\frac34\av_3.\]
		Let us take $\omega=\omega_1$, and introduce coordinates $x_1, x_2, x_3$ of $V$ by
		\[\mathbf{x}=c\omega_1+x_1 \av_1+x_2 \av_2+x_3\av_3.\]
		Since $\kappa = 1$, we have $\lambda=\mathrm{e}^{-2\pi i c}$.
		The invariant $\lambda$-Fourier polynomials are given by
		\begin{align*}
			y_1=\lambda (\xi_1+\xi_2+\xi_3),\quad
			y_2=\lambda (\xi_1\xi_2+\xi_1\xi_3+\xi_2\xi_3),\quad
			y_3=\lambda^{\frac12}\tilde\xi_1\tilde\xi_2\tilde\xi_3.
		\end{align*}
		Here 
		\begin{align*}
			\xi_1&=\lambda^{-1}\mathrm{e}^{2\pi i x_1}+\lambda \mathrm{e}^{-2\pi i x_1},\quad
			\xi_2=\mathrm{e}^{2\pi i(x_2-x_1)}+\mathrm{e}^{-2\pi i(x_2-x_1)},\\
			\xi_3&=\mathrm{e}^{2\pi i (2x_3-x_2)}+\mathrm{e}^{-2\pi i (2x_3-x_2)},
		\end{align*}
		and 
		\begin{align*}
			\tilde\xi_1&=\lambda^{-\frac12}e^{\pi i x_1}+\lambda^{\frac12} \mathrm{e}^{-\pi i x_1},\quad
			\tilde\xi_2=e^{\pi i(x_2-x_1)}+e^{-\pi i(x_2-x_1)},\\
			\tilde\xi_3&=e^{\pi i (2x_3-x_2)}+e^{-\pi i (2x_3-x_2)}.
		\end{align*}
		The metric on $V$ has the form 
		\[
		\left((dx_i,dx_j)\right)=\begin{pmatrix}2&-1&0\\-1&2&-2\\ 0&-2&4\end{pmatrix}^{-1}=\begin{pmatrix}
			1&1&\frac12\\ 1&2&1\\
			\frac12&1&\frac34
		\end{pmatrix}.
		\]
		In this case, $y_1, y_2, y_3$ are not pencil generators. A set of proper generators should have the form
		\begin{align*}
			z_1=y_1+a_1 \lambda,\quad
			z_2=y_2+a_2 \lambda,\quad
			z_3=y_3,
		\end{align*}
		where $a_1, a_2$ are two constants.
		We have the following three choices for $a_1$ and $a_2$:
		\begin{align}
			&a_1=6, \quad a_2=-12,\label{zh-3}\\
			&a_1=2,\quad a_2=4,\label{zh-4}\\
			&a_1=-2,\quad a_2=4.\label{zh-5}
		\end{align}
		
		We first consider the $a_1=6,\, a_2=-12$ case. In this case we have
		\begin{align*}
			\left(g_\lambda^{ij}\right)=&\,\frac1{4\pi^2} \left((\nd z_i, \nd z_j)\right)=\frac1{4\pi^2} \sum_{k,r=1}^3 \frac{\p z_i}{\p x_k}\frac{\p z_j}{\p x_r}(\nd x_k,\nd x_r)\\
			=&\,\begin{pmatrix}
				-z_1^2
				&-z_1 z_2
				&-\frac{1}{2} z_1 z_3\\
				-z_1 z_2
				&-2 z_2^2-4 z_1 z_2+2 z_1 z_3^2
				&2 z_1 z_3-z_2 z_3\\
				-\frac{1}{2} z_1 z_3
				&2 z_1 z_3-z_2 z_3
				&-\frac{3}{4} z_3^2+4 z_1+z_2
			\end{pmatrix}\\
			&\,+\lambda \begin{pmatrix}
				12 z_1+2 z_2
				&3  z_3^2-16  z_1
				&6  z_3\\
				3  z_3^2-16  z_1
				&-32  z_2-12 z_3^2-64  z_1
				&-24  z_3\\
				6  z_3
				&-24  z_3
				&0
			\end{pmatrix}\\
			=&\,\left(g^{ij}\right)+\lambda\left(\eta^{ij}\right).
		\end{align*}
		The metric $\eta$ has the flat coordinates
		\begin{align*}
			t_1&=-\frac1{1920 z_3^{7/3}}\left(1280 z_1^2+640 z_1 z_2+80 z_2^2-480 z_1 z_3^2+120 z_2 z_3^2-27 z_3^4\right),\\
			t_2&=\frac1{8 z_3} (16 z_1+4 z_2-z_3^2),\\
			t_3&=z_3^{1/3}.
		\end{align*}
		In these coordinates we have
		\[(\eta^{\alpha\beta})=\begin{pmatrix}0&0&1\\ 0&3&0\\ 1&0&0\end{pmatrix}.\]
		We then obtain a generalized Frobenius manifold structure $\mathcal M(B_3,\omega_1)_1$ with potential
		\begin{align*}
			F=&\,\frac{t_1^2 t_2}{3 t_3}-\frac{t_1^2 
				t_3^2}{12}-\frac{2 t_1 t_2^3}{27 t_3^2}-\frac{1}{36} 
			t_1 t_2^2 t_3+\frac{1}{360} t_1 t_2 
			t_3^4-\frac{t_1 t_3^7}{30240}\\
			&+\frac{t_2^5}{135 \
				t_3^3}-\frac{t_2^4}{432}+\frac{t_2^3 
				t_3^3}{2160}-\frac{t_2^2 t_3^6}{8640}+\frac{t_2 
				t_3^9}{172800}-\frac{t_3^{12}}{7603200}.
		\end{align*}
		The Euler vector field and the unit vector field are given by
		\begin{align*}
			E&=\frac56 t_1\frac{\p}{\p t_1}+\frac12 t_2\frac{\p}{\p t_2}+\frac16 t_3\frac{\p}{\p t_3},\\
			e&=-\frac6{D}\left(\left(40 t_1+20 t_2 t_3^2+t_3^5\right)\frac{\p}{\p t_1}+\left(40 t_2+20 t_3^3\right) \frac{\p}{\p t_2}+40 t_3\frac{\p}{\p t_3}\right).
		\end{align*}
		Here 
		$D=40 t_2^2+240 t_1 t_3+40 t_2 t_3^3+t_3^6$.
		We also have $e = -\mathrm{grad}_\eta \log z_1$.
		
		Now we consider the $a_1=2, a_2=4$ case.
		We have
		\begin{align*}
			\left(g_\lambda^{ij}\right)=&\,\frac1{4\pi^2} \left((\nd z_i, \nd z_j)\right)=\frac1{4\pi^2} \sum_{k,r=1}^3 \frac{\p z_i}{\p x_k}\frac{\p z_j}{\p x_r}(\nd x_k,\nd x_r)\\
			=&\,\begin{pmatrix}
				-z_1^2
				&-z_1 z_2
				&-\frac{1}{2} z_1 z_3\\
				-z_1 z_2
				&-2 z_2^2-4 z_1 z_2+2 z_1 z_3^2
				&2 z_1 z_3-z_2 z_3\\
				-\frac{1}{2} z_1 z_3
				&2 z_1 z_3-z_2 z_3
				&-\frac{3}{4} z_3^2+4 z_1+z_2
			\end{pmatrix}\\
			&\,+\lambda \begin{pmatrix}
				4 z_1+2 z_2
				&3  z_3^2-4  z_2
				&4  z_3\\
				3  z_3^2-4  z_2
				&16  z_2-4 z_3^2
				&0\\
				4  z_3
				&0
				&0
			\end{pmatrix}\\
			=&\,\left(g^{ij}\right)+\lambda\left(\eta^{ij}\right).
		\end{align*}
		The metric $\eta$ has the flat coordinates
		\[
		t_1=\frac{48 z_1+12 z_2-5 z_3^2}{48\sqrt{z_3}},\quad
		t_2=\frac14\sqrt{4z_2-z_3^2},\quad
		t_3=\frac12 \sqrt{z_3}\,.
		\]
		In these coordinates we have
		\[\left(\eta^{\alpha\beta}\right)=\begin{pmatrix}0&0&1\\ 0&1&0\\ 1&0&0\end{pmatrix}.\]
		We then obtain a generalized Frobenius manifold structure $\mathcal M(B_3,\omega_1)_2$ with potential
		\begin{align*}
			F=&\,-\frac{t_2^4}{48}+\frac{t_1^3}{24 t_3}-\frac14 t_1 t_2^2 t_3-\frac{t_1^2 t_3^2}{12}+\frac{t_2^2t_3^4}{24}+\frac{t_1t_3^5}{180}-\frac{t_3^8}{2268}.
		\end{align*}
		The Euler vector field and the unit vector field are given by
		\begin{align*}
			E&=\frac34 t_1\frac{\p}{\p t_1}+\frac12 t_2\frac{\p}{\p t_2}+\frac14 t_3\frac{\p}{\p t_3},\\
			e&=\frac1{3 t_2^2-6 t_1 t_3-2 t_3^4}\left(\left(6 t_1+8t_3^3\right)\frac{\p}{\p t_1}-6 t_2\frac{\p}{\p t_2}+6 t_3\frac{\p}{\p t_3}\right).
		\end{align*}
		We also have $e = -\mathrm{grad}_\eta \log z_1$.
		
		Finally, let us consider the $a_1=-2, a_2=4$ case. We have
		\begin{align*}
			\left(g_\lambda^{ij}\right)=&\,\frac1{4\pi^2} \left((\nd z_i, \nd z_j)\right)=\frac1{4\pi^2} \sum_{k,r=1}^3 \frac{\p z_i}{\p x_k}\frac{\p z_j}{\p x_r}(\nd x_k,\nd x_r)\\
			=&\,\begin{pmatrix}
				-z_1^2
				&-z_1 z_2
				&-\frac{1}{2} z_1 z_3\\
				-z_1 z_2
				&-2 z_2^2-4 z_1 z_2+2 z_1 z_3^2
				&2 z_1 z_3-z_2 z_3\\
				-\frac{1}{2} z_1 z_3
				&2 z_1 z_3-z_2 z_3
				&-\frac{3}{4} z_3^2+4 z_1+z_2
			\end{pmatrix}\\
			&\,+\lambda \begin{pmatrix}
				-4 z_1+2 z_2
				&3  z_3^2-8  z_2
				&2  z_3\\
				3  z_3^2-8  z_2
				&4 z_3^2
				&8 z_3\\
				2z_3
				&8 z_3
				&16
			\end{pmatrix}\\
			=&\,\left(g^{ij}\right)+\lambda\left(\eta^{ij}\right).
		\end{align*}
		The metric $\eta$ has the flat coordinates
		\[
		t_1=-\frac{48 z_1+8 z_2-5 z_3^2}{48\left(4z_2-z_3^2\right)^{1/4}},\quad
		t_2=\frac14 z_3,\quad
		t_3=\frac12 \left(4z_2-z_3^2\right)^{1/4}.
		\]
		In these coordinates we have
		\[\left(\eta^{\alpha\beta}\right)=\begin{pmatrix}0&0&1\\ 0&1&0\\ 1&0&0\end{pmatrix},\]
		and we obtain the same potential as the one given in the previous case with $a_1=2, a_2=4$, hence the third generalized Frobenius manifold structure $\mathcal M(B_3, \omega_1)_3$ is isomorphic to the second one.
	\end{ex}
	
	\begin{ex}[$(R,\omega)=(G_2, \omega_2)$]
		\label{G2, omega2}
		Let $R$ be the root system of type $G_2$, and $V$ be the subspace of $\mathbb{R}^3$ spanned by the simple roots
		\[\al_1=e_1-e_2, \quad \al_2=-2e_1+e_2+e_3.\]
		We have the coroots
		\[\av_1=\al_1,\quad \av_2=\frac13\al_2.\]
		The fundamental weights are given by
		\[
		\omega_1=2\av_1+3\av_2,\quad \omega_2=3\av_1+6\av_2.
		\]
		Let us take $\omega=\omega_2$ and introduce coordinates $x_1, x_2$ of $V$ by
		\[
		\mathbf{x}=c\,\omega_2+x_1 \av_1+x_2 \av_2=(x_1+3c) \av_1+(x_2+6c) \av_2,
		\]
		Since $\kappa = 3$, we have $\lambda=\mathrm{e}^{-6\pi i c}$.
		The invariant $\lambda$-Fourier polynomials are given by
		\begin{align*}
			y_1=& \,\mathrm{e}^{2\pi i x_1}+\lm^2\mathrm{e}^{-2\pi i x_1}+\lm \mathrm{e}^{2\pi i (2x_1-x_2)}+\lm \mathrm{e}^{-2\pi i (2x_1-x_2)}+\lm^2 \mathrm{e}^{2\pi i (x_1-x_2)}+\mathrm{e}^{-2\pi i (x_1-x_2)},\\
			y_2=&\,\mathrm{e}^{2\pi i x_2}+\lm^4 \mathrm{e}^{-2\pi i x_2}+\lm^3 \mathrm{e}^{2\pi i (3x_1-2x_2)}+\lm \mathrm{e}^{-2\pi i (3x_1-2x_2)}+\lm^3 \mathrm{e}^{2\pi i (x_2-3x_1)}+\lm \mathrm{e}^{-2\pi i (x_2-3x_1)}.
		\end{align*}
		The metric on $V$ has the form 
		$$ 
		\left((\nd x_i,\nd x_j)\right)= \begin{pmatrix}2&-1\\-1&\frac23\end{pmatrix}^{-1}= \begin{pmatrix}
			2&3\\[3pt]
			3&6
		\end{pmatrix}.
		$$
		In the $y_1, y_2$ coordinates the metric
		does not depend at most linearly on $\lambda$, hence $y_1$ and $y_2$ are not pencil generators. Proper generators $z_1$ and $z_2$ can be introduced by
		\[z_1=y_1-6\lambda,\quad
		z_2=y_2-3\lambda y_1+12\lambda^2.\]
		Then in the $z_1, z_2$ coordinates we have
		\[\left(g^{\al\beta}_\lambda\right)=\left(g^{\al\beta}\right)+\lambda\left(\eta^{\al\beta}\right)
		=\begin{pmatrix}
			-2 z_1^2+2 z_2 &-3 z_1 z_2 \\ -3 z_1 z_2 &-6 z_2^2
		\end{pmatrix}
		+\lambda\begin{pmatrix}
			-12 z_1 &3 z_1^2-36 z_2 \\ 3 z_1^2-36 z_2 &6 z_1^3-36 z_1 z_2
		\end{pmatrix}.\]
		The metric $\eta$ has flat coordinates
		\[t_1=\frac{z_1}{2\left(4 z_2-z_1^2\right)^{\frac16}},\quad
		t_2=-\frac13\left(4 z_2-z_1^2\right)^{\frac16},\]
		in these coordinates the metric $\eta$ has the form
		\[\left(\eta^{\al\beta}\right)=\begin{pmatrix}
			0&1\\ 1&0
		\end{pmatrix}.\]
		We thus obtain a generalized Frobenius manifold with potential
		\[F=\frac{t_1^4}{2916 t_2^2}-\frac18 t_1^2 t_2^2+\frac{81 t_2^6}{320}.\]
		The Euler vector field and the unit vector field are given by
		\begin{align*}
			E&=\frac23 t_1\frac{\p}{\p t_1}+\frac13 t_2\frac{\p}{\p t_1},\\
			e&=-\frac2{3 t_2\left(4 t_1^2+81 t_2^4\right)}\left((4 t_1^2+243 t_2^4)\frac{\p}{\p t_1}+4 t_1 t_2\frac{\p}{\p t_2}\right).
		\end{align*}
		We also have $e = -\frac 13\mathrm{grad}_\eta \log z_2$.
	\end{ex}
	
	\begin{ex}[$(R,\omega)=(A_3,\omega_1+\omega_3)$]
		\label{A3,omega1+omega3}
		Let $R$ be the root system of type $A_3$, $V$ be the subspace of $\mathbb{R}^4$ spanned by the simple roots
		\[\al_1=e_1-e_2,\quad \al_2=e_2-e_3,\quad \al_3=e_3-e_4.\] 
		We have $\av_i=\al_i, ~i=1,2,3$, and the fundamental weights are given by
		\[\omega_1=\frac34\av_1+\frac12\av_2+\frac14\av_3,\quad 
		\omega_2=\frac12\av_1+\av_2+\frac12\av_3,\quad
		\omega_3=\frac14\av_1+\frac12\av_2+\frac34\av_3.\]
		Let us take $\omega=\omega_1+\omega_3$, and introduce coordinates $x_1, x_2, x_3$ of $V$ by
		\[\mathbf{x}=c (\omega_1+\omega_3)+x_1 \av_1+x_2 \av_2+x_3\av_3,\quad \mathbf{x}\in V.\]
		Since $\kappa = 1$, we have $\lambda = \mathrm{e}^{-2\pi i c}$.
		We have the invariant $\lambda$-Fourier polynomials
		\begin{align*}
			y_1=&\,\mathrm{e}^{2\pi i x_1}+\lambda \mathrm{e}^{-2\pi i (x_1-x_2)}+\lambda^2 \mathrm{e}^{-2\pi i x_3}+\lambda \mathrm{e}^{-2\pi i (x_2-x_3)},\\
			y_2=&\,\mathrm{e}^{2\pi i x_2}+\lambda^2 \mathrm{e}^{-2\pi i x_2}+\lambda \mathrm{e}^{2\pi i (x_1-x_3)}+\lambda^2 \mathrm{e}^{-2\pi i (x_1-x_2+x_3)}\\
			&\,+\lambda \mathrm{e}^{-2\pi i (x_1-x_3)}+\mathrm{e}^{2\pi i (x_1-x_2+x_3)},\\
			y_3=&\,\mathrm{e}^{2\pi i x_3}+\lambda^2 \mathrm{e}^{-2\pi i x_1}+\lambda \mathrm{e}^{2\pi i (x_1-x_2)}+\lambda \mathrm{e}^{2\pi i (x_2-x_3)}.
		\end{align*}
		The metric on $V$ is given by
		\[\left((\av_i, \av_j)\right)=\begin{pmatrix}2 &-1&0\\ -1 &2&-1\\0&-1&2\end{pmatrix},\quad
		\left(\nd x_i,\nd x_j\right)={\left((\av_i, \av_j)\right)}^{-1}
		=\frac 14\begin{pmatrix}3 &2&1\\ 2 &4&2\\1&2&3\end{pmatrix}.\]
		In the coordinates $y_1, y_2, y_3$ we have
		\begin{align*}
			\left(g_\lambda^{ij}\right)&=\frac1{4\pi^2}\left((\nd y_i, \nd y_j)\right)=\frac1{4\pi^2}\sum_{k,r=1}^3 \frac{\p y_i}{\p x_k}\frac{\p y_j}{\p x_r}(\nd x_k,\nd x_r)\\
			&=\begin{pmatrix}-\frac34 y_1^2+2\lambda y_2 & -\frac12 y_1 y_2+3 \lambda y_3&-\frac14 y_1 y_3+4\lambda^2 \\ -\frac12 y_1 y_2+3\lambda y_3&-y_2^2+2 y_1 y_3+4\lambda^2&-\frac12 y_2y_3+3\lambda y_1\\-\frac14 y_1 y_3+4\lambda^2&-\frac12 y_2y_3+3\lambda y_1&-\frac34 y_3^2+2 \lambda y_2
			\end{pmatrix}.
		\end{align*}
		We introduce the pencil generators
		\[z_1=y_1+4\lambda,\quad
		z_2=y_2-6 \lambda,\quad
		z_3=y_3+4 \lambda.
		\]
		In the new coordinates, we obtain a flat pencil of metrics $\left(g^{ij}\right)+\lambda \left(\eta^{ij}\right)$ with
		\[
		\left(\eta^{ij}\right)
		=\begin{pmatrix} 6 z_1+2 z_2 & -3 z_1+2 z_2+3 z_3 & z_1+z_3\\
			-3 z_1+2 z_2+3 z_3&-8 z_1-12 z_2-8 z_3& 3 z_1+2 z_2-3 z_3\\
			z_1+z_3& 3 z_1+2 z_2-3 z_3 &2 z_2+6 z_3
		\end{pmatrix}.
		\]
		The metric $\eta$ has the flat coordinates $t_1, t_2, t_3$ which are given by
		\begin{align*}
			z_1&=\frac16 t_1^4+ t_1^2 t_2+\frac12 t_2^2+ t_1 t_3,\\
			z_2&=\frac23 t_1^4-t_2^2-2 t_1 t_3,\\
			z_3&=\frac16 t_1^4- t_1^2 t_2+\frac12 t_2^2+ t_1 t_3;
		\end{align*}
		in these coordinates the metric $\eta$ has the form
		\begin{align*}
			\left(\eta^{\alpha\beta}\right)&=\begin{pmatrix}0&0 &1\\ 0&1&0\\ 1&0&0\end{pmatrix}.
		\end{align*}
		Thus we obtain a generalized Frobenius manifold $\mathcal M(A_3, \omega_1 + \omega_3)$ with potential
		\begin{align*}
			F=&-\frac{t_1^8}{18144}+\frac{t_1^5 
				t_3}{720}-\frac{t_2^6}{96 t_1^4}-\frac{t_1^4 
				t_2^2}{288}+\frac{t_2^4 t_3}{16 
				t_1^3}-\frac{t_2^2 t_3^2}{8 t_1^2}\\
			&-\frac{t_1^2 
				t_3^2}{24}-\frac{1}{24} t_1 t_2^2 
			t_3+\frac{t_3^3}{24 t_1}-\frac{t_2^4}{96},
		\end{align*}
		the Euler vector field 
		\[E=\frac14 t_1\frac{\p}{\p t_1}+\frac12 t_2\frac{\p}{\p t_2}+\frac34 t_3\frac{\p}{\p t_3},\]
		and the unit vector field
		\begin{align*}
			e=&\,\frac1{D}\left(\left(-12 t_1^5-72t_1^2 t_3 -36t_1 t_2^2 \right)\frac{\p}{\p t_1}
			+\left(60 t_2 t_1^4-72 t_1 t_2 t_3-36 t_2^3\right)\frac{\p}{\p t_2}\right.\\
			&\,
			\left.+\left(-8 t_1^7-60 t_1^4 t_3+120 t_1^3 t_2^2-72 t_1 t_3^2-36 t_2^2 t_3\right)\frac{\p}{\p t_3}\right),
		\end{align*}
		where
		\[D=\left(t_1^4-6 t_1^2 t_2+6 t_1 t_3+3 t_2^2\right) \left(t_1^4+6
		t_1^2 t_2 +6 t_1 t_3+3 t_2^2\right).\]
		We also have $e = -\mathrm{grad}_\eta(\log z_1 + \log z_3)$.
	\end{ex}
	\begin{ex}[$(R,\omega) = (D_4,\omega_2)$]
		\label{(D4,omega2)}
		Let $R$ be the root system of type $D_4$, and $V=\mathbb{R}^4$. We take the simple roots
		\[
		\alpha_1 = e_1-e_2,\quad \alpha_2 = e_2-e_3,\quad \alpha_3=e_3-e_4,\quad \alpha_4=e_3+e_4.
		\]
		Then we have the coroots $\alpha_i^\vee = \alpha_i,~i=1,2,3$, 
		and fundamental weights
		\begin{align*}
			\omega_1=&~\alpha_1^\vee+\alpha_2^\vee+\frac{1}{2} \alpha_3^\vee+\frac{1}{2}\alpha_4^\vee,\quad \omega_2=\alpha_1^\vee+2 \alpha_2^\vee+\alpha_3^\vee+\alpha_4^\vee,\\
			\omega_3=&~\frac{1}{2}\alpha_1^\vee+ \alpha_2^\vee+ \alpha_3^\vee+\frac{1}{2}\alpha_4^\vee,\quad\omega_4=\frac{1}{2} \alpha_1^\vee+ \alpha_2^\vee+\frac{1}{2}\alpha_3^\vee+ \alpha_4^\vee.
		\end{align*}
		Let us take $\omega=\omega_2$ and introduce coordinates $x_1, x_2$ of $V$ by
		\[
		\mathbf{x}=c\,\omega_2+x_1 \av_1+x_2 \av_2 + x_3 \av_3+x_4 \av_4. 
		\]
		Since $\kappa = 1$, we have $\lambda = \mathrm{e}^{-2\pi i c}$.
		We have the invariant $\lambda$-Fourier polynomials
		\begin{align*}
			y_1 =&~\lm(\xi_1+\xi_2+\xi_3+\xi_4),\\
			y_2 =&~\lm^2(\xi_1 \xi_2+\xi_1 \xi_3+\xi_1 \xi_4+\xi_2 \xi_3+\xi_2 \xi_4+\xi_3 \xi_4),\\
			y_3 =&~\frac \lm2 (\psi_1\psi_2\psi_3\psi_4 - \phi_1\phi_2\phi_3\phi_4),\\
			y_4 =&~\frac \lm2 (\psi_1\psi_2\psi_3\psi_4 + \phi_1\phi_2\phi_3\phi_4).
		\end{align*}
		Here we denote $v_1 = x_1+c$, $v_2 = x_2-x_1+c$, $v_3 = x_4+x_3-x_2$, $v_4 = x_4-x_3$, and
		\begin{align*}
			\xi_k = \mathrm{e}^{2\pi iv_k}+ \mathrm{e}^{-2\pi i v_k},\quad \psi_k = \mathrm{e}^{\pi iv_k}+ \mathrm{e}^{-\pi i v_k},\quad \phi_k = \mathrm{e}^{\pi iv_k}- \mathrm{e}^{-\pi i v_k},\quad k=1,2,3,4.
		\end{align*}
		The metric on $V$ is given by
		\begin{align*}
			\left((\alpha_i^\vee, \alpha_j^\vee)\right) = \left(
			\begin{array}{cccc}
				2 & -1 & 0 & 0 \\
				-1 & 2 & -1 & -1 \\
				0 & -1 & 2 & 0 \\
				0 & -1 & 0 & 2 \\
			\end{array}
			\right),\quad \left(\dif x_i, \dif x_j\right) = \left(
			\begin{array}{cccc}
				1 & 1 & \frac{1}{2} & \frac{1}{2} \\
				1 & 2 & 1 & 1 \\
				\frac{1}{2} & 1 & 1 & \frac{1}{2} \\
				\frac{1}{2} & 1 & \frac{1}{2} & 1 \\
			\end{array}
			\right).
		\end{align*}
		In the coordinates $y_1,\dots,y_4$ we have
		\[
		\left(g_\lambda^{ij}\right) = \frac{1}{4\pi^2}\left(
		\begin{array}{cccc}
			16 \lm^2-y_1^2+2 y_2 & 3 \lm y_3 y_4-y_1 y_2 & 4 \lm y_4-\frac{1}{2} y_1 y_3 & 4 \lm y_3-\frac{1}{2} y_1 y_4 \\
			3 \lm y_3 y_4-y_1 y_2 & g_\lm^{22} & 3 \lm y_1 y_4-y_2 y_3 & 3 \lm y_1 y_3-y_2 y_4 \\
			4 \lm y_4-\frac{1}{2} y_1 y_3 & 3 \lm y_1 y_4-y_2 y_3 & 16 \lm^2-y_3^2+2 y_2 & 4 \lm y_1-\frac{1}{2} y_3 y_4 \\
			4 \lm y_3-\frac{1}{2} y_1 y_4 & 3 \lm y_1 y_3-y_2 y_4 & 4 \lm y_1-\frac{1}{2} y_3 y_4 & 16 \lm^2-y_4^2+2 y_2 \\
		\end{array}
		\right),
		\]
		where $g_\lm^{22} = -64 \lm^4+4 \lm^2 y_1^2+4 \lm^2 y_3^2+4 \lm^2 y_4^2-24 \lm^2 y_2+2 \lm y_1 y_3 y_4-2 y_2^2$. We introduce the pencil generators
		\[
		z_1 = y_1+ 8\lm,\quad z_2 = y_2+24\lm^2 + \lm(2y_1+2y_3-2y_4),\quad z_3 = y_3+8\lm,\quad z_4 = y_4-8\lm.
		\]
		In the new coordinates, we obtain a flat pencil of metrics $(g^{ij}) + \lm(\eta^{ij})$. The flat coordinates $t_\alpha$ of $\eta^{ij}$ are given by polynomials
		\begin{align*}
			z_1 &= \frac{8}{3}t_1^4 + 2 t_1 t_4 + \frac{8}{3}t_2^4 + 2 t_2 t_3, \\
			z_2 &= \frac{4}{9}t_1^8 - \frac{4}{3} t_1^5 t_4 + \frac{56}{9} t_1^4 t_2^4 + \frac{20}{3} t_1^4 t_2 t_3 + t_1^2 t_4^2 + \frac{20}{3} t_1 t_2^4 t_4 + 2 t_1 t_2 t_3 t_4 + \frac{4}{9}t_2^8 + t_2^2 t_3^2 - \frac{4}{3} t_2^5 t_3, \\
			z_3 &= -\frac{4}{3}t_1^4 + 8 t_1^2 t_2^2 + 2 t_1 t_4 + 2 t_2 t_3 - \frac{4}{3}t_2^4, \\
			z_4 &= \frac{4}{3}t_1^4 + 8 t_1^2 t_2^2 - 2 t_1 t_4 + \frac{4}{3}t_2^4 - 2 t_2 t_3.
		\end{align*}
		In these coordinates the metric $\eta$ has the form
		\[
		\left(\eta^{\alpha\beta}\right)=\begin{pmatrix}0&0&0 &1\\ 0&0&1&0\\ 0&1&0&0\\1&0&0&0\end{pmatrix}.
		\]
		Thus we obtain a generalized Frobenius manifold $\mathcal M(A_3, \omega_1+\omega_3)$ with potential
		\begin{align*}
			F &= \frac{1}{t_1^4 - t_2^4} \Bigg(
			-\frac{t_1^{12}}{1134} -\frac{t_1^9 t_4}{180} +\frac{t_1^8 t_2 t_3}{36} -\frac{13 t_1^8 t_2^4}{1134} -\frac{t_1^6 t_4^2}{24}
			+\frac{t_1^5 t_2^4 t_4}{30} -\frac{t_1^5 t_2 t_3 t_4}{12} +\frac{t_1^4 t_3^3}{96 t_2} -\frac{t_1^4 t_2^2 t_3^2}{24} \\
			&\quad -\frac{t_1^4 t_2^5 t_3}{30} +\frac{13 t_1^4 t_2^8}{1134} -\frac{t_1^3 t_4^3}{96} +\frac{t_1^2 t_2 t_3 t_4^2}{16} +\frac{t_1^2 t_2^4 t_4^2}{24} +\frac{t_1 t_2^5 t_3 t_4}{12}
			-\frac{t_1 t_2^2 t_3^2 t_4}{16} -\frac{t_1 t_2^8 t_4}{36} \\
			&\quad +\frac{t_2^3 t_3^3}{96} +\frac{t_2^6 t_3^2}{24} +\frac{t_2^9 t_3}{180} +\frac{t_2^{12}}{1134} -\frac{t_2^4 t_4^3}{96 t_1}\Bigg).
		\end{align*}
		the Euler vector field
		\[
		E = \frac 14t_1\tangentvector{t_1} + \frac 14t_2\tangentvector{t_2} + \frac 34t_3\tangentvector{t_3} + \frac 34t_4\tangentvector{t_4},
		\]
		and the unit vector field
		\begin{align*}
			e = \frac 1{D}&\left(6t_1 \left(2 t_1^4-3 t_4 t_1-10 t_2^4-3 t_2 t_3\right)\tangentvector{t_1}+6 t_2 \left(-10 t_1^4-3 t_4 t_1+2 t_2^4-3 t_2 t_3\right)\tangentvector{t_2}\right.\\
			&-2\left(16 t_2^7-30 t_3 t_2^4+112 t_1^4 t_2^3+120 t_1 t_4 t_2^3+9 t_3^2 t_2+30 t_1^4 t_3+9 t_1 t_3 t_4\right)\tangentvector{t_3}\\
			&-2\left.\left(16 t_1^7-30 t_4 t_1^4+112 t_2^4 t_1^3+120 t_2 t_3 t_1^3+9 t_4^2 t_1+30 t_2^4 t_4+9 t_2 t_3 t_4\right)\tangentvector{t_4}\right),
		\end{align*}
		where
		\[
		D=4 t_1^8 + 56 t_1^4 t_2^4 + 4 t_2^8 + 60 t_1^4 t_2 t_3 - 
		12 t_2^5 t_3 + 9 t_2^2 t_3^2 - 12 t_1^5 t_4 + 
		60 t_1 t_2^4 t_4 + 18 t_1 t_2 t_3 t_4 + 9 t_1^2 t_4^2.
		\]
		We also have $e = -\mathrm{grad}_\eta\log z_2$.
	\end{ex}
	
	\section{Conclusions}
	\label{conclusion}
	Given a root system $R$, a weight $\omega$ and a set $\{z^1,\dots,z^\ell\}$ of pencil generators, we present an approach to construct generalized Frobenius manifold structure $\mathcal{M}(R,\omega)$ on the orbit spaces of the associated affine Weyl groups $W_a(R)$, and we show that the monodromy group \( \mathrm{Mono}(\mathcal{M}(R,\omega)) \) of the generalized Frobenius manifold is isomorphic to the subgroup $\mathrm{Stab}_{W}(\omega) \ltimes \mathbb Z^\ell$ of $W_a(R)$.
	
	In the subsequent paper \cite{JLTZ-2}, we will explicitly construct the pencil generators for affine Weyl groups of types $A_\ell $, $B_\ell $, $C_\ell $ and $D_\ell $ together with appropriately chosen weights $\omega$, and compute the corresponding generalized Frobenius manifold structures in detail. For the case $(A_\ell, \omega_\ell)$, the generalized Frobenius manifold structure that we will present in \cite{JLTZ-2} coincides with the one constructed in \cite{Cao}, where this generalized Frobenius manifold structure is constructed in terms of a super-potential given by a Fourier polynomial.  It is shown in \cite{Cao} that the bihamiltonian structure that is given by the associated flat pencil of metrics coincides with the dispersionless limit of the bihamiltonian structure of the q-deformed Gelfand-Dickey hierarchy \cite{Fr, Fr-R}. 
	
	Finally, let us outline some potential directions for further research.
	\begin{enumerate}[i)]
		\item \textbf{General construction of pencil generators}:
		Can pencil generators be constructed for every affine Weyl group 
		$W_a(R)$ and an arbitrarily chosen weight $\omega$?
		
		\item \textbf{Relation to other Frobenius manifold constructions}:  
		Are these generalized Frobenius manifolds related to the Frobenius manifolds constructed in \cite{2d-tft, DZ1998} on the orbit spaces of the Coxeter groups and the extended affine Weyl groups?
		
		\item \textbf{Almost duality}:  
		What is the Dubrovin almost-duality structure \cite{Du-04} for \( \mathcal{M}(R,\omega) \)? 
		\item \textbf{Legendre transformations}:  
		What are the properties of Legendre transformations \cite{Legendre, Legendre2} for these generalized Frobenius manifolds?
		
		\item \textbf{Associated integrable systems}:  
		Which integrable hierarchies already known in the theory of integrable systems are given by the topological deformations of the Principal Hiearchies \cite{ZLQW, gfm} of these generalized Frobenius manifolds?
	\end{enumerate}	
	
	\appendix
	\section{Proof of Lemma \ref{zh-06-02b}}
	\label{degge0}
	We are to prove that
	\begin{align}
		\mathrm{Re}(d_\alpha) > 0,\quad \alpha = 1,\dots,\ell.
	\end{align}
	To this end, we first show a general version of Theorem \ref{ext coord}, the conclusion of which is not based on the assertion of Lemma \ref{zh-06-02b}. Since we do not assume a priori that all degrees $d_\alpha$ have positive real part, we may come across resonance cases when solving the differential equation of quasi-homogeneous periods, \textit{i.e.}, it could happen that
	\[
	d_\beta - d_\alpha = j\kappa,\quad j \in \mathbb Z
	\]
	for some $\alpha$ and $\beta$, so in such a case the quasi-homogeneous periods
	may contain logarithmic terms. In what follows we follow the notations of Sect.\,\ref{Polynomial Property and Monodromy Groups}.
	
	\begin{lem}
		\label{gen ext coord}
		For each quasi-homogeneous flat coordinate $t^\alpha = t^\alpha(z)$ of the metric $\eta$ with degree $d_\alpha$, there is a quasi-homogeneous period $u \in \mathcal S_\Lambda$ such that
		\begin{align}
			u = v^\alpha(z,\rho) = \rho^{d_\alpha/\kappa}h^\alpha(z,\rho),\label{zh-07-08-1}
		\end{align}
		where $h^{\alpha}(z, \rho)$ has the form
		\begin{align}
			\label{rho log}
			h^\alpha(z,\rho) = t^\alpha(z) + \sum_{j=1}^\infty \rho^j\Big(h^\alpha_{j,0}(z) + \log \rho~ h^{\alpha}_{j,1}(z) \Big).
		\end{align}
		\begin{proof}
			To find a solution $u$ of the form \eqref{zh-07-08-1} for the equations \eqref{concrete extended 1} and \eqref{extended flat 2}, we need to solve the following system of equations:
			\begin{align}
				\pdf{\vec{\xi}}{z^i} &= \Gamma_i(\rho)\vec{\xi},\quad i=1,\cdots,\ell,\label{to solve 1}\\
				\kappa\rho \pdf{\vec{\xi}}{\rho} &= (A(\rho) - d_\alpha)\vec{\xi},\label{to solve 2}
			\end{align}
			where		
			\[
			\vec{\xi} = \left(\pdf{h^\alpha}{z^1},\dots,\pdf{h^\alpha}{z^\ell}\right)^T,\quad (\Gamma_i(\rho))_j^k = \Gamma_{ij}^k(\rho).\] 
			We expand $A(\rho)$ into a power series
			\[
			A(\rho) = A + \rho A_1 + \rho^2 A_2 + \dots
			\]
			and expand $\vec{\xi}$ into the form given by \eqref{rho log} as
			\[
			\vec{\xi} = \vec{\xi}_{(0)} + \sum_{j=1}^\infty \rho^j\Big(\vec{\xi}_{(j,0)} + \log \rho~ \vec{\xi}_{(j,1)} \Big).
			\]
			Then the equation \eqref{to solve 2} is equivalent to
			\begin{align}
				(A-d_\alpha)\vec{\xi}_{(0)} &= 0,\label{solve-new1}\\
				(A-d_\alpha-j\kappa)\vec{\xi}_{(j,1)} &= -\sum_{r=1}^{j-1}A_r\vec{\xi}_{(j-r,1)},\quad j=1,2,\dots,\label{solve-new2}\\
				(A-d_\alpha - j\kappa)\vec{\xi}_{(j,0)} &= \kappa\vec{\xi}_{(j,1)} - A_j\vec{\xi}_{(0)} - \sum_{r=1}^{j-1}A_r\vec{\xi}_{(j-r,0)},\quad j=1,2,\dots,\label{solve-new3}
			\end{align}
			here \eqref{solve-new1} is automatically satisfied. Note that $A$ is diagonalizable due to the assumption that $\nabla E$ is diagonalizable, hence for each $\zeta \in \mathbb C$, $\ker (A-\zeta) \oplus \mathrm{Im}(A-\zeta) = \mathbb C^\ell$. Therefore, for each $j=1,2,\dots$, whether $d_\alpha + j\kappa$ is an eigenvalue of $A$ or not, we can solve $\vec{\xi}_{(j,0)}$ and $\vec{\xi}_{(j,1)}$ uniquely from the equations \eqref{solve-new2} and \eqref{solve-new3}. Then we obtain a solution $\vec{\xi}$ of \eqref{to solve 2} by induction on $j$.
			Since $\vec{\xi}_{(0)}$ solves \eqref{to solve 1}, and equations \eqref{to solve 1} are compatible with \eqref{to solve 2}, we know that $\vec{\xi}$ also solve \eqref{to solve 1}. The lemma is proved.
		\end{proof}
	\end{lem}
	
	It follows from the above Lemma that for each index $\alpha$, there exist constants $C^{\alpha}_\beta$ and $S^\alpha$, such that
	\begin{align}
		\label{lin comb}
		v^\alpha(z,\rho) = C^\alpha_\beta \check x^\beta(z,\rho) + S^\alpha.
	\end{align}
	Now we focus on the behaviour of the right hand side of \eqref{lin comb} when $\rho \to 0$. Note that the Laurant polynomials $Z_1,\cdots,Z_\ell$ given in \eqref{zr = lambda Zr} induce the canonical projection
	\[
	Z = (Z_1,\cdots,Z_\ell)\colon \mathbb C^\ell \to \mathbb C^\ell/W_a(R),
	\]
	where $\check{x}^1,\dots,\check{x}^\ell$ are coordinates on $\mathbb C^\ell$ and the affine Weyl group $W_a(R)$ acts canonically on $\mathbb C^\ell=\mathbb R^\ell\otimes\mathbb C$. Denote the inverse map of $Z$ by $\check{X}$, then we have
	\[
	\check{x}^\beta(z,\rho) = \check{X}^\beta(\rho^{\theta_1/\kappa}z_1,\cdots,\rho^{\theta_\ell/\kappa}z_\ell).
	\]
	Note that $Z$ is a branched covering map, hence there exists a neighbourhood $U$ of $(Z_1,\dots,Z_\ell) = (0,\dots,0)$ such that its preimage is a disjoint union of some open subsets of $\mathbb C^\ell$, on each one of which merely one point is mapped to $(Z_1,\dots,Z_\ell) = (0,\dots,0)$. 
	Since $(\rho^{\theta_1/\kappa}z_1,\dots,\rho^{\theta_\ell/\kappa}z_\ell)$ lies in $U$ for sufficiently small $\rho$, we know that on each monodromy branch, $\check{x}(z,\rho) = (\check{x}^1(z,\rho),\dots,\check{x}^\ell(z,\rho))$ converges to a constant point when $\rho \to 0$, regardless of the value of $z = (z_1,\dots,z_\ell)$. In particular, this constant point is one of the common roots of $Z_r$ for $r = 1,\cdots,\ell$.
	Therefore, the right hand side of \eqref{lin comb} converges to a constant $V^\alpha$ when $\rho \to 0$ on each monodromy branch. It follows from Lemma \ref{gen ext coord} and \eqref{lin comb} that
	\[
	\rho^{d_\alpha/\kappa}\left(t^\alpha(z) + \sum_{j=1}^\infty \rho^j\Big(h^\alpha_{j,0}(z) + \log \rho~ h^{\alpha}_{j,1}(z) \Big)\right) \to V^\alpha, \quad \textrm{when}\ \rho \to 0.
	\]
	Since $t^\alpha(z)$ is not a constant, we have $V^\alpha = 0$ and $\rho^{d_\alpha/\kappa} \to 0$ when $\rho \to 0$ on each monodromy branch. This implies that $\mathrm{Re}(d_\alpha/\kappa) >0$, hence $\mathrm{Re}(d_\alpha)>0$. The theorem is proved.
	
	\section*{Acknowledgement}
		This work is supported by
		NSFC No.\,12571266.

\end{document}